\documentclass[a4paper, 11pt]{amsart}

\title{Riemann--Roch isometries in the non-compact orbifold setting}
\author[Freixas i Montplet]{Gerard Freixas i Montplet}
\author[von Pippich]{Anna-Maria von Pippich}
\date{}
\address{CNRS -- Institut de Math\'ematiques de Jussieu - Paris Rive Gauche, 4 Place Jussieu,
   75005 Paris, France}
\email{gerard.freixas@imj-prg.fr}
\address{Fachbereich Mathematik -- Technische Universit\"at Darmstadt, Schlo{\ss}gartenstr. 7, D-64289 Darmstadt, Germany}
\email{pippich@mathematik.tu-darmstadt.de}

\usepackage{amsmath, mathrsfs, amssymb, amscd, amsthm}
\usepackage[]{fontenc}
\usepackage{xy}
\usepackage{enumerate}
\xyoption{all}
\usepackage{hyperref,xcolor,graphicx}

\numberwithin{equation}{section}

\theoremstyle{plain}
\newtheorem{theorem}{Theorem}[section]
\newtheorem{proposition}[theorem]{Proposition}
\newtheorem{lemma}[theorem]{Lemma}
\newtheorem{corollary}[theorem]{Corollary}

\theoremstyle{definition}
\newtheorem{definition}[theorem]{Definition}

\newtheorem{notation}[theorem]{Notation}

\theoremstyle{remark}
\newtheorem{remark}[theorem]{Remark}
\newtheorem{example}[theorem]{Example}
\DeclareMathOperator{\Spec}{Spec}
\DeclareMathOperator{\hyp}{hyp}
\DeclareMathOperator{\ps}{ps}

\DeclareMathOperator{\PSL}{PSL}
\DeclareMathOperator{\SL}{SL}

\DeclareMathOperator{\Real}{Re}
\DeclareMathOperator{\Imag}{Im}
\DeclareMathOperator{\adeg}{\widehat{deg}}

\DeclareMathOperator{\APic}{\widehat{Pic}}

\DeclareMathOperator{\gdiv}{div}
\DeclareMathOperator{\detp}{det^{\prime}}
\DeclareMathOperator{\deta}{det^{\ast}}

\DeclareMathOperator{\Spf}{Spf}

\DeclareMathOperator{\tr}{tr}

\DeclareMathOperator{\cusp}{cusp}
\DeclareMathOperator{\cone}{cone}
\DeclareMathOperator{\vol}{vol}
\DeclareMathOperator{\loc}{loc}
\DeclareMathOperator{\ord}{ord}
\DeclareMathOperator{\arcsinh}{arcsinh}
\DeclareMathOperator{\fin}{fin}
\DeclareMathOperator{\length}{length}

\newcommand{\ccar}{1\hspace{-0.15cm}1}
\newcommand{\OO}{\mathcal{O}}
\newcommand{\XX}{\mathcal{X}}
\newcommand{\BS}{\mathcal{S}}

\newcommand{\CC}{\mathbb{C}}
\newcommand{\Int}{\mathbb{Z}}

\newcommand{\PP}{\mathbb{P}}
\newcommand{\HH}{\mathbb{H}}
\newcommand{\RR}{\mathbb{R}}
\newcommand{\QQ}{\mathbb{Q}}
\newcommand{\Circ}{\mathbb{S}^{1}}

\newcommand{\pd}{\partial}

\newcommand{\C}{\mathcal{C}}
\newcommand{\E}{\mathcal{E}}

\newcommand{\cpd}{\overline{\pd}}

\newcommand{\ov}{\overline}
\newcommand{\NN}{\mathbb{N}}

\begin{document}
\setcounter{tocdepth}{1}
\setcounter{section}{0}
\maketitle

\begin{abstract}
We generalize work of Deligne and Gillet--Soul\'e on a functorial Riemann--Roch type isometry, to the case of the trivial sheaf on cusp compactifications of Riemann surfaces $\Gamma\backslash\HH$, for $\Gamma\subset\PSL_{2}(\RR)$ a fuchsian group of the first kind, equipped with the Poincar\'e metric. This metric is singular at cusps and elliptic fixed points, and the original results of Deligne and Gillet--Soul\'e do not apply to this setting. Our theorem relates the determinant of cohomology of the trivial sheaf, with an explicit Quillen type metric in terms of the Selberg zeta function of $\Gamma$, to a metrized version of the $\psi$ line bundle of the theory of moduli spaces of pointed orbicurves, and the self-intersection bundle of a suitable twist of the canonical sheaf $\omega_{X}$. We make use of surgery techniques through Mayer--Vietoris formulae for determinants of laplacians, in order to reduce to explicit evaluations of such for model hyperbolic cusps and cones. We carry out these computations, that are of independent interest: we provide a rigorous method that fixes incomplete computations in theoretical physics, and that can be adapted to other geometries. We go on to derive an arithmetic Riemann--Roch formula in the realm of Arakelov geometry, that applies in particular to integral models of modular curves with elliptic fixed points. This vastly extends previous work of the first author, whose deformation theoretic methods were limited to the presence only of cusps. As an application, we treat in detail the case of the modular curve $X(1)$, that already reveals the interesting arithmetic content of the metrized $\psi$ line bundles. From this, we solve the longstanding question of evaluating the Selberg zeta special value $Z^{\prime}(1,\PSL_{2}(\Int))$. The result is expressed in terms of logarithmic derivatives of Dirichlet $L$ functions. In the analogy between Selberg zeta functions and Dedekind zeta functions of number fields, this formula can be seen as the analytic class number formula for $Z(s,\PSL_{2}(\Int))$. The methods developed in this article were conceived so that they afford several variants, such as the determinant of cohomology of a flat unitary vector bundle with finite monodromies at cusps. Our work finds its place in the program initiated by Burgos--Kramer--K\"uhn of extending arithmetic intersection theory to singular hermitian vector bundles.
\end{abstract}

\tableofcontents

\section{Introduction}
The Grothendieck--Riemann--Roch theorem in algebraic geometry describes the lack of commutativity of the Chern character with derived direct images of vector bundles under proper morphisms. In different degrees of generality, this fundamental statement has several incarnations, all of them interrelated. For smooth complex algebraic varieties, the theorem can be stated with values in Chow groups and also de Rham cohomology. With values in de Rham cohomology, the equality of characteristic classes can be lifted to the level of closed differential forms, by means of Chern--Weil theory and the introduction of the so-called holomorphic analytic torsion. The identity is then a generalization of the curvature formula of Bismut--Gillet--Soul\'e \cite{BGS1, BGS2, BGS3} for the determinant of cohomology endowed with the Quillen metric (relying on previous work of Bismut-Freed \cite{BF1, BF2}). In an attempt to put all these variants in a unified formalism, Deligne had the idea of lifting the Grothendieck--Riemann--Roch theorem at the level of sheaves (actually to the virtual category). Hence, instead of an equality of characteristic classes, one would ideally have a functorial isomorphism between some ``characteristic sheaves" that incarnate the direct images of characteristic classes. Furthermore, the formalism would refine the Chern--Weil construction: if the vector bundles are endowed with smooth hermitian metrics, these characteristic sheaves should carry some hermitian data as well, and the isomorphism should become an isometry in a suitable sense. Deligne achieved this goal for fibrations in curves, in an influent article \cite{Deligne}, which was a precursor of the arithmetic Riemann--Roch theorem in Arakelov geometry, to be later proven by Gillet--Soul\'e \cite{GS}.\footnote{To be rigorous, Deligne mistakenly applies Bismut--Freed's curvature theorem for Quillen connections. The subtle point is the compatibility of the Quillen connection with the holomorphic structure of the Knudsen--Mumford determinant. Nevertheless, the Riemann--Roch isometry he claims can be established by appealing instead to the results of Bismut--Gillet--Soul\'e, where this compatibility is addressed. Notice however these results came later in time.}
 Extensions to higher dimensions were provided by Franke and later Eriksson, in important works that unfortunately are still unpublished.

The Riemann--Roch isometry and the arithmetic Riemann--Roch theorem for arithmetic surfaces apply only to vector bundles endowed with \emph{smooth} hermitian metrics, and require a choice of \emph{smooth} metric on the  dualizing sheaf. Several relevant arithmetic examples one has in mind don't satisfy these assumptions though. This is already the case of the trivial sheaf and the dualizing sheaf of a modular curve, endowed with the Poincar\'e metric. This metric is singular at the cusps and the elliptic fixed points. In presence only of cusps, hence excluding elliptic fixed points, one can still prove a version of the arithmetic Riemann--Roch theorem for the trivial sheaf. This was accomplished by the first author \cite{GFM}, and was motivated by work of Takhtajan--Zograf on the corresponding curvature formula \cite{ZT}. The approach of \cite{GFM} was a combination of several compatibilities of purely algebraic geometric constructions of ``tautological bundles" on Deligne--Mumford stacks over $\Int$ of moduli spaces of curves, with purely analytic constructions in the realm of Teichm\"uller theory. Difficult results of Wolpert on degenerations of Selberg zeta functions \cite{Wolpert:Selberg}, asymptotics of geodesic length functions and goodness of the family hyperbolic metric \cite{Wolpert:Good} were of central importance. These were combined with delicate estimates of small eigenvalues in terms of geodesic lengths, provided by Burger \cite{Burger}. In presence of elliptic fixed points, these \emph{deformation theoretic} methods (algebraic or analytic) don't carry over. Furthermore, they don't work for more general bundles (for instance, flat unitary bundles) nor in higher dimensions. Therefore there is need to develop other ideas that are better suited to these more general settings.

The aim of this article is to introduce and apply a new method of extending the Riemann--Roch isometry to singular metrics. We deal with smooth complex curves that have a natural structure of orbicurves with cusps. More precisely, we consider cusp compactifcations of quotients $\Gamma\backslash\HH$, where $\HH$ is the upper half-plane and $\Gamma\subset\PSL_{2}(\RR)$ is a fuchsian group of the first kind, admitting both parabolic and elliptic elements. This is the first case that remains out of reach in previous research. Instead of deformation theory, we use analytic surgery techniques and Mayer--Vietoris type formulae for determinants of laplacians, Bost's $L^{2}_{1}$ formalism of arithmetic intersection theory, the Selberg trace formula, and ``asymptotic" evaluations of determinants of laplacians on model cusps and cones.

\subsection{Statement of the main theorem}\label{subsec:statement}
\subsubsection{} Let $\Gamma\subset\PSL_{2}(\RR)$ be a fuchsian group of the first kind. The quotient space $\Gamma\backslash\HH$ can be given a canonical structure of Riemann surface. This Riemann surface is the ``coarse moduli" for the orbifold $[\Gamma\backslash\HH]$. The latter captures the presence of torsion in $\Gamma$. The underlying topological spaces to $[\Gamma\backslash\HH]$ and $\Gamma\backslash\HH$ coincide, and it thus makes sense to confound their points. The points with non-trivial automorphisms are called \emph{elliptic fixed points}. They are in bijective correspondence with conjugacy classes of primitive elliptic elements in $\Gamma$. By adding a finite  number of \emph{cusps}, the Riemann surface $\Gamma\backslash\HH$ can be completed into a compact Riemann surface $X$. Cusps are in bijective correspondence with conjugacy classes of primitive parabolic elements in $\Gamma$. We denote the set of elliptic fixed points and cusps by $p_{1},\ldots,p_{n}$, and assign to them multiplicities $m_{1},\ldots,m_{n}\in\NN_{>1}\cup\lbrace\infty\rbrace$. The multiplicity of a cusp is $\infty$, while for an elliptic fixed point it is the order of its automorphism group. Finally, by $c$ we denote the number of cusps of $\Gamma\backslash\HH$.

\subsubsection{}The Poincar\'e or hyperbolic \emph{riemannian} metric on $\HH$ is given by
\begin{displaymath}
	ds_{\hyp}^{2}=\frac{dx^{2}+dy^{2}}{y^{2}},
\end{displaymath}
where $x+iy$ is the usual parametrization of $\HH$. In terms of the holomorphic coordinate $\tau=x+iy$, this is also written
\begin{equation}\label{eq:riemannian}
	ds_{\hyp}^{2}=\frac{|d\tau|^{2}}{(\Imag\tau)^{2}}.
\end{equation}
This metric is invariant under $\PSL_{2}(\RR)$, and descends to a riemannian metric on $[\Gamma\backslash\HH]$ (in the orbifold sense). As a metric on $\Gamma\backslash\HH$ it has singularities at the elliptic fixed points, of hyperbolic conical type. We can also interpret this metric as a singular metric on the compactified Riemann surface $X$, with additional singularities at the cusps. In a neighborhood of an elliptic fixed point $p$ of order $\ell$, there is a distinguished holomorphic coordinate $w$, unique up to a constant of modulus 1, such that the singular riemannian tensor is locally written as
\begin{displaymath}
	ds_{\hyp}^{2}=\frac{4|dw|^{2}}{\ell^{2}|w|^{2-2/\ell}(1-|w|^{2/\ell})^{2}}.
\end{displaymath}
In a neighborhood of a cusp $q$, there is also a distinguished coordinate $z$, unique up to a constant of modulus $1$, such that the singular riemannian tensor is given by
\begin{displaymath}
	ds_{\hyp}^{2}=\frac{|dz|^{2}}{(|z|\log|z|)^{2}}.
\end{displaymath}
This type of coordinates were called \emph{rs} (rotationally symmetric) by Wolpert \cite{Wolpert:Good, Wolpert:cusp}. Following this author, we define a renormalization of the hyperbolic metric at elliptic fixed points and cusps, by the rules
\begin{displaymath}
	\|dw\|_{W,p}=1,\quad \|dz\|_{W,q}=1.
\end{displaymath}
This actually defines hermitian metrics on the holomorphic cotangent spaces at $p$ and $q$, written $\omega_{X,p}$ and $\omega_{X,q}$. We call them \emph{Wolpert metrics} (see \textsection\ref{subsubsec:truncated} and Definition \ref{definition:Wolpert}). We then formally define
\begin{displaymath}
	\psi_{W}=\sum_{i}\left(1-\frac{1}{m_{i}^{2}}\right)(\omega_{X,p_{i}},\|\cdot\|_{W,p_{i}}).
\end{displaymath}
This object is to be understood as a hermitian $\QQ$-line bundle over $\Spec\CC$. Letting $m=\prod_{m_{i}<\infty}m_{i}$, then what is really defined is the hermitian line bundle
\begin{displaymath}
	\psi_{W}^{\otimes m^{2}}=\bigotimes_{i}(\omega_{X,p_{i}},\|\cdot\|_{W,p_{i}})^{m^{2}(1-m_{i}^{-2})}.
\end{displaymath}
The underlying $\QQ$-line bundle is denoted by $\psi$. Our criterion of use of additive or multiplicative notations for hermitian $\QQ$-line bundles will always be to simplify notations as much as possible. 

\subsubsection{}\label{subsubsec:normal-metric} The singular riemannian metric on $X$ induces a singular hermitian metric on the $\QQ$-line bundle
\begin{displaymath}
	\omega_{X}(D),\quad D:=\sum_{i}\left(1-\frac{1}{m_{i}}\right)p_{i}.
\end{displaymath}
By $\omega_{X}(D)_{\hyp}$ we denote the resulting hermitian $\QQ$-line bundle over $X$. To avoid any confusion, at this point we make a remark on normalizations. While the riemannian metric on $\HH$ is given by \eqref{eq:riemannian}, the corresponding \emph{hermitian} metric on the holomorphic tangent bundle $T_{\HH}$ is given by
\begin{displaymath}
	\frac{|d\tau|^{2}}{2(\Imag\tau)^{2}},
\end{displaymath}
hence the holomorphic vector field $\partial/\partial\tau$ has square norm $1/2(\Imag\tau)^{2}$. With this normalization, the hermitian metric on $T_{\HH}$ has constant Ricci curvature $-1$, and the expected K\"ahler identity between the scalar and the $\ov{\partial}$ laplacians is satisfied. The hermitian metric on $\omega_{X}(D)_{\hyp}$ is obtained by duality and descent from the hermitian metric on $T_{\HH}$.

The expression of the riemannian tensor in \emph{rs} coordinates shows that $\omega_{X}(D)_{\hyp}$ fits the $L_{1}^{2}$ formalism of Bost \cite{Bost}. Namely, as a metric on this line bundle, the singularities are mild enough so that the metrized \emph{Deligne pairing}
\begin{displaymath}
	\langle\omega_{X}(D)_{\hyp},\omega_{X}(D)_{\hyp}\rangle
\end{displaymath}
is defined. This is a hermitian $\QQ$-line bundle over $\Spec\CC$. One can also appeal to the generalized arithmetic intersection theory of K\"uhn \cite{Kuhn}, which was developed to deal with this kind of metrics. We refer to Deligne's article \cite{Deligne} and Soul\'e's survey \cite{Soule} for more information on the Deligne pairing.

\subsubsection{} The determinant of cohomology of the trivial sheaf $\OO_{X}$ on $X$ is the complex line
\begin{displaymath}
	\det H^{\bullet}(X,\OO_{X})=\det H^{0}(X,\OO_{X})\otimes\det H^{1}(X,\OO_{X})^{-1}.
\end{displaymath}
We define a \emph{Quillen metric} on it as follows. First of all, there is no difficulty in defining the $L^{2}$ metric for the singular hyperbolic metric, since it has finite volume (this is discussed in \textsection\ref{subsection:Deligne-isom}). The Quillen metric is a rescaling of the $L^{2}$ metric. Denote by $Z(s,\Gamma)$ the Selberg zeta function of $\Gamma$. For $\Real(s)>1$, it is defined by an absolutely and locally uniformly convergent product
\begin{displaymath}
	Z(s,\Gamma)=\prod_{\gamma}\prod_{k=0}^{\infty}(1-e^{-(s+k)\ell_{\hyp}(\gamma)}),
\end{displaymath} 
where $\gamma$ runs over the closed primitive geodesics in $\Gamma\backslash\HH$ for the 
hyperbolic metric, and $\ell_{\hyp}(\gamma)$ denotes the hyperbolic length of $\gamma$.  The Selberg zeta function has a meromorphic continuation to the whole plane $\CC$, and has a simple zero at $s=1$. Moreover, $Z^{\prime}(1,\Gamma)$ is a positive real number. We then define
\begin{displaymath}
	\|\cdot\|_{Q}=(C(\Gamma)Z^{\prime}(1,\Gamma))^{-1/2}\|\cdot\|_{L^{2}},
\end{displaymath}
where $C(\Gamma)$ is the real positive constant determined by the cumbersome expression
\begin{equation}\label{eq:ctt-Gamma}
	\begin{split}
		\log C(\Gamma)=&\sum_{m_{i}<\infty}\sum_{k=0}^{m_{i}-2}\frac{2k+1-m_{i}}{m_{i}^{2}}
		\log\Gamma\left(\frac{k+1}{m_{i}}\right)+\frac{1}{3}\sum_{m_{i}<\infty}\log m_{i}
		-\frac{1}{6}\sum_{m_{i}<\infty}\frac{\log m_{i}}{m_{i}^{2}}\\
		&-\frac{1}{2}\sum_{m_{i}<\infty}\frac{\log m_{i}}{m_{i}}
		-\frac{1}{6}\sum_{m_{i}<\infty}m_{i}\log m_{i}-\left(\-2\zeta^{\prime}(-1)-\frac{1}{6}\right)\sum_{m_{i}<\infty}m_{i}\\
		&-\left(2\zeta^{\prime}(-1)+\frac{1}{2}\log(2\pi)+\frac{2}{3}\log 2-\frac{\gamma}{6}-\frac{1}{6}\right)
		   \sum_{m_{i}<\infty}\frac{1}{m_{i}}
		-\frac{1}{6}(\log 2)\sum_{m_{i}<\infty}\frac{1}{m_{i}^{2}}\\
		&+g\left(\log(2\pi)+\frac{1}{3}\log 2-\frac{1}{3}\right)+n\left(\frac{1}{2}\log 2-\frac{5}{12}\right)\\
		&-c\left(\log 2-\frac{3}{4}\right)-\log(2\pi)+\frac{2}{3}\log 2+\frac{1}{3}.
	\end{split}
\end{equation}
Here $g$ is the genus of $X$, $\zeta(s)$ is the Riemann zeta function, and $\gamma$ is the Euler--Mascheroni constant. The determinant of cohomology together with this Quillen metric will be denoted $\det H^{\bullet}(X,\OO_{X})_{Q}$.

\subsubsection{} Deligne's isomorphism for the trivial line bundle on $X$ is a canonical, up to sign, isomorphism
\begin{equation}\label{eq:deligne-iso}
	\det H^{\bullet}(X,\OO_{X})^{\otimes 12}\overset{\sim}{\longrightarrow}\langle\omega_{X},\omega_{X}\rangle.
\end{equation}
It can be put in families of curves, and commutes with base change. It is unique with these properties.\footnote{For the trivial sheaf, and with a different presentation, Deligne's isomorphism was actually obtained by Mumford \cite{Mumford}.} Again, we refer to \cite{Deligne, Soule} for more information on this isomorphism. By the bilinearity of the Deligne pairing with respect to tensor product of line bundles, \eqref{eq:deligne-iso} induces a canonical isomorphism of $\QQ$-line bundles
\begin{displaymath}
	\det H^{\bullet}(X,\OO_{X})^{\otimes 12}\otimes\psi\overset{\sim}{\longrightarrow}\langle\omega_{X}(D),\omega_{X}(D)\rangle.
\end{displaymath}

\begin{theorem}\label{theorem:main}
Deligne's isomorphism induces a canonical (up to sign) isometry of hermitian $\QQ$-line bundles
\begin{equation}\label{eq:main}
	\det H^{\bullet}(X,\OO_{X})_{Q}^{\otimes 12}\otimes\psi_{W}\overset{\sim}{\longrightarrow}\langle\omega_{X}(D)_{\hyp},\omega_{X}(D)_{\hyp}\rangle.
\end{equation}
\end{theorem}
The ``connaisseur" will realize the absence of the exotic $R$-genus of Gillet--Soul\'e in our formulation. To render the presentation more compact, we included it in the constant $C(\Gamma)$ in the definition of the Quillen metric.\footnote{The axiomatic approach proposed by \cite{BFL} shows that a theory of holomorphic analytic torsion is unique up to a topological genus. The normalization we follow corresponds to the \emph{homogenous} theory in \cite{BFL}.} In the sequel we discuss the strategy of proof and the applications of this statement.
\newpage
\subsection{Simplified strategy of the proof}\label{subsec:strategy}
\subsubsection{}\label{subsubsec:divergence} Our proof accomplishes the naive and difficult idea of comparing the Selberg zeta value $Z^{\prime}(1,\Gamma)$ to the determinant of a scalar laplacian on $X$, for an appropriate choice of smooth riemannian metric.\footnote{This comparison is actually equivalent to an anomaly formula of Polyakov type, for non-compactly supported deformations of the hyperbolic metric.} We truncate the hyperbolic metric in $\varepsilon$-neighborhoods of the cusps and elliptic fixed points, and replace it by a smooth flat one in this area. This is done by freezing the value of the hyperbolic metric at height $\varepsilon$ in \emph{rs} coordinates. For instance, in a neighborhood of an elliptic fixed point $p$ of order $\ell$, and for $w$ a \emph{rs} coordinate at $p$, we replace the hyperbolic metric on $|w|\leq\varepsilon$ with the flat metric
\begin{displaymath}
	\frac{4|dw|^{2}}{\ell^{2}\varepsilon^{2-2/\ell}(1-\varepsilon^{2/\ell})^{2}}.
\end{displaymath}
Similarly at the cusps. We refer to those as $\varepsilon$ metrics.

\begin{figure}[h]
\begin{center}
\includegraphics[scale=0.4]{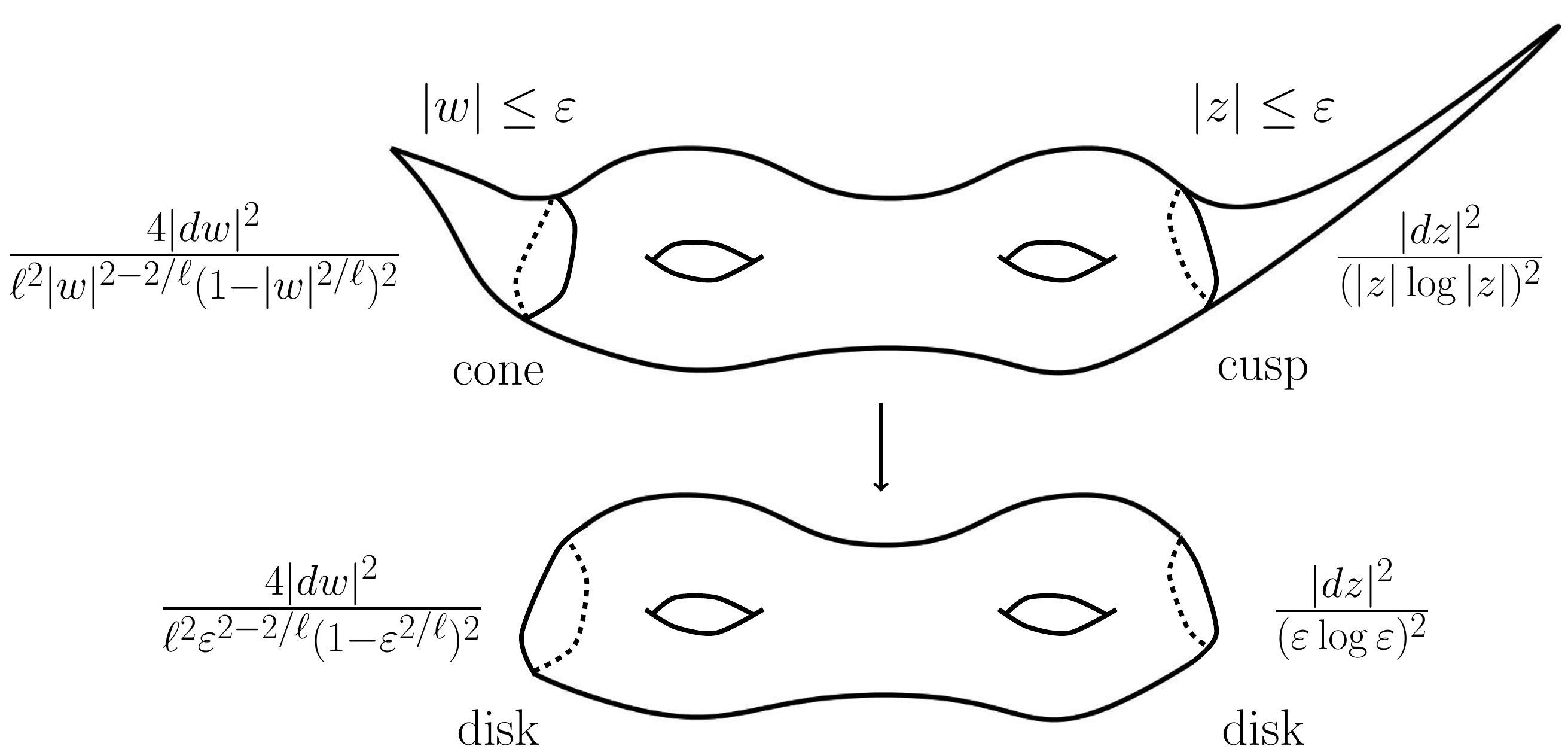}
\end{center}
\caption{The freezing of the hyperbolic metric at height $\varepsilon$.}
\end{figure}

Let us leave aside, for the moment, the fact that the resulting riemannian metric is only continuous and piecewise smooth, and pretend that the Riemann--Roch isometry holds for this metric. We would then like to let $\varepsilon\to 0$. Of course, both sides of \eqref{eq:deligne-iso} (with Quillen and Deligne pairing metrics depending on $\varepsilon$) will blow up in the limit. However, on the Deligne pairing side, easy computations show the exact shape of divergence, and how to correct the Deligne pairing metrics so as to obtain a finite limit. This is responsible for the appearance of the twist by $D$ and the line bundle $\psi$. The divergent quantities one needs to substract to the right hand side, necessarily have to compensate the divergence of the Quillen metrics on the cohomological side. This provides an \emph{ad hoc} definition of a Quillen metric in the limit, which tautologically (\emph{i.e.}~by construction) fits into an isometry of the kind \eqref{eq:main}. The details occupy Section \ref{section:DRR}, and abut to Proposition \ref{prop:naive_laplacian}.

\subsubsection{} The difficulty now is to give a spectral meaning to the so-defined limit Quillen metric. At this point we introduce analytic surgery. Let us suppose, for simplicity, that we only have one cusp and one elliptic fixed point, as illustrated in the figure above. We write $A$ for the $\varepsilon$-neighborhood of the cusp and $B$ for an $\varepsilon$-neighborhood of the elliptic fixed point, with respect to \emph{rs} coordinates. Let $C$ be the complement. By $\Delta_{A,\hyp}$, $\Delta_{B,\hyp}$, $\Delta_{C,\hyp}$ we denote the hyperbolic laplacians, with Dirichlet boundary conditions on these regions. Also, we write $\Delta_{A,\varepsilon}$ resp.~$\Delta_{B,\varepsilon}$ for the laplacians on $A$ resp.~$B$ with respect to the $\varepsilon$ metrics, and with Dirichlet boundary conditions. Pretending that the $\varepsilon$ metric on the whole $X$ is smooth, we also have the laplacian $\Delta_{X,\varepsilon}$ and its zeta regularized determinant $\detp\Delta_{X,\varepsilon}$ (the prime indicates that we remove the zero eigenvalue). We invoke the Mayer--Vietoris formula of Burghelea--Friedlander--Kappeler \cite{BFK}:
\begin{equation}\label{eq:mv-intro-1}
	\frac{\detp\Delta_{X,\varepsilon}}{\det\Delta_{A,\varepsilon}\det\Delta_{B,\varepsilon}\det\Delta_{C,\varepsilon}}=\frac{V_{\varepsilon}}{\ell_{\varepsilon}}\detp R_{\varepsilon}.
\end{equation}
Here, $V_{\varepsilon}$ denotes the volume of $X$, $\ell_{\varepsilon}$ the total length of the boundary of $C$, and $R_{\varepsilon}$ the Dirichlet-to-Neumann jump operator for the harmonic Dirichlet problem. All these are defined with respect to the $\varepsilon$ metric. An analogous formula for the hyperbolic metric holds as well, by an easy adaptation of work of Carron \cite{Carron}. It involves the relative determinant of the couple $(\Delta_{X,\hyp},\Delta_{A,\hyp})$, which plays the role of the quotient of their \emph{undefined} determinants:
\begin{displaymath}
	\det(\Delta_{X,\hyp},\Delta_{A,\hyp})\,``=\frac{\det\Delta_{X,\hyp}}{\det\Delta_{A,\hyp}}\,".
\end{displaymath}
It also involves the determinant $\det\Delta_{B,\hyp}$, that, despite the conical singularities, is well-defined because of their orbifold origin. The formula is then
\begin{equation}\label{eq:mv-intro-2}
	\frac{\det(\Delta_{X,\hyp},\Delta_{A,\hyp})}{\det\Delta_{B,\hyp}\det\Delta_{C,\hyp}}=\frac{V_{\hyp}}{\ell_{\hyp}}\detp R_{\hyp}.
\end{equation}
An elementary but crucial observation is the coincidence
\begin{displaymath}
	\frac{\detp R_{\hyp}}{\ell_{\hyp}}=\frac{\detp R_{\varepsilon}}{\ell_{\varepsilon}}.
\end{displaymath}
Moreover, by construction, $\det\Delta_{C,\varepsilon}=\det\Delta_{C,\hyp}$. Therefore equations \eqref{eq:mv-intro-1}--\eqref{eq:mv-intro-2} imply
\begin{equation}\label{eq:mv-intro-3}
	\det(\Delta_{X,\hyp},\Delta_{A,\hyp})=\frac{V_{\hyp}}{V_{\varepsilon}}\detp\Delta_{X,\varepsilon}\frac{\det\Delta_{B,\hyp}}{\det\Delta_{A,\varepsilon}\det\Delta_{B,\varepsilon}}.
\end{equation}
There is a further step, that consists in writing the relative determinant $\detp(\Delta_{X,\hyp},\Delta_{A,\hyp})$ in terms of the Selberg zeta function. This is done by introducing an auxiliary operator and applying the Selberg trace formula. One can show that
\begin{displaymath}
	\detp(\Delta_{X,\hyp},\Delta_{A,\hyp})=\alpha(\varepsilon)\frac{Z^{\prime}(1,\Gamma)}{\det\Delta_{\ps}},
\end{displaymath}
where $\Delta_{\ps}$ is the so-called pseudo-laplacian for the hyperbolic metric on the $\varepsilon$-neighborhood of the cusp, and $\alpha(\varepsilon)$ is an explicit elementary function of $\varepsilon$. We conclude with
\begin{equation}\label{eq:mv-intro-4}
	Z^{\prime}(1,\Gamma)=\frac{1}{\alpha(\varepsilon)}\frac{V_{\hyp}}{V_{\varepsilon}}\detp\Delta_{X,\varepsilon}\frac{\det\Delta_{\ps}\det\Delta_{B,\hyp}}{\det\Delta_{A,\varepsilon}\det\Delta_{B,\varepsilon}}.
\end{equation}
But the divergence of $\detp\Delta_{X,\varepsilon}$ was already determined in the previous step \textsection\ref{subsubsec:divergence}, and $Z^{\prime}(1,\Gamma)$ does not depend on $\varepsilon$. If we could exactly evaluate all the other quantities in \eqref{eq:mv-intro-4} and compare them with the known divergence of $\detp\Delta_{X,\varepsilon}$, we would infer an expression of the limit Quillen metric in terms of $Z^{\prime}(1,\Gamma)$, with an explicit constant $C(\Gamma)$ as required for Theorem \ref{theorem:main}! The Mayer--Vietoris formulas are given in Section \ref{section:Mayer--Vietoris}. The first step towards the spectral interpretation constitutes Section \ref{section:spectral}. The computations with the Selberg zeta function are collected in Section \ref{section:selberg}.

\subsubsection{} The previous digression reduces our work to the explicit evaluation of $\det\Delta_{\ps}$, $\det\Delta_{B,\hyp}$, $\det\Delta_{A,\varepsilon}$, and $\det\Delta_{B,\varepsilon}$. For some simple geometries, such computations have been considered by several authors. For instance, for euclidean disks the value was obtained by Spreafico \cite{Spreafico}. From this one easily gets the values of $\det\Delta_{A,\varepsilon}$ and $\det\Delta_{B,\varepsilon}$. For $\det\Delta_{\ps}$ and $\det\Delta_{B,\hyp}$, one can find some inspiration in the theoretical physics literature \cite{Barvinsky, Bordag, Fucci, Flachi}. Unfortunately, the various published methods that tackle these kind of determinants suffer from serious gaps, that make them unsuitable for mathematical purposes. One of the aims of our article is to provide a rigorous approach, through a tricky use of asymptotic properties of Bessel and Legendre functions. This part of our work is of independent interest for theoretical physics. The spectral problems for the model cusps and cones are presented in Section \ref{section:models}. The explicit computations of determinants of laplacians are differed to sections \ref{section:evaluation-determinants} and \ref{subsection:determinant-cone}.

\subsubsection{} The simplified strategy reviewed above presents some issues. The operation of freezing the hyperbolic metric near the singularities produces a piecewise smooth riemannian metric on $X$. We thus have to smoothen the metric at the jumps. This is a technical and routine step, but requires some work. Also, for the  determinants $\det\Delta_{\ps}$, $\det\Delta_{B,\hyp}$ we only obtain their asymptotics as $\varepsilon\to 0$. This is enough for our purposes, as well as in theoretical physics.\footnote{It is not even clear that the exact values obtained by the physicists methods are correct.}

To conclude the presentation of the method, we bring the reader's attention to the parallelism between the strategy we follow and the proof of the Selberg trace formula for non cocompact fuchsian groups of the first kind: truncation of fundamental domains, Maass--Selberg relations, and miraculous cancellations when the truncation exhausts the fundamental domain.  

\subsection{Applications: arithmetic Riemann--Roch and the special value $Z^{\prime}(1,\PSL_{2}(\Int))$}
\subsubsection{} The advantage of the Riemann--Roch isometry is that it easily leads to arithmetic versions of the Riemann--Roch formula, in the sense of Arakelov geometry. In Section \ref{section:arithmetic} we discuss arithmetic applications of Theorem \ref{theorem:main}, that we next summarize. Let $K$ be a number field and $\XX\to\BS=\Spec\OO_{K}$ a flat and projective regular arithmetic surface. We suppose given sections $\sigma_{1},\ldots,\sigma_{n}$, that are generically disjoint. We also assume that for every complex embedding $\tau\colon K\hookrightarrow\CC$, the compact Riemann surface $\XX_{\tau}(\CC)$ arises as the compactification of a quotient $\Gamma_{\tau}\backslash\HH$, and that the set of elliptic fixed points and cusps is precisely given by the sections. The constructions of \textsection\ref{subsec:statement} combined with the arithmetic structure of $\XX$, produce hermitian $\QQ$-line bundles over $\BS$, with classes in the arithmetic Picard group (up to torsion) $\APic(\BS)\otimes_{\Int}\QQ$. We use similar notations as in the complex case. Through the arithmetic degree map
\begin{displaymath}
	\adeg\colon\APic(\BS)\otimes_{\Int}\QQ\longrightarrow\RR
\end{displaymath}
we build numerical invariants, such as the arithmetic self-intersection number
\begin{displaymath}
	(\omega_{\XX/\BS}(D)_{\hyp},\omega_{\XX/\BS}(D)_{\hyp})=\adeg\,\langle\omega_{\XX/\BS}(D)_{\hyp},\omega_{\XX/\BS}(D)_{\hyp}\rangle\in\RR.
\end{displaymath}
The arithmetic degrees of the \emph{psi} classes encapsulate interesting information. Loosely speaking, they measure how far the \emph{rs} coordinates are from being (formal) algebraic. For instance, on moduli spaces of genus $0$ curves with $n$ cusps, and as a consequence of the arithmetic Riemann--Roch formula of the first author, they define heights with good finiteness properties \cite[Sec.~7]{GFM2}. Also, it is a striking coincidence that this kind of invariants play a prominent role in the work of Bost on Lefschetz theorems on arithmetic surfaces \cite[Thm.~1.2]{Bost}.

\subsubsection{} A straightforward application of Theorem \ref{theorem:main} produces the following arithmetic Riemann--Roch formula (Theorem \ref{theorem:ARR}):
\begin{equation}\label{eq:arr-intro}
	\begin{split}
		12\adeg H^{\bullet}(\XX,\OO_{\XX})_{Q}-\delta+\adeg\psi_{W}=&
		(\omega_{\XX/\BS}(D)_{\hyp},\omega_{\XX/\BS}(D)_{\hyp})\\
		&-\sum_{i\neq j}\left(1-\frac{1}{m_{i}}\right)\left(1-\frac{1}{m_{j}}\right)(\sigma_{i},\sigma_{j})_{\fin},
	\end{split}
\end{equation}
where $\delta$ is a suitable measure of the bad reduction of the structure morphism $\XX\to\BS$ (in terms of Artin conductors), and the right most intersection numbers account for the intersections of the sections happening at finite places.

\subsubsection{} Because our results cover arbitrary fuchsian groups, they apply to situations where dramatic simplifications occur. This is the case of $\PP^{1}_{\Int}$, seen as an integral model of $\PSL_{2}(\Int)\backslash\HH\cup\lbrace\infty\rbrace$. That is, the coarse modular curve $X(1)\to\Spec\Int$. The cusp and the elliptic fixed points $i$ and $\rho=e^{2\pi i/3}$ give raise to sections $\sigma_{\infty}$, $\sigma_{i}$, and $\sigma_{\rho}$, and the arithmetic Riemann--Roch formula \eqref{eq:arr-intro} applies with these data. This case provides a beautiful example of the arithmetic significance of $\psi_{W}$. Indeed, we show that the contribution of $\sigma_{\infty}$ is equal to zero, as a consequence of the compatibility between the theory of the Tate curve and the Fourier expansions of modular forms (known as the $q$-expansion principle). We also show that $\sigma_{i}$ and $\sigma_{\rho}$ contribute with the Faltings heights of the corresponding CM elliptic curves. These can be evaluated by the Chowla-Selberg formula, in terms of logarithmic derivatives of Dirichlet $L$ functions (see Lemma \ref{lemma:selberg-2} \emph{infra} for the precise formulation). Finally, the arithmetic self-intersection number was computed by Bost and K\"uhn \cite{Kuhn}, and relies on the first Kronecker limit formula. Because the determinant of cohomology of the trivial sheaf of $\PP^{1}_{\Int}$ reduces to $\Int$, this example leads to an explicit evaluation of the special value $Z^{\prime}(1,\PSL_{2}(\Int))$ (Theorem \ref{thm:selberg-zeta}). In particular, we obtain
\begin{equation}\label{eq:selberg-intro}
	\log Z^{\prime}(1,\PSL_{2}(\Int))\in\QQ\left\langle\frac{L^{\prime}(0,\chi_{i})}{L(0,\chi_{i})},\frac{L^{\prime}(0,\chi_{\rho})}{L(0,\chi_{\rho})},\frac{\zeta^{\prime}(0)}{\zeta(0)},
	\frac{\zeta^{\prime}(-1)}{\zeta(-1)},\gamma,\log 3, \log 2\right\rangle,
\end{equation}
where $\chi_{i}$ is the quadratic character of $\QQ(i)$ and $\chi_{\rho}$ is the quadratic character of $\QQ(\rho)$. This is an intriguing relation, especially in view of the expression of $Z(s,\PSL_{2}(\Int))$ obtained by Sarnak 
\cite[Cor.~1.5]{Sarnak82}:
\begin{displaymath}
	Z(s,\PSL_{2}(\Int))=\prod_{d\in\mathscr{D}}\prod_{k=0}^{\infty}(1-\varepsilon_{d}^{-2(s+k)})^{h(d)}.
\end{displaymath}
Here $\mathscr{D}$ is the set of square-free integers $d>0$, with $d\equiv 0,1\mod 4$, $\varepsilon_{d}>1$ is the fundamental solution of the Pell equation $x^{2}-dy^{2}=4$ and $h(d)$ is the class number of binary integral quadratic forms of discriminant $d$. To our knowledge, the evaluation of $Z^{\prime}(1,\PSL_{2}(\Int))$ was a longstanding question among specialists, and there is currently no other approach to this kind of computations. It would be interesting to have a direct ``analytic number theoretic" evaluation, and differently see how the special values of $L$ functions above arise. The advantage of the Arakelov theoretic strategy is that the result has a geometric interpretation. 

A more general but anecdotic remark, is the formal resemblance between \eqref{eq:arr-intro} to the analytic class number formula of Dedekind zeta functions. For fuchsian groups and their Selberg zeta functions, the isometry theorem and its arithmetic counterpart, can thus be seen as providing the analogue to the analytic class number formula. In particular, we call the explicit expression for $\log Z^{\prime}(s,\PSL_{2}(\Int))$ the analytic class number formula for $\PSL_{2}(\Int)$.

\newpage
\subsection{Other implications and perspectives}
\subsubsection{} The main theorem \ref{theorem:main} has other consequences that we don't discuss in this article. Nevertheless, we would like to comment on them:
\begin{enumerate}
	\item[--] the theorem vastly generalizes the Takhtajan--Zograf local index theorem \cite{ZT} on moduli spaces of punctured Riemann surfaces, by including elliptic fixed points and refining their curvature equality to an isometry of line bundles (the isometry can be put in family, on the moduli space of curves with marked points with weights, then one can take Chern forms). For an elliptic fixed point $p$, the role of the parabolic Eisenstein series $E_{q}(z,2)$ associated to the cusp $q$ in the definition of the Takhtajan--Zograf K\"ahler metric, is now played by the automorphic Green's function $G_{s}(z,p)$ at $s=2$. At the time of typing this article, Takhtajan and Zograf were indeed able to extend \cite{ZT} in presence of elliptic fixed points too. We refer the interested reader to their forthcoming work.
		\item[--] as we already explained, the main theorem also generalizes the Riemann--Roch isometry obtained by the first author \cite{GFM} in presence only of cusps. Even in this case, our methods are completely different and don't make use of deformation theory, nor of delicate properties of small eigenvalues, geodesic lengths, etc. Actually, by reversing the reasoning followed in \cite{GFM}, one can recover Wolpert's results on the degeneracy of Selberg zeta functions \cite{Wolpert:Selberg}. Our theorem even covers the orbifold setting.
	\item[--] by combining our isometry with the work of Schumacher--Trapani \cite{Sch-Trap}, it seems possible to deduce a result of Garbin--Jorgenson \cite[Cor.~7.2]{GJ} on the convergence of Selberg zeta values under elliptic degeneration.
\end{enumerate}
These implications of the main theorem, and the fact that the cited results themselves are difficult, provide further justification of the intricacy of the computations we carry out.

\subsubsection{} The strategy of proof that we follow can be adapted to other interesting situations, for which the present work is actually a ``toy example". This is the case of flat vector bundles corresponding to unitary representations of a fuchsian group, at least under a finite monodromy condition at the cusps.\footnote{This was the origin of the present article.} It is then plausible that the isometry, applied to dual pairs of flat vector bundles, implies relevant cases of the Saito--Terasoma theorem \cite{ST} on periods of connections on arithmetic varieties, in the dimension 1 case (to which they ultimately reduce by use of Lefschetz pencils). This will be a theme of our future research.

\subsection{Related works}
Other authors have studied neighbouring questions, and we would like to say some words about their work.

\subsubsection{} In absence of elliptic fixed points, hence only cusps, there is a different approach by Weng \cite{Weng} to Riemann--Roch type isometries. However the proof of the main theorem in \emph{loc.~cit.} is wrong. Also in T.~Hahn's PhD. thesis, the author worked with metric degenerations and heat kernel methods, and proved an arithmetic Riemann--Roch theorem in presence of cusps. The latter holds only up to an undetermined topological constant. Notice that in \cite{GFM} as well as the present article, the main difficulty consists in precisely pinning down this constant.

\subsubsection{} Recently, Friedman--Jorgenson--Smajlovic \cite{FJS} propose the use of scattering theory to define analytic torsion for Riemann surfaces with cusps. While it is conceptually satisfactory to have a systematic approach to defining analytic torsion, this is unfortunately not enough to establish an arithmetic Riemann--Roch formula (or a Riemann--Roch isometry).

\subsubsection{} In higher dimensions, there is work in progress of the first author with D. Eriksson and S. Sankaran, on an arithmetic Riemann-Roch theorem on Hilbert modular surfaces. The approach heavily depends on the theory of authomorphic representations, and can't be transposed to other settings. 

\subsubsection{} Recently J.-M. Bismut informed us on his results comparing orbifold holomorphic analytic torsion (introduced and studied by Ma \cite{Ma}) with usual holomorphic analytic torsion, when they are both defined. More precisely, his geometric setting consists in the quotient of a compact complex variety by an involution. The techniques should extend to a quotient by the action of any cyclic group. The result is compatible with the arithmetic Lefschetz fixed point formula of K\"ohler--R\"ossler \cite{KR}. Actually, Theorem \ref{theorem:main} for cocompact fuchsian groups (hence only elliptic fixed points) is also compatible with the theorem of K\"ohler--R\"ossler.

\subsection{Structure of the article}
Let us briefly review the organization and contents of this article. In Section \ref{section:models}, we discuss the model cases of the geometries and degenerate metrics we will have to tackle (hyperbolic cusps and cones), and the spectral theory of the corresponding laplacians. We state the evaluations of the determinants of laplacians. Section \ref{section:Mayer--Vietoris} introduces the global situation (our compactified orbicurves) and explains how the local models serve to the study of determinants of global laplacians. Here is where we use surgery techniques, through the Mayer--Vietoris formulae that we recall. Next, in Section \ref{section:DRR}, we work on the Riemann--Roch isometry (for smooth hermitian metrics) and derive from it a first weak version of the main theorem, with a naive Quillen metric that lacks of spectral interpretation. We will naturally see the appearance of the \emph{psi} line bundle and the Wolpert metric. The aim of Section \ref{section:spectral} is to explain how the Mayer--Vietoris formula will intervene in giving a spectral interpretation of the naive Quillen metric. The argument will reduce to the explicit evaluation of determinants of the (pseudo-)laplacians on the model cusp and the model cone, a \emph{tour de force} that occupies sections \ref{section:evaluation-determinants} and \ref{subsection:determinant-cone}. In Section \ref{section:selberg} we apply the Selberg trace formula to express a suitable relative determinant in terms of the Selberg zeta function, following a well-known method of d'Hoker--Phong, Sarnak, and Efrat. The proof of the main theorem is the content of Section \ref{section:proof-main}. Finally, the arithmetic applications conclude the article in Section \ref{section:arithmetic}.

A preliminary form of our work has been presented in several seminars and conferences since 2011. It took us several years to develop rigorous methods for the evaluation of determinants of laplacians. We apologize for the important delay since the first announcements were made.

\section{Hyperbolic cusps, hyperbolic cones, and laplacians}\label{section:models}
In this section we describe the local geometry of a hyperbolic cusp and a hyperbolic cone, as well as the spectral theory of their scalar laplacians. We also state explicit evaluations of their zeta regularized determinants, whose proves are differed to sections \ref{section:evaluation-determinants} and \ref{subsection:determinant-cone}.

\subsection{The model cusp}\label{subsec:model-cusp}
We define the model hyperbolic cusp and introduce several related Laplace type operators acting on suitable functional spaces.

\subsubsection{}\label{subsubsec:model-cusp} Let $\Circ$ be the unit circle parametrized by the coordinate $x\in [0,1]$ and $a>0$ a positive real number. The hyperbolic cusp, with apex at infinity and horocycle at height $a$, is the non-compact surface with boundary $\C_{a}:=\Circ\times [a,+\infty)$ endowed with the Poincar\'e riemannian metric
\begin{displaymath}
	ds^{2}_{\cusp}=\frac{dx^{2}+dy^{2}}{y^{2}}.
\end{displaymath}
Observe this is a complete metric of gaussian curvature $-1$ and finite volume. 
 
The hyperbolic cusp $\C_{a}$ can equivalently be presented as a Riemann surface with boundary, parame\-trized by the complex coordinate $z=e^{2\pi i (x+iy)}$, valued in the punctured disk $D^{\times}(0,e^{-2\pi a})$. The hyperbolic metric can then be written as
\begin{displaymath}
	ds_{\cusp}^{2}=\frac{|dz|^{2}}{(|z|\log |z|)^{2}}.
\end{displaymath}
The coordinate $z$ is uniquely determined by this condition, up to a factor of modulus 1. We call it a \emph{rs} coordinate.
\subsubsection{} The hyperbolic laplacian $\Delta_{D}$ on the cusp $\C_{a}$ with Dirichlet boundary condition at height $a$ 
is constructed as follows. Consider the scalar laplacian as a densely defined positive symmetric operator in $L^{2}(\C_{a}):=L^{2}(\C_{a},ds_{\cusp}^{2})$
\begin{displaymath}
	\Delta=-y^{2}\left(\frac{\pd^{2}}{\pd x^{2}}+\frac{\pd^{2}}{\pd y^{2}}\right):\mathcal{C}_{0}^{\infty}(\overset{\circ}{\C}_{a})\to \mathcal{C}_{0}^{\infty}(\overset{\circ}{\C}_{a})\subset{L}^{2}(\C_{a}).
\end{displaymath}
We define $\Delta_{D}$ to be the Friedrichs extension of $\Delta$. By the completeness of the hyperbolic metric at the cusp, this is the unique closed positive self-adjoint extension. Its domain consists of functions $f\in L^{2}(\C_{a})$ such that the distribution $\Delta f$ is again in $L^{2}(\C_{a})$ and $f$ vanishes on the boundary $y=a$.

\subsubsection{}\label{subsubsec:pseudo-laplacian} The pseudo-laplacian $\Delta_{\ps}$ is the Friedrichs extension of $\Delta$ restricted to the space $L^{2}_{0}(\C_{a})$ plus the Dirichlet boundary condition at height $a$, where
\begin{displaymath}
	L^{2}_{0}(\C_a):=\left\lbrace f\in L^{2}(\C_{a})\mid f_{0}=0\quad a.e.\right\rbrace.
\end{displaymath}
Here, $f_{0}$ denotes the constant Fourier coefficient of $f_{\mid \Circ\times (a,+\infty)}$ thought as a distribution. Namely, if $\pi(x,y)=y$ is the projection to the second factor, then $f_{0}:=\pi_{\ast}(f)$. We refer to \cite{CdVII} for a detailed study of such operators. In particular, we quote from \emph{loc.~cit.}~that $\Delta_{\ps}$ has pure point spectrum, the eigenvalues have finite multiplicity and satisfy a Weyl type law. This guarantees that the heat operator $e^{-t\Delta_{\ps}}$ is trace class and that the spectral zeta function is defined for $\Real(s)>1$. 

By a Fourier decomposition argument, one easily sees that the orthogonal complement to $L^{2}_{0}(\C_{a})$ in $L^{2}(\C_{a})$ consists of functions which do not depend on the variable $x$. In its turn, this last space can be seen to be isometric to $L^{2}([a,+\infty),y^{-2}dy)$. Therefore, we have
\begin{displaymath}
	L^{2}(\C_{a})=L^{2}_{0}(\C_a)\overset{\perp}{\oplus} L^{2}([a,+\infty),y^{-2}dy).
\end{displaymath}
We denote by $\Delta_{a}$ the restriction of $\Delta_{D}$ to $L^{2}([a,+\infty),y^{-2}dy)$, namely the Friedrichs extension of 
the operator $-y^{2}\frac{d^{2}}{dy^{2}}$, with Dirichlet boundary condition at $y=a$.

\subsubsection{}\label{subsubsec:eigen_problem_cusp} We consider the eigenvalue problem for the pseudo-laplacian $\Delta_{\ps}$. By the preceding description, it amounts to study the differential equation
\begin{displaymath}
	-y^{2}\left(\frac{\pd^{2}}{\pd x^{2}}+\frac{\pd^{2}}{\pd y^{2}}\right)\psi(x,y)=\lambda\psi(x,y),
\end{displaymath}
where $\psi$ is smooth and square integrable on $\Circ\times [a,+\infty)$, with $\psi(x,a)=0$, and with vanishing constant Fourier coefficient
\begin{equation}\label{eq:1}
	\int_{\Circ}\psi(x,y)dx=0,\quad y>a.
\end{equation}
The space of solutions to this problem is spanned by functions with separate variables $\psi(x,y)=g(y)h(x)$. Moreover, one can take $h$ of the form $h_{k}(x)=e^{2\pi i k x}$, for some integer $k$, with the only restriction $k\neq 0$ by \eqref{eq:1}. For $k$ being fixed, we find for $g$ the differential equation
\begin{equation}\label{eq:2}
	-\left(y^{2}\frac{d^{2}}{dy^{2}}+\lambda - y^{2}(2\pi k)^{2}\right)g(y)=0.
\end{equation}
In particular $k$ and $-k$ give raise to the same equation. The change of variables $Z=2\pi |k|y$ transforms \eqref{eq:2} into a modified Bessel differential equation. The solutions to this equation are the well-known Bessel functions of the first and second kind. 
For the definition and standard properties of these Bessel functions, we refer the reader to numerous references in the literature, e.g., \cite{GR07}, which study in depth the many known properties of these special functions.
Taking into account the integrability condition, one sees that for $g$ one can take
\begin{displaymath}
	g(y)=g_{s}(y):=y^{1/2}K_{s-1/2}(2\pi |k|y)
\end{displaymath}
with the modified Bessel function $K_{\nu}$ of the second kind, where, as it is customary, we wrote $\lambda=s(1-s)$. 
The Dirichlet boundary condition implicitly determines the possible values of $s$:
\begin{equation}\label{eq:3}
	K_{s-1/2}(2\pi|k| a)=0.
\end{equation}
By the properties of Bessel functions, this condition imposes that $s=\frac{1}{2}+ir$ for $r$ real and non-zero, so that $\lambda=\frac{1}{4}+r^{2}>\frac{1}{4}$. Moreover, we may restrict to strictly positive $r$, because of the symmetry property $K_{\nu}=K_{-\nu}$. For every integer $k\neq 0$, we will denote by $\lambda_{k,j}$ the eigenvalues obtained in this way. 
\begin{proposition}
 i. The eigenvalues of $\Delta_{\ps}$ are of multiplicity $2$ and $>1/4$.\\
 ii. The zeros of $K_{s-1/2}(2\pi |k|a)$, as a function of $s$, are simple.
\end{proposition}
\begin{proof}
The first property has already been shown. We just observe that with the preceding notations we have $\lambda_{k,j}=\lambda_{-k,j}$, while the corresponding eigenfunctions $g_{s_j}(y)h_{k}(x)$ and $g_{s_j}(y)h_{-k}(x)$ are linearly independent. The second property can be proven analogously to \cite[App.~A]{Saharian}.
\end{proof}

\subsubsection{} We end this subsection by considering the determinant of the preceding laplacians. Observe that the usual zeta function regularization technique does not apply to $\Delta_{D}$ and $\Delta_{a}$, since they have continuous spectrum. For instance, when $a=1$, the function $y^{s}-y^{1-s}$ is clearly an eigenfunction of both operators (even satisfying the Dirichlet boundary condition) of eigenvalue $s(1-s)$. However, in Section \ref{section:evaluation-determinants} we will prove that the zeta regularized determinant $\det\Delta_{\ps}$ is  
well-defined. This amounts to showing suitable asymptotics of the trace $\tr(e^{-t\Delta_{\ps}})$ as $t\to 0$ or, equivalently, the meromorphic continuation of the spectral zeta function and its regularity at the origin. This can be reformulated by saying that M\"uller's formalism of relative determinants \cite{Muller:relative} applies to the couple $(\Delta_{D},\Delta_{a})$. By construction, we have
\begin{displaymath}
	\det(\Delta_{D},\Delta_{a})=\det\Delta_{\ps}.
\end{displaymath}  
Computing $\det\Delta_{\ps}$ explicitly, we thus deduce the following asymptotics. 
\begin{theorem}\label{thm:det-pseudo-laplacian}
For the model cusp $\C_{a}$ of height $a$, we have the equality
\begin{displaymath}
	\log\det(\Delta_{D},\Delta_{a})=-4\pi\zeta(-1)a-\zeta(0)\log(a)+o(1),
\end{displaymath}
as $a\to +\infty$.
\end{theorem}
The proof of Theorem \ref{thm:det-pseudo-laplacian} finds inspiration in the physics literature (see for instance \cite{Barvinsky, Bordag, Flachi, Fucci}). Unfortunately, these references lack of mathematical rigor and present important gaps. After a \emph{tour de force}, we were able to address the various difficulties. For the sake of exposition, we postpone the details until Section \ref{section:evaluation-determinants}.

\subsection{The model cone}\label{subsec:model-cone}
\subsubsection{}\label{subsubsec:model-cone} 
Let $\Circ$ be the unit circle parametrized by the coordinate $x\in [0,1]$ and $a>0$ a positive real number.
Further, let $\omega$ be a positive integer. The hyperbolic cone of angle $\alpha=2\pi/\omega$ and boundary at height $a$ is
the non-compact surface with boundary $\E_{a}:=\Circ\times [a,+\infty)$, equipped with the constant curvature $-1$ metric
\begin{displaymath}
	ds_{\cone}^{2}=\alpha^{2}\frac{dx^{2}+dy^{2}}{\sinh(\alpha y)^{2}}.
\end{displaymath}
In contrast with the cusp case, this metric is not complete. A suitable change of variables provides a parametrization 
of the hyperbolic cone by $\E_{\eta}:=(0,\eta]\times[0,2\pi]$ with coordinates $(\rho,\theta)$. The metric tensor becomes
\begin{displaymath}
	ds_{\cone}^{2}=d\rho^{2}+\omega^{-2}\sinh(\rho)^{2}d\theta^{2}.
\end{displaymath}
Finally, the hyperbolic cone can also be seen as a Riemann surface with boundary, parametrized by the complex coordinate $z\in D^{\times}(0,R)$, such that
\begin{displaymath}
	ds_{\cone}^{2}=\frac{4|dz|^{2}}{\omega^{2}|z|^{2-2/\omega}(1-|z|^{2/\omega})^{2}}.
\end{displaymath}
A coordinate $z$ with this property is unique up to a factor of modulus 1 and called \emph{rs} coordinate after Wolpert \cite{Wolpert:Good, Wolpert:cusp}. The conversion relating the parameters $\eta$ and $R$ can easily be obtained by computing and comparing the volumes. The link between $\eta$ and $R$ is such that, as $\eta\to 0$, there are the following asymptotics
\begin{align}
	&\eta\sim 2 R^{1/\omega},\notag \\
	&\log\eta\sim\log(2R^{1/\omega}).\label{eq:relation-eps-R}
\end{align}
\subsubsection{} As for the model cusp, we will consider the corresponding scalar laplacian $\Delta$ on 
$L^{2}(\E_{\eta}):=L^{2}(\E_{\eta},ds_{\cone}^{2})$, with Dirichlet boundary condition. The expression of $\Delta$ in coordinates $(\rho,\theta)$ is given by
\begin{displaymath}
	\Delta=-\frac{\pd^{2}}{\pd\rho^{2}}\frac{1}{\tanh(\rho)}\frac{\pd}{\pd\rho}-\frac{\omega^{2}}{\sinh^{2}(\rho)}\frac{\pd^{2}}{\pd\theta^{2}},
\end{displaymath}
with domain $\C_{0}^{\infty}((0,\eta)\times [0,2\pi])$ (we actually have to impose further periodicity in $\theta$). Let $\Delta_{D}$ be the Friedrichs extension of $\Delta$.\footnote{Actually, the operator $\Delta_{D}$ is essentially self-adjoint, and hence there is a unique closed self-adjoint extension.} The operator $\Delta_{D}$ has pure point spectrum, the eigenvalues have finite multiplicity and they satisfy a Weyl type law. This can be seen by reduction to the case $\omega=1$: just pull-back $\Delta$ to $D(0,R^{1/\omega})$ under the map $z\to z^{\omega}$. As will follow from Section \ref{subsection:determinant-cone}, the zeta regularized determinant $\det\Delta_{D}$ is defined, although this can also be deduced easily from the desingularisation $z\mapsto z^{\omega}$ (and is even a well-known fact). 

\subsubsection{} For the computation of the determinant $\det\Delta_{D}$, we will need the following lower bound for the eigenvalues.
\begin{lemma}\label{lemma:min-eigen-cone}
The smallest eigenvalue of the operator $\Delta_{D}$ on the hyperbolic cone is strictly bigger than $1/4$.
\end{lemma}
\begin{proof}
We will work in coordinates $(\rho,\theta)$. Let $f$ be an eigenfunction of $\Delta_{D}$ with eigenvalue $\lambda$. First of all, we have
\begin{equation}\label{eq:4}
	\langle f,f\rangle=\int_{[0,\eta]\times[0,2\pi]}f^{2}d\vol< \int_{[0,\eta]\times[0,2\pi]}f^{2}\frac{1}{\omega}\cosh(\rho)d\rho d\theta,
\end{equation}
where we put $d\vol=\frac{1}{\omega}\sinh(\rho)d\rho d\theta$ for the volume form. This inequality is strict because $f$ is not identically 0 and $\sinh(\rho)<\cosh(\rho)$. We treat the last integral in \eqref{eq:4} by integration by parts. For this, we first recall that $f$ satisfies the Dirichlet boundary condition at $\rho=\eta$ and observe it is bounded at $\rho=0$ by the Friedrichs extension condition. Hence
\begin{displaymath}
	f(\eta)=0,\quad\lim_{\rho\to 0}f(\rho)\sinh(\rho)=0.
\end{displaymath}
We thus find
\begin{equation}\label{eq:5}
	\int_{[0,\eta]\times[0,2\pi]}f^{2}\frac{1}{\omega}\cosh(\rho)d\rho d\theta=
	-\int_{[0,\eta]\times[0,2\pi]}2f\frac{\pd f}{\pd\rho}\frac{1}{w}\sinh(\rho)d\rho d\theta.
\end{equation}
We apply the Cauchy--Schwartz inequality to the right hand side of \eqref{eq:5} and combine it with \eqref{eq:4}. We derive
\begin{displaymath}
	\int_{[0,\eta]\times[0,2\pi]}f^{2}d\vol<\left (\int_{[0,\eta]\times[0,2\pi]}4f^{2}d\vol\right)^{1/2}
	\left(\int_{[0,\eta]\times[0,2\pi]}\left(\frac{\pd f}{\pd\rho}\right)^{2}d\vol\right)^{1/2},
\end{displaymath}
and hence
\begin{equation}\label{eq:6}
	\frac{1}{4}\int_{[0,\eta]\times[0,2\pi]}f^{2}d\vol<\int_{[0,\eta]\times[0,2\pi]}\left(\frac{\pd f}{\pd\rho}\right)^{2}d\vol.
\end{equation}
To conclude, recall that $\Delta=d^{\ast}d$ and $\Delta f=\lambda f$, so we find
\begin{equation}\label{eq:7}
	\begin{split}
	\lambda\langle f,f\rangle=\langle\Delta f, f\rangle=\langle df,df\rangle=&\int_{[0,\eta]\times[0,2\pi]}\left[\left(\frac{\pd f}{\pd\rho}\right)^{2}
	+\frac{1}{\omega^{2}}\left(\frac{\pd f}{\pd\theta}\right)^{2}\right]d\vol\\
	&\geq\int_{[0,\eta]\times[0,2\pi]}\left(\frac{\pd f}{\pd\rho}\right)^{2}d\vol.
	\end{split}
\end{equation}
From equations \eqref{eq:6}--\eqref{eq:7}, we finally arrive at $\lambda>1/4$, as was to be shown.
\end{proof}
Let us now consider the eigenvalue problem for $\Delta_{D}$. The spaces of eigenfunctions are spanned by functions with separate variables. Let $\psi(\rho,\theta)=g(\rho)h(\theta)$ be an eigenfunction with eigenvalue $\lambda$ $(>1/4)$. The function $h(\theta)$ can be supposed to be of the form $h_{k}(\theta)=e^{ik\theta}$ with $k\in\Int$. Then, for $k$ being fixed, we find for $g$ the differential equation
\begin{displaymath}
	-\frac{\partial^{2}}{\partial\rho^{2}}g(\rho)
- \frac{1}{\tanh(\rho)}\frac{\partial}{\partial\rho} g(\rho)
-\left(\lambda-\frac{k^2 \omega^{2}}{ \sinh^{2}(\rho)}\right)  g(\rho)=0.
\end{displaymath}
We put $\lambda=s(1-s)$ with $s=\frac{1}{2}+ir$ and $r>0$, and make the change of variables $Z=\cosh(\rho)$. Then, if $g(\rho)=G(Z)$, we arrive at the following associated Legendre differential equation
\begin{displaymath}
(1-Z^2) \frac{\partial^2}{\partial Z^2}G(Z)
-2Z \frac{\partial}{\partial Z}G(Z)
+\left((s-1)s-\frac{ k^2 \omega^2}{ 1-Z^2}\right) G(Z)=0.
\end{displaymath}
The solutions to this equation are the well-known Legendre functions of the first and second kind.
For the definition and standard properties of these Legendre functions, we refer the reader to numerous references in the literatur, e.g., \cite{GR07}.
The Friedrichs condition amounts in this case to require $G$ to be bounded, and the Dirichlet boundary condition is $G(\cosh(\eta))=0$. The only possible solution is then expressed in terms of the associated Legendre function 
$P$ of the first kind, namely
\begin{displaymath}
	G(Z)=P_{\nu}^{-\mu}(\cosh(\rho)),\quad \nu:=-\frac{1}{2}+ir,\quad\mu:=|k|\omega.
\end{displaymath}
The possible values of $\nu$ are implicitly determined by the boundary condition
\begin{displaymath}
	P_{\nu}^{-\mu}(\cosh(\eta))=0.
\end{displaymath}
For every $k\in\Int$, we enumerate the possible values of $\nu$ (resp.~$\lambda$) by $\nu_{k,n}$ (resp.~$\lambda_{k,n}$), $n\geq 0$. Observe that $\nu_{k,n}=\nu_{-k,n}$ (resp.~$\lambda_{k,n}=\lambda_{-k,n}$).
\begin{proposition}\label{prop:eigen-cone}
The spectrum of $\Delta_{D}$ is formed by the $\lambda_{0,n}$ with multiplicity 1, and the $\lambda_{k,n}$, $k>0$, with multiplicity 2.
\end{proposition}
\begin{proof}
By the preceding discussion we just have to justify that the zeros of the function of $\nu$, $P_{\nu}^{-|k|\omega}(\cosh(\eta))$ are simple. This is again proven as in \cite[App.~A]{Saharian}.
\end{proof}
With these results at hand, one can then show the next statement.
\begin{theorem}\label{thm:det_cone}
For the model cone $\E_{\eta}$ of angle $2\pi/\omega$,
we have the equality
\begin{displaymath}
	\begin{split} 
		\log\det\Delta_{D}=&-\left(\frac{\omega}{6}+\frac{1}{6\omega}\right)\log(\eta)
	    -\omega\left(-2\zeta^{\prime}(-1)+\frac{1}{6}-\frac{1}{6}\log 2\right)
		-\frac{1}{\omega}\left(\frac{5}{12}-\frac{1}{6}\log 2+\frac{\gamma}{6}\right)\\
		&+\frac{1}{2}\log\omega+\frac{1}{6}\omega\log\omega+\frac{1}{6}\frac{\log\omega}{\omega}+\frac{1}{4}+o(1),
\end{split}	
\end{displaymath}
as $\eta\to 0$.
\end{theorem}
As in the case for cusps, the proof of Theorem \ref{thm:det_cone} 
is very technical and lengthy. To ease the exposition, we postpone it to Section \ref{subsection:determinant-cone}.

\section{Mayer--Vietoris formulas for determinants of laplacians}\label{section:Mayer--Vietoris}
The strategy depicted in the introduction of the article, ultimately leading to Theorem \ref{theorem:main}, is based on surgery techniques for determinants of laplacians. The idea is to replace neighborhoods of the cusps and elliptic fixed points by flat disks, and keep track of the change of the determinant of the laplacian. The change involves the explicit evaluations of laplacians stated in the previous section. In these ``cut and paste" operations, known as Mayer--Vietoris formulas, the ``glue" is given by the determinant of the so-called Dirichlet-to-Neumann jump operators. In this section we review these surgery techniques, after work of Burghelea--Friedlander--Kappeler \cite{BFK} (compact case) and Carron \cite{Carron} (non-compact  case). The bridge between both is provided by Proposition \ref{prop:DN_invariant}: suitably normalized, the determinants of Dirichlet-to-Neumann jump operators, for both of our compact and non-compact settings, coincide. This observation is the key to deduce a Riemann--Roch isometry in the non-compact orbifold case from the compact case.

\subsection{The case of smooth riemannian metrics}
Let $(X,g)$ be a compact riemannian surface. Let $\Sigma$ be a finite disjoint union of smooth closed curves embedded in $X$. Then $\Sigma$ delimitates compact surfaces with boundary $X_{0},\ldots,X_{n}$ in $X$. The Mayer--Vietoris type formula of Burghelea--Friedlander--Kappeler \cite{BFK} relates the determinant of the scalar laplacian of $(X,g)$ to the determinants of the Dirichlet laplacians of the riemannian surfaces with boundary $(X_{i},g_{\mid {X_i}})$. It involves a contribution from the boundary $\Sigma$, that is the determinant of the Dirichlet-to-Neumann jump operator.

\subsubsection{}\label{subsubsec:Dirichlet-smooth} We will make the following simplifying hypothesis on $\Sigma$ and $X_{0},\ldots,X_{n}$. We assume that the surfaces $X_{1},\ldots, X_{n}$ are isomorphic to closed disjoint disks with boundaries $\Sigma_{1},\ldots,\Sigma_{n}$, and that $X_{0}$ is the complement of  $\overset{\circ}{X}_{1}\cup\ldots\cup \overset{\circ}{X}_{n}$ in $X$. Then, we have $\Sigma=\Sigma_{1}\cup\ldots\cup\Sigma_{n}$. We denote by $\Delta$ the Laplace--Beltrami operator attached to $g$ on $X$, and by $\Delta_{i}$ the restriction of $\Delta$ on $X_{i}$ with Dirichlet boundary condition. This is the situation we will encounter later.

Next we recall the construction of the Dirichlet-to-Neumann jump operator \cite[Sec.~3]{BFK}. Let $f$ be a smooth function on $\Sigma$. We consider the Dirichlet problem
\begin{equation}\label{eq:9}
	\begin{cases} &\Delta_{0}\varphi_{0}=0,\quad\varphi_{0}\in\C^{\infty}(X_{0}),\quad\varphi_{0\mid\Sigma_{i}}=f_{\mid\Sigma_{i}}\text{ for all } i, \\
			&\Delta_{i}\varphi_{i}=0,\quad\varphi_{i}\in\C^{\infty}(X_{i}),\quad\varphi_{i\mid\Sigma_i}=f_{\mid\Sigma_i}\text{ for all }i.
	\end{cases}
\end{equation}
To be rigorous, the formulation should be $\varphi_{i}\in\C^{\infty}(\overset{\circ}{X}_{i})\cap\C^{0}(X_{i})$, but we will allow the abuse of notation $\varphi_{i}\in\C^{\infty}(X_{i})$. The solutions to the Dirichlet problem are known to exist and are unique. We then define a function $Rf$ on $\Sigma$ by:
\begin{equation}\label{eq:10}
	Rf_{\mid\Sigma_{i}}=-\left(\frac{\pd\varphi_{0}}{\pd n_{i}^{+}}\Big |_{\Sigma_{i}}+\frac{\pd\varphi_{i}}{\pd n_{i}^{-}}\Big |_{\Sigma_{i}}\right).
\end{equation}
In this expression, $n_{i}^{+}$ (resp.~$n_{i}^{-}$) is the unitary normal vector field to $\Sigma_{i}$ pointing into the interior of $X_{0}$ (resp.~$X_{i}$). The functions $\varphi_i$ glue into a continuous function $\varphi$ on $X$. Then $Rf$ measures how far $\varphi$ is from being differentiable. Observe this is the case if, and only if, $Rf=0$. But then $\varphi$ is harmonic on the whole compact surface $X$, hence constant. This shows that $0$ will be an eigenvalue of multiplicity 1. 

The operator $R:\C^{\infty}(\Sigma)\to\C^{\infty}(\Sigma)$ is known to be an elliptic pseudo-differential operator of order $1$ \cite[Prop.~3.2]{BFK}. Its zeta regularized determinant can be defined by the theory of Seeley \cite{Seeley}. As usual we write by $\det^{\prime}R$ for the determinant with the zero eigenvalue removed (this is the meaning of the prime symbol). We then have the following Mayer--Vietoris type formula.
\begin{theorem}[Mayer--Vietoris type formula]\label{thm:MV-smooth}
We have the equality of real numbers
\begin{equation}\label{eq:8}
	\frac{\detp\Delta}{\prod_{i=0}^{n}\det\Delta_{i}}=\frac{\vol(X,g)}{\ell(\Sigma,g)}\detp R,
\end{equation}
where $\vol(X,g)$ is the volume of $X$ and $\ell(\Sigma,g)$ is the total length of $\Sigma$, with respect to $g$.
\end{theorem}
\begin{proof}
We refer to \cite[Thm.~B*]{BFK} for the case when $\Sigma$ is connected. The general case is similarly treated (for this recall that $0$ is an eigenvalue of multiplicity 1).
\end{proof}
\subsection{The case of singular hyperbolic metrics}\label{subsec:case_singular_hyp}
We will need a version of the Mayer--Vietoris type formula that applies to compact Riemann surfaces with possibly singular Poincar\'e metrics. These metrics arise from fuchsian uniformizations.
\subsubsection{}\label{subsubsec:discussion-fuchsian} Let $\Gamma\subset\PSL_{2}(\RR)$ be a fuchsian group. The quotient space $\Gamma\backslash\HH$ can be given a structure of Riemann surface. It can be compactified by addition of a finite number of cusps, bijectively corresponding to conjugacy classes of primitive parabolic elements of $\Gamma$. Let $X$ be the compact Riemann surface thus obtained. We denote by $p_{1},\ldots,p_{n}\in X$ the distinct points constituting the cusps and elliptic fixed points (the later are in bijective correspondence to conjugacy classes of primitive elliptic elements of $\Gamma$). For every $p_{i}$, we denote by $m_{i}$ its order or multiplicity, namely $\infty$ for cusps, or the order of the stabilizers in $\Gamma$ for elliptic fixed points. The hyperbolic metric on $\HH$ descends to a singular metric on $X$, that we still call Poincar\'e or hyperbolic metric.\footnote{The terminology \emph{hyperbolic} is somewhat abusive, since the metric is not complete at the elliptic fixed points.} We write $ds_{\hyp}^{2}$ for the riemannian metric tensor. A small neighborhood of a cusp (resp.~elliptic fixed point of order $m_{i}$) is isometric to the model cusp \textsection \ref{subsubsec:model-cusp} (resp.~the model cone of order $m_{i}$ \textsection \ref{subsubsec:model-cone}). We fix closed disjoint disks around the points $p_{i}$, to be denoted $X_{i}$, isometric to either model cusps or model cones of order $m_{i}$. We denote by $X_{0}$ the complement of $\overset{\circ}{X}_{1},\ldots,\overset{\circ}{X}_{n}$. We introduce $\Delta_{\hyp}$ the scalar hyperbolic laplacian acting on $\mathcal{C}_{0}^{\infty}(X\setminus\lbrace p_{i}\rbrace_{i})$. It is known to uniquely extend to an unbounded closed positive self-adjoint operator on $L^{2}(X,ds^{2}_{\hyp})$ (see for instance \cite{Iwaniec}). We define $\Delta_{\hyp, i}$ as the restriction of $\Delta_{\hyp}$ to $X_{i}$, with the Dirichlet boundary condition. These are Friedrichs extensions. Finally, we put $\Delta_{\hyp,\cusp}=\oplus_{m_{i}=\infty}\Delta_{\hyp,i}$ for the orthogonal sum of operators, and let it act trivially on the orthogonal complement of $\oplus_{m_{i}=\infty}L^{2}(X_{i},ds^{2}_{\hyp\mid X_{i}})$ naturally embedded into $L^{2}(X,ds^{2}_{\hyp})$. The following statement is covered by Section \ref{section:evaluation-determinants}, to which the reader is referred. We give it now for a clearer exposition.
\begin{proposition}
i. The zeta regularized relative determinant $\det(\Delta_{\hyp},\Delta_{\hyp,\cusp})$ is well-defined.

ii. For $i=0$ or $m_{i}<\infty$, the zeta regularized determinant $\det(\Delta_{\hyp,i})$ is well-defined.
\end{proposition}

\subsubsection{}\label{subsubsec:Dirichlet-singular} As in Theorem \ref{thm:MV-smooth}, the relation between the determinants $\det(\Delta_{\hyp},\Delta_{\hyp,\cusp})$, $\det(\Delta_{\hyp,i})$ can be expressed in terms of the determinant of a jump operator, that we now introduce. Let $f$ be a smooth function on $\Sigma$. We consider the Dirichlet problem
\begin{displaymath}
	\begin{cases}
	&\Delta_{\hyp,0}\varphi_{0}=0,\quad\varphi_{0}\in\C^{\infty}(X_{0}),\quad\varphi_{0\mid\Sigma_{i}}=f_{\mid\Sigma_{i}}\text{ for all } i, \\
			&\Delta_{\hyp,i}\varphi_{i}=0,\quad\varphi_{i}\in\C^{\infty}(X_{i}\setminus\lbrace p_{i}\rbrace)\cap L^{\infty}(X_{i}),\quad\varphi_{i\mid\Sigma_i}=f_{\mid\Sigma_i}\text{ if } m_{i}<\infty,\\
			&\Delta_{\hyp,i}\varphi_{i}=0,\quad\varphi_{i}\in\C^{\infty}(X_{i}\setminus\lbrace p_{i}\rbrace)\cap L^{\infty}(X_{i}),\quad\varphi_{i\mid\Sigma_i}=f_{\mid\Sigma_i}\text{ if } m_{i}=\infty.
	\end{cases}
\end{displaymath}
The growth conditions imposed to the solutions $\varphi_{i}$ near the points $p_{i}$ correspond to the Friedrichs extension. We made the same abuse of notations as in \eqref{eq:9}. 
\begin{lemma}\label{lemma:harmonic-extension}
The Dirichlet problem \ref{subsubsec:Dirichlet-singular} has a unique solution. Moreover, the solution satisfies:
\begin{displaymath}
	\varphi_{i}\in\C^{\infty}(X_{i})\text{ and }dd^{c}\varphi_{i}=0\text{ for all }i=0,\ldots,n.
\end{displaymath}
\end{lemma}
\begin{proof}
Let $z$ be a local holomorphic coordinate on $X$. Then the Laplace--Beltrami operator in coordinate $z$ is of the form
\begin{displaymath}
	\lambda\frac{\pd^{2}}{\pd z\cpd z},
\end{displaymath}
where $\lambda$ is a non-vanishing smooth function away from the points $p_{i}$. The harmonicity condition $\Delta_{\hyp,i}\varphi_{i}=0$ is therefore equivalent to $dd^{c}\varphi_{i}=0$. Because the Friedrichs condition imposes boundedness of $\varphi_{i}$, by Riemann's extension theorem $\varphi_{i}$ extends to a harmonic function at $p_{i}$, in particular smooth. 

Conversely, any solution to the problem
\begin{displaymath}
	(D): \begin{cases}
		dd^{c}\varphi_{0}=0,\quad\varphi_{0}\in\C^{\infty}(X_{0}),\quad \varphi_{0\mid\Sigma_{i}}=f_{\mid\Sigma_{i}},\\
		dd^{c}\varphi_{i}=0,\quad\varphi_{i}\in\C^{\infty}(X_{i}),\quad\varphi_{i\mid\Sigma_{i}}=f_{\mid\Sigma_{i}}.
	\end{cases}
\end{displaymath}
will clearly be a solution to the Dirichlet problem \ref{subsubsec:Dirichlet-singular}. A solution to $(D)$ is known to exist and is unique. For instance, by the same argument as above, this problem is equivalent to a Dirichlet problem with respect to a Laplace--Beltrami operator for any smooth \emph{conformal} metric on $X$. This concludes the proof.
\end{proof}
Let $\varphi_{i}$, $i=0,\ldots,n$, be the solution provided by Lemma \ref{lemma:harmonic-extension}. Then we define the jump operator $R_{\hyp}f$ by the same equation as in the smooth case \eqref{eq:10}. Observe that 
the normal unit vectors $n_{i}^{+}$ and $n_{i}^{-}$ are well-defined because the hyperbolic metric is smooth at the boundaries $\Sigma_{i}$. The following proposition allows to compare jump operators with respect to the hyperbolic and a suitably chosen smooth metric.
\begin{proposition}\label{prop:DN_invariant}
i. The operator $R_{\hyp}$ is an elliptic pseudo-differential operator of order 1.

ii. Let $h$ be a smooth conformal metric on $X$ and $R_{h}$ its jump operator. Then
\begin{displaymath}
	\frac{\detp R_{\hyp}}{\ell(\Sigma, ds_{\hyp}^{2})}=\frac{\detp R_{h}}{\ell(\Sigma, h)};
\end{displaymath} 
here, $\ell(\Sigma,ds_{\hyp}^{2})$ resp.~$\ell(\Sigma, h)$ denotes the total length of $\Sigma$ with respect to 
$ds_{\hyp}^{2}$ and $h$, respectively.
\end{proposition}
\begin{proof}
First of all, choose a smooth hermitian metric $g$ which coincides with $ds_{\hyp}^{2}$ in a neighborhood of $\Sigma$. Let $R_{g}$ be its jump operator. Then we have $R_{\hyp}=R_{g}$. Indeed, this is consequence of Lemma \ref{lemma:harmonic-extension}, because the metric $g$ is conformal, by uniqueness of (smooth) solutions, and by the comparison of normal vectors with respect to $ds_{\hyp}^{2}$ and $g$: $n_{i,\hyp}^{\pm}=n_{i,g}^{\pm}$ (with self-explanatory notations). This proves the first point, because $R_{g}$ is known to be an elliptic pseudo-differential operator of order 1. For the second assertion, we use that for two conformal smooth metrics such as $g$ and $h$, we have
\begin{displaymath}
	\frac{\detp R_{g}}{\ell(\Sigma,g)}=\frac{\detp R_{h}}{\ell(\Sigma,h)}.
\end{displaymath}
Namely, in the smooth case this quotient is a conformal invariant. By the choice of $g$, we clearly have $\ell(\Sigma,ds_{\hyp}^{2})=\ell(\Sigma,g)$. We conclude because $R_{\hyp}=R_{g}$.
The proof is complete.
\end{proof}
We can finally state the Mayer--Vietoris type formula for the possibly singular hyperbolic metric $ds_{\hyp}^{2}$.
\begin{theorem}\label{thm:MV_hyp}
We have the equality of real numbers
\begin{displaymath}
	\frac{\det(\Delta_{\hyp},\Delta_{\hyp,\cusp})}{\detp\Delta_{\hyp,0}\prod_{m_{i}<\infty}\detp\Delta_{\hyp,i}}=\frac{\vol(X,ds_{\hyp}^{2})}{\ell(\Sigma,ds_{\hyp}^{2})}\detp R_{\hyp},
\end{displaymath}
where $\vol(X,ds_{\hyp}^{2})=2\pi(2g_{X}-2+\sum(1-\frac{1}{m_i}))$ is the volume of $X$ with respect to the hyperbolic metric and $\ell(\Sigma,ds_{\hyp}^{2})$ is the total length of $\Sigma$.
\end{theorem}
\begin{proof}
In the absence of elliptic fixed points, the theorem is a particular application of \cite[Thm.~1.4]{Carron}. The prove of the general case is a straightforward adaptation.
\end{proof}

\section{Bost's $L_{1}^{2}$ formalism and Riemann--Roch isometry}\label{section:DRR}
The statement of the Riemann--Roch isometry we aim to prove, as well as the surgery methods we follow, require an ``arithmetic" intersection formalism well-suited to hermitian line bundles whose metrics are not smooth, but with some mild singularities. Precisely, we seek a formalism that includes both the singularities of Poincar\'e metrics and ``piecewise" smooth metrics. These are particular instances of $L_{1}^{2}$ hermitian metrics. Bost extended the arithmetic intersection theory to this setting, and it turns out to provide exactly what we need. In this section we discuss the theory and combine it with the language of Deligne pairings. We briefly recall the Riemann--Roch isometry as proven by Deligne \cite{Deligne} and Gillet--Soul\'e \cite{GS}. With these preliminaries at hand, we then establish Proposition \ref{prop:naive_laplacian}, that is a first naive version of Theorem \ref{theorem:main}. The result is not satisfactory in that it lacks of a spectral interpretation. This will be remedied later by using the Mayer--Vietoris formulas and the explicit evaluations of determinants of laplacians.

\subsection{Deligne's pairing and Bost's $L_{1}^{2}$ theory}\label{subsection:Deligne-Bost}
In what follows, we work on the Deligne pairing of two line bundles on a smooth proper complex curve \cite[Sec.~1.4, 6]{Deligne} and the theory of $L_{1}^{2}$ equilibrium potentials, Green currents, and their $\ast$-products, developed by Bost \cite{Bost}. For reasons of space, we refer the reader to these references for the main definitions and statements.

\subsubsection{} Let $X$ be a compact Riemann surface and $L$ a line bundle on $X$.
\begin{definition}
A $L_{1}^{2}$ hermitian metric on $L$ consists in giving, for every local holomorphic trivializing section $e\in\Gamma(U,L)$, a $L_{1}^{2}(U)_{\loc}$ function	$g=g_{(U,e)}$ satisfying the compatibility property
\begin{displaymath}
	g_{(V,fe)}=-\log|f|^{2}+g_{(U,e)\mid V}
\end{displaymath}
whenever $V$ is open in $U$ and $f$ is a non-vanishing holomorphic function on $V$. As is customary, we write $g=-\log\|\cdot\|^{2}$, so that $\|\cdot\|$ satisfies the usual rules of a hermitian norm (where defined). We say that $\overline{L}:=(L,\|\cdot\|)$ is a $L_{1}^{2}$ hermitian line bundle.
\end{definition}
\begin{remark}
\emph{i}. Fix a smooth hermitian metric on $\|\cdot\|_{0}$ on $L$. Then the set of $L_{1}^{2}$ hermitian metrics on $L$ is in bijective correspondence with $L_{1}^{2}(X)$. Indeed, by definition any $L_{1}^{2}$ metric can be uniquely written $e^{-f/2}\|\cdot\|_{0}$, where $f\in L_{1}^{2}(X)$.

\emph{ii}. Let $\ov{L}$ be a $L_{1}^{2}$ hermitian line bundle on $X$ and $s$ a rational section of $L$. Then the function $\log\|s\|^{-2}$ is locally integrable on $X$ and defines a current denoted $[\log\|s\|^{-2}]$.
\end{remark}
\begin{notation}
Let $\overline{L}$ be a $L_{1}^{2}$ hermitian line bundle on $X$. Then $c_{1}(\overline{L})$ denotes the first Chern current of $\ov{L}$ on $X$. That is, if $s$ is a meromorphic section of $L$, not identically vanishing on any connected component of $X$, then $c_{1}(\overline{L})$ is the current
\begin{displaymath}
	c_{1}(\overline{L}):=dd^{c}[\log\|s\|^{-2}]+\delta_{\gdiv s}.
\end{displaymath}
By the Poincar\'e--Lelong formula, the current $c_{1}(\overline{L})$ does not depend on the particular choice of section $s$.
\end{notation}
Let $\ov{L}_{0}$ and $\ov{L}_{1}$ be $L_{1}^{2}$ hermitian line bundles on $X$. Let $l_{0},l_{1}$ be rational sections of $L_0$ and $L_1$, respectively, with disjoint divisors. Then there is a corresponding trivialization of the Deligne pairing $\langle L_{0},L_{1}\rangle$ given by a symbol $\langle l_{0},l_{1}\rangle$. The functions $g_{0}=\log\|\l_{0}\|_{0}^{-2}$ and $g_{1}=\log\|l_{1}\|_{1}^{-2}$ are $L_{1}^{2}$ Green functions for $\gdiv l_{0}$ and $\gdiv l_{1}$ \cite[Sec.~3.1.3]{Bost}. By Bost's theory, one can define the integral of the $\ast$-product \cite[Sec.~5.1]{Bost}: $\int_{X}g_{0}\ast g_{1}$.
\begin{definition}\label{definition:Deligne-Bost}
We define the Deligne--Bost norm on $\langle L_{0},L_{1}\rangle$ by the rule
\begin{displaymath}
	\log\|\langle l_{0},l_{1}\rangle\|^{-2}=\int_{X} g_{0}\ast g_{1}.
\end{displaymath}
\end{definition}
It is easily seen that the previous rule is compatible with the relations defining the Deligne pairing, and therefore the construction makes sense. We record in the following lemma several useful properties that will simplify the computation of Deligne--Bost norms.

\begin{lemma}\label{lemma:computation_star-products}
i. Consider sequences $\|\cdot\|_{0}^{n}$ and $\|\cdot\|_{1}^{n}$ of $L_{1}^{2}$ metrics on $L_{0}$, $L_{1}$, respectively, such that the functions $\log(\|\cdot\|_{0}^{n}/\|\cdot\|_{0})$ and $\log(\|\cdot\|_{1}^{n}/\|\cdot\|_{1})$ converge to $0$ in $L_{1}^{2}$ norm. For the norms $\|\cdot\|_{n}$ on $\langle \overline{L}_{0}^{n},\overline{L}_{1}^{n}\rangle$, we have
\begin{displaymath}
	\|\langle l_{0},l_{1}\rangle\|_{n}\to\|\langle l_{0}, l_{1}\rangle\|,\,\,\, n\to +\infty.
\end{displaymath}

ii. Suppose that $c_{1}(\overline{L}_{0})$ is locally $L^{\infty}$ on some open subset $\Omega$ of $X$ and $\gdiv l_{1}$ is contained in $\Omega$. Then $\log\| l_{0}\|_{0}^{-2}$ is continuous on a neighborhood of $\gdiv l_{1}$, $\log\| l_{0}\|_{0}^{-2}\delta_{\gdiv l_{1}}$ and $\log\|l_{1}\|_{1}^{-2}c_{1}(\overline{L}_{0})$ are well-defined currents and we have
\begin{displaymath}
	\log\|\langle l_{0},l_{1}\rangle\|^{-2}=\int_{X}(\log\| l_{0}\|_{0}^{-2}\delta_{\gdiv l_{1}}
	+\log\| l_{1}\|_{1}^{-2}c_{1}(\overline{L}_{0})).
\end{displaymath}
In particular, if $\int_{X}\log\|l_{1}\|^{-2}c_{1}(\overline{L}_{0})=0$, then
\begin{displaymath}
	\begin{split}
		\log\|\langle l_{0}, l_{1}\rangle\|^{-2}&=\int_{X}\log\| l_{0}\|_{0}^{-2}\delta_{\gdiv l_{1}}\\
		&=\sum_{p\in X}(\ord_{p} l_{1})\log\| l_{0}\|^{-2}(p).
	\end{split}
\end{displaymath}
\end{lemma}
\begin{proof}
We refer to \cite[Sec.~5.1]{Bost} for the first property. The second point follows from \cite[Lemma 5.2]{Bost}.
\end{proof}
\begin{definition}
Under the assumptions of Lemma \ref{lemma:computation_star-products} \textit{i}, we say that $\|\cdot\|_{0}^{n}$, $\|\cdot\|_{1}^{n}$ converge to $\|\cdot\|_{0}$, $\|\cdot\|_{1}$ in $L_{1}^{2}$, respectively, and we indicate the convergence of Deligne--Bost norms by
\begin{displaymath}
	\langle \overline{L}_{0}^{n},\overline{L}_{1}^{n}\rangle\to\langle\overline{L}_{0},\overline{L}_{1}\rangle\,\,\,\text{as}\,\,\,
	n\to +\infty.
\end{displaymath}
\end{definition}
\subsubsection{}\label{subsubsec:truncated} We now discuss several examples of $L_{1}^{2}$ hermitian line bundles and Deligne pairings needed for later developments.
\begin{example}[Truncated trivial metrics]\label{example:truncated_trivial}
Let $P$ be a point in $X$. Consider the line bundle $\OO(-P)$ over $X$. Fix $\gamma=(U,z)$ an analytic chart in a neighborhood of $P$, with $z(P)=0$. Suppose that the closed disk $\overline{D}(0,\varepsilon)$ is contained in the image of $z$. Observe that $e:=z$ is a trivialization of $\OO(-P)\mid_{U}$. We define a metric $\|\cdot\|_{\gamma,\varepsilon}$ on $\OO(-P)\mid_{U}$ by
\begin{displaymath}
	\|e\|_{\gamma,\varepsilon,x}=\begin{cases}
		|z(x)|, &\text{if } |z(x)|\geq\varepsilon,\\
		\varepsilon, &\text{if } |z(x)|\leq\varepsilon.
	\end{cases}
\end{displaymath}
The metric $\|\cdot\|_{\gamma,\varepsilon}$ on $U\setminus z^{-1}(\overline{D}(0,\varepsilon))$ coincides with the trivial metric on $\OO(-P)$ induced from the inclusion $\OO(-P)\hookrightarrow\OO_{X}$ and the trivial metric on $\OO_{X}$ (\textit{i.e.}~the absolute value). Therefore we can extend $\|\cdot\|_{\gamma,\varepsilon}$ to all of $X$ by letting it coincide with the trivial metric on $X\setminus z^{-1}(\overline{D}(0,\varepsilon))$. We employ the same notation $\|\cdot\|_{\gamma,\varepsilon}$ for the resulting metric. We will skip the reference to $\gamma$ and use the same notation for the dual metric on $\OO(P)$. Also we put $\OO(P)_{\varepsilon}$ for the couple $(\OO(P),\|\cdot\|_{\varepsilon})$.
\end{example}
A truncated trivial metric $\|\cdot\|_{\varepsilon}$ on $\OO(P)$ is easily seen to be $L_{1}^{2}$. It is routine to check the following identity for the first Chern current:
\begin{equation}\label{eq:12}
	c_{1}(\OO(P)_{\varepsilon})=\frac{1}{2\pi}\delta_{\pd z^{-1}(D(0,\varepsilon))},
\end{equation}
namely the current of integration along the positively oriented boundary of $z^{-1}(\ov{D}(0,\varepsilon))$.
\begin{lemma}\label{lemma:equilibrium-point}
Let $\ccar_{P}:\OO_{X}\hookrightarrow\OO(P)$ be the canonical section with constant value $1$ and $\gdiv\ccar_{P}=P$. 

i. The function $\log\|\ccar_{P}\|^{-2}_{\varepsilon}$ is the equilibrium potential of the non-polar compact subset $X\setminus z^{-1}(D(0,\varepsilon))$ at $P$.

ii. The following vanishing property holds:
\begin{displaymath}
	\int_{X}\log\|\ccar_{P}\|_{\varepsilon}^{-2}c_{1}(\OO(P)_{\varepsilon})=0.
\end{displaymath}
\end{lemma}
\begin{proof}
It is readily checked that $\log\|\ccar_{P}\|_{\varepsilon}^{-2}$ fulfills the characterization of equilibrium potentials given in \cite[Thm.~3.1]{Bost}. For instance, equation \eqref{eq:12} shows that the current $dd^{c}[\log\|\ccar_{P}\|_{\varepsilon}^{-2}]+\delta_{P}$ is a probability measure supported on $\pd z^{-1}(D(0,\varepsilon))$. The vanishing of the integral is then obtained as for \cite[eq.~(5.12)]{Bost}.
\end{proof}

\begin{proposition}\label{prop:appearance_wolpert_metric_1}
Let $P_{0}$ and $P_{1}$ be distinct points and consider truncated hermitian line bundles $\ov{\OO(P_{0})}_{\varepsilon}$ and $\ov{\OO(P_{1})}_{\varepsilon}$, attached to disjoint coordinate neighborhoods $(U_0, z_0), (U_1, z_1)$, respectively, and small $\varepsilon$. Then, restriction to $P_{0}$ and residue at $P_{0}$ induce canonical isometries
\begin{align*}
	&\langle\ov{\OO(P_{0})}_{\varepsilon},\ov{\OO(P_{1})}_{\varepsilon}\rangle\simeq (\CC,|\cdot|),\\
	&\langle\ov{\OO(P_{0})}_{\varepsilon},\ov{\OO(P_{0})}_{\varepsilon}\rangle\simeq (\omega_{X,P_{0}}, \varepsilon|\cdot|_{P_{0}})^{-1},
\end{align*}
where $(\CC,|\cdot|)$ is the trivial hermitian line bundle (on a point) and $|\cdot|_{P_{0}}$ is defined by the rule
\begin{displaymath}
	\left| dz_{0}\right|_{P_{0}}=1.
\end{displaymath}
\end{proposition}
\begin{proof}
The proof follows easily from Lemma  \ref{lemma:computation_star-products} \emph{ii} and Lemma \ref{lemma:equilibrium-point}.
\end{proof}

\begin{example}[Truncated hyperbolic metrics]\label{example:truncated_hyperbolic}
Let $X$ be a cusp compactification of a Riemann surface $\Gamma\backslash\HH$, for some fuchsian group $\Gamma\subset\PSL_{2}(\RR)$. We follow the notations in \textsection\ref{subsubsec:discussion-fuchsian}. Let $\|\cdot\|_{\hyp}$ be the \emph{hermitian norm} on $T_{X}$ associated to $ds_{\hyp}^{2}$. Recall from \textsection\ref{subsubsec:normal-metric} that the hermitian metric tensor is $1/2$ the riemannian one. For every $p_{i}$ we choose a \emph{rs} coordinate $\zeta_{i}$, containing $\ov{D}(0,\varepsilon)$ in its image ($\varepsilon>0$ small enough).  We suppose that the neighborhoods of the $p_i$ so defined are disjoint. We construct a norm $\|\cdot\|_{\hyp,\varepsilon}$ on $T_{X}$ by the following assignments:
\begin{itemize}
	\item if $p\in X\setminus (\cup_{i}\zeta_{i}^{-1}(D(0,\varepsilon)))$, then for every $e\in T_{X,p}$ we put
	\begin{displaymath}
		\|e\|_{\hyp,\varepsilon,p}=\|e\|_{\hyp,p}.
	\end{displaymath}
	\item if $m_{i}=\infty$ and $p\in\zeta_{i}^{-1}(\ov{D}(0,\varepsilon))$, then we declare
	\begin{displaymath}
		\left\|\frac{\pd}{\pd\zeta_{i}}\right\|_{\hyp,\varepsilon,p}^{2}=\begin{cases}
			\left\|\frac{\pd}{\pd\zeta_{i}}\right\|_{\hyp,p}^{2} &\text{if } |\zeta_{i}(p)|\geq\varepsilon,\\
			\frac{1}{2}\frac{1}{\varepsilon^{2}(\log\varepsilon^{-1})^{2}} &\text{if }|\zeta_{i}(p)|\leq\varepsilon.
		\end{cases}
	\end{displaymath}
	\item if $m_{i}<\infty$ and $p\in\zeta_{i}^{-1}(\ov{D}(0,\varepsilon))$, then we put
	\begin{displaymath}
		\left\|\frac{\pd}{\pd\zeta_{i}}\right\|_{\hyp,\varepsilon,p}^{2}=\begin{cases}
			\left\|\frac{\pd}{\pd\zeta_{i}}\right\|_{\hyp,p}^{2} &\text{if } |\zeta_{i}(p)|\geq\varepsilon,\\
			\frac{2}{m_{i}^{2}\varepsilon^{2-2/m_{i}}(1-\varepsilon^{2/m_{i}})^{2}} &\text{if }|\zeta_{i}(p)|\leq\varepsilon.
		\end{cases}
	\end{displaymath}
\end{itemize}
\end{example}
A truncated hyperbolic metric is an example of $L_1^2$ metric on the tangent bundle $T_{X}$. We will employ the same notation $\|\cdot\|_{\hyp,\varepsilon}$ for the dual metric on $\omega_{X}$, since no confusion can arise. Also we may write $\ov{\omega}_{X,\hyp,\varepsilon}$ to refer to the hermitian line bundle $(\omega_{X},\|\cdot\|_{\hyp,\varepsilon})$. For the first Chern current one can check the identity
\begin{displaymath}
	\begin{split}
		c_{1}(\ov{\omega}_{X,\hyp,\varepsilon})=&c_{1}(\omega_{X\setminus\lbrace p_{1},\ldots, p_{n}\rbrace},\|\cdot\|_{\hyp})\chi_{X_{\varepsilon}}\\
		&+\frac{1}{2\pi}\left(1-\frac{1}{\log\varepsilon^{-1}}\right)\sum_{m_{i}=\infty}\delta_{\pd D_{i,\varepsilon}}\\
		&\;\;\;+\frac{1}{2\pi}\sum_{m_{i}<\infty}\left(1-\frac{1}{m_{i}}-\frac{2}{m_{i}}\frac{\varepsilon^{2/m_{i}}}{1-\varepsilon^{2/m_{i}}}\right)
		\delta_{\pd D_{i,\varepsilon}}
	\end{split}
\end{displaymath}
where we wrote $X_{\varepsilon}=X\setminus (\cup_{i} \ov{D}_{i,\varepsilon})$, $D_{i,\varepsilon}=\zeta_{i}^{-1}(D(0,\varepsilon))$ and $\chi_{X_{\varepsilon}}$ for the characteristic function of $X_{\varepsilon}$.
\begin{lemma}\label{lemma:vanishing-integral}
Let $\ov{\OO(p_{i})}_{\varepsilon}$ be endowed with the truncated trivial metric in the rs coordinate $\zeta_{i}$, and $\ccar_{i}:\OO_{X}\hookrightarrow\OO(p_{i})$ the trivial section with constant value $1$. Then we have
\begin{displaymath}
	\int_{X}\log\|\ccar_{i}\|_{\varepsilon}^{-2}c_{1}(\ov{\omega}_{X,\hyp,\varepsilon})=0.
\end{displaymath}
\end{lemma}
\begin{proof}
With the previous notations, the Chern current $c_{1}(\omega_{X,\hyp,\varepsilon})$ is supported on $X\setminus(\cup_{i}D_{i,\varepsilon})$. Moreover, by Lemma \ref{lemma:equilibrium-point} we know that $\log\|\ccar_{i}\|_{\varepsilon}^{-2}$ is the equilibrium potential of $X\setminus D_{i,\varepsilon}$ at $p_{i}$. Again, one concludes as in \cite[eq.~(5.12)]{Bost}.
\end{proof}
\begin{proposition}\label{prop:appearance_wolpert_metric_2}
Restriction at $p_{i}$ induces the following isometries:
\begin{align*}
	&\langle \omega_{X,\hyp,\varepsilon},\ov{\OO(p_{i})}_{\varepsilon}\rangle\simeq( \omega_{X,p_{i}},2^{1/2}(\varepsilon\log\varepsilon^{-1})\|\cdot\|_{W,p_{i}}),\quad\text{if } m_{i}=\infty,\\
	&\langle \omega_{X,\hyp,\varepsilon},\ov{\OO(p_{i})}_{\varepsilon}\rangle\simeq (\omega_{X,p_{i}},2^{-1/2}m_{i}\varepsilon^{1-1/m_{i}}(1-\varepsilon^{2/m_{i}})\|\cdot\|_{W,p_{i}}),\quad\text{if } m_{i}<\infty,
\end{align*}
where $\|\cdot\|_{W_{p_i}}$ is defined by $\|d\zeta_{i}\|_{W,p_{i}}=1$.
\end{proposition}
\begin{proof}
The proof is a straightforward consequence of Lemma \ref{lemma:computation_star-products} \emph{ii} and Lemma \ref{lemma:vanishing-integral}.
\end{proof}
\begin{definition}[Wolpert metric]\label{definition:Wolpert}
Let $z$ be a \emph{rs} coordinate on a hyperbolic cone or cusp $p\in X$. The Wolpert metric on the cotangent space at $p$, $\omega_{X,p}$, is defined by
\begin{displaymath}
	\|dz\|_{W,p}=1.
\end{displaymath}
\end{definition}
\begin{remark}
Observe the definition is legitimate, because \emph{rs} coordinates are unique up to a constant of modulus 1.
\end{remark}

\subsubsection{} We close the discussion on Deligne pairings by studying to what extend the truncated hyperbolic metrics are good approximations of the hyperbolic metric.

Let $X$ be a cusp compactification of some $\Gamma\backslash\HH$ for a fuchsian group $\Gamma\subset\PSL_{2}(\RR)$. Following the notations settled in \textsection\ref{subsubsec:discussion-fuchsian}, we define
\begin{displaymath}
	m:=\prod_{m_{i}<\infty}m_{i},
\end{displaymath}
with the convention that $m=1$ whenever $m_{i}=\infty$ for all $i$, and the $\QQ$-divisor
\begin{displaymath}
	D=\sum_{i}\left(1-\frac{1}{m_{i}}\right)p_{i}.
\end{displaymath}
The $m$-th tensor power of the hyperbolic metric induces a singular hermitian metric on the line bundle
\begin{displaymath}
	\omega_{X}^{m}(mD)=\left(\omega_{X}(D)\right)^{m}.
\end{displaymath}
We still use the index $\hyp$ to refer to this metric. Given $\varepsilon>0$ small enough, we have a truncated hyperbolic metric $\|\cdot\|_{\hyp,\varepsilon}$ on $\omega_{X}$ and truncated trivial metrics on the sheaves $\OO(p_{i})$. Recall these last metrics are constructed after introduction of \emph{rs} coordinates. We indicate with an index $\varepsilon$ the tensor product of these truncated metrics on $\omega_{X}^{m}(mD)$.
\begin{proposition}\label{prop:convergence_deligne_hyp}
i. The hermitian metric $\|\cdot\|_{\hyp}$ on $\omega_{X}^{m}(mD)$ is $L_{1}^{2}$.

ii. The sequence of metrics $\|\cdot\|_{\varepsilon}$ on $\omega_{X}^{m}(mD)$ converges to $\|\cdot\|_{\hyp}$ in $L_{1}^{2}$ as $\varepsilon\to 0$.

iii. We have the convergence of hermitian line bundles
\begin{displaymath}
	\langle\omega_{X}^{m}(mD)_{\varepsilon},\omega_{X}^{m}(mD)_{\varepsilon}\rangle
	\overset{\varepsilon\to 0}{\longrightarrow}
	\langle\omega_{X}^{m}(mD)_{\hyp},\omega_{X}^{m}(mD)_{\hyp}\rangle.
\end{displaymath}
\end{proposition}
\begin{proof}
The first property has already been observed. The second property is an easy computation left as an exercise. The last item is an application of Lemma \ref{lemma:computation_star-products} \emph{i}.
\end{proof}

\subsection{Riemann--Roch isometry and truncated metrics}\label{subsection:Deligne-isom}
The Riemann--Roch isometry refines the arithmetic Riemann--Roch theorem of Gillet--Soul\'e in the case of arithmetic surfaces \cite{Deligne, GS, Soule}. It is stated as an isometry between hermitian line bundles, whose arithmetic Chern classes represent both sides of the arithmetic Grothendieck--Riemann--Roch theorem. The isometry is canonical in that it commutes with base change, and is unique, up to sign, with this property (at least for the trivial line bundle; for more general line bundles, one needs to impose extra functorialities). We next recall the statement in the situation that concerns us, and derive consequences for convergent families of $L_{1}^{2}$ metrics.

\subsubsection{}\label{subsec:prelim-L2} Let $X$ be a compact Riemann surface, equivalently seen as a connected smooth complex projective curve. We suppose that the tangent bundle $T_{X}$ comes equipped with a smooth hermitian metric. Its dual bundle $\omega_{X}$ will then be endowed with the dual metric. The determinant of the cohomology of the trivial sheaf, namely
\begin{displaymath}
	\lambda(\OO_{X})=\det H^{0}(X,\OO_{X})\otimes\det H^{1}(X,\OO_{X})^{-1}
\end{displaymath}	
carries the Quillen metric. It is build up from the $L^{2}$ metric on the cohomology spaces, given by Hodge theory up to normalization, and the regularized determinant of the $\cpd$ laplacian acting on smooth complex functions on $X$.\footnote{On a Riemann surface, the definition of the Quillen metric actually involves $\det^{\prime}\Delta^{0,1}_{\cpd}$. However, since we are in complex dimension 1, we have $\det^{\prime}\Delta^{0,1}_{\cpd}=\det^{\prime}\Delta^{0,0}_{\cpd}$.} Let us review these elements. The $L^{2}$ metric on $H^{0}(X,\OO_{X})=\CC$ is expressed in terms of the volume:
\begin{displaymath}
	\|1\|_{L^{2}}^{2}=\int_{X}\omega,
\end{displaymath}
where $\omega$ is the normalized K\"ahler form, locally given in terms of trivializing sections $\theta$ of $\omega_{X}$ by
\begin{displaymath}
	\omega=\frac{i}{2\pi}\frac{\theta\wedge\ov{\theta}}{\|\theta\|^{2}}.
\end{displaymath}
The $L^{2}$ metric on $H^{1}(X,\OO_{X})$ is such that the canonical Serre duality isomorphism
\begin{displaymath}
	H^{1}(X,\OO_{X})\overset{\sim}{\longrightarrow} H^{0}(X,\omega_{X})^{\vee}
\end{displaymath}
becomes an isometry, where the hermitian structure on $H^{0}(X,\omega_{X})$ is the usual pairing
\begin{displaymath}
	\langle\alpha,\beta\rangle_{L^{2}}=\frac{i}{2\pi}\int_{X}\alpha\wedge\ov{\beta}.
\end{displaymath}
The hermitian metric deduced on $\lambda(\OO_{X})$ by dual, tensor product and determinant operations, will be denoted $\|\cdot\|_{L^{2}}$. The $\cpd$ laplacian on functions is the differential operator $\Delta_{\cpd}=\Delta_{\cpd}^{0,0}=\cpd^{\ast}\cpd$
\begin{displaymath}
	\mathcal{C}^{\infty}(X,\OO_{X})\overset{\cpd}{\rightarrow}\mathcal{C}^{\infty}(X,\Omega_{X}^{0,1})
	\overset{\cpd^{\ast}}{\rightarrow}\mathcal{C}^{\infty}(X,\OO_{X}),
\end{displaymath}
where $\Omega_{X}^{0,1}$ is the sheaf of smooth differential forms of type $(0,1)$. We recall that $\cpd^{\ast}$ is the formal adjoint of $\cpd$ with respect to the $L^{2}$ hermitian structures on these functionals spaces. The operator $\Delta_{\cpd}$ can be related to the Laplace--Beltrami operator $\Delta_{d}$ by the K\"ahler identity
\begin{displaymath}
	\Delta_{\cpd}=\frac{1}{2}\Delta_{d}.
\end{displaymath}
The zeta regularized determinant $\detp\Delta_{\cpd}$ can be defined and we have the identity
\begin{equation}\label{eq:13}
	\detp\Delta_{\cpd}=2^{(g+2)/3}\detp\Delta_{d},
\end{equation}
where $g$ is the genus of $X$. This relation will be useful later, when we combine Riemann--Roch isometry with 
Mayer--Vietoris type formulas, that are stated in terms of Laplace--Beltrami operators. The Quillen metric on $\lambda(\OO_{X})$ is a normalization of the $L^{2}$ metric given by
\begin{displaymath}
	\|\cdot\|_{Q}=(\detp\Delta_{\cpd})^{-1/2}\|\cdot\|_{L^{2}}.
\end{displaymath}

Let us introduce the hermitian complex line $\OO(C(g))$, whose underlying complex line is $\CC$, and whose metric is the renormalization of the absolute value $C(g)|\cdot|$, where
\begin{displaymath}
	C(g)=\exp\left((2g-2)\left(\frac{\zeta^{\prime}(-1)}{\zeta(-1)}+\frac{1}{2}\right)\right).
\end{displaymath}
With these notations, Deligne's Riemann--Roch isometry provides a canonical isometry of hermitian line bundles
\begin{equation}\label{eq:14}
	(\lambda(\OO_{X}),\|\cdot\|_{Q})^{\otimes 12}\overset{\sim}{\longrightarrow}\langle\ov{\omega}_{X},\ov{\omega}_{X}\rangle\otimes\OO(C(g)).
\end{equation}
Up to an unknown topological constant, the theorem is stated in this form by Deligne \cite{Deligne}. To pin down the exact value of the constant, one can appeal to the theorem of Gillet--Soul\'e \cite{GS}. An explanation for that is also exposed in \cite{Soule}.

\subsubsection{} Let $\|\cdot\|$ be a $L_{1}^{2}$ hermitian metric on $\omega_{X}$. For instance, we may think of the truncated hyperbolic metrics. We are interested in giving a sense to the Quillen metric in this case, and if possible to the $L^{2}$ metric and the regularized determinant of the laplacian, for which a Riemann--Roch type isometry is true. For our later purposes, we don't even need to give a meaning to the laplacian and study its spectral resolution. It will be enough the following indirect approach.
\begin{lemma}\label{lemma:convergence_quillen}
Let $\|\cdot\|$ be a $L_{1}^{2}$ hermitian metric and $\|\cdot\|_{n}$ a sequence of smooth hermitian metrics, on $\omega_{X}$, converging to $\|\cdot\|$ in $L_{1}^{2}$. Let us denote by $\|\cdot\|_{Q,n}$, $\|\cdot\|_{L^{2},n}$ and $\detp\Delta_{\cpd,n}$ the corresponding Quillen metric, $L^{2}$ metric and determinant of laplacian.

i. The sequence $\|\cdot\|_{Q,n}$ converges to a metric on $\lambda(\OO_{X})$. The limit only depends on $\|\cdot\|$.

ii. If $\|\cdot\|$ is bounded and $\|\cdot\|_{n}$ is uniformly bounded, then $\|\cdot\|_{L^{2},n}$ and $\detp\Delta_{\cpd,n}$ converge, and the limits only depend on $\|\cdot\|$.
\end{lemma} 
\begin{proof}
We fix a reference smooth hermitian metric on $\omega_{X}$, $\|\cdot\|_{0}$. We write $\|\cdot\|=e^{-f/2}\|\cdot\|_{0}$ and $\|\cdot\|_{n}=e^{-f_{n}/2}\|\cdot\|_{n}$. The convergence hypothesis is that $f_{n}$ converges to $f$ in $L_{1}^{2}$.

The first assertion is an immediate consequence of Lemma \ref{lemma:computation_star-products} together with the Riemann--Roch isometry \eqref{eq:14}. For the second assertion, it is enough to show the convergence of the $L^{2}$ metric, with limit only depending on $\|\cdot\|$. This is elementary and we leave the details to the reader.\end{proof}

\begin{remark}
Given a $L_{1}^{2}$ metric on a hermitian line bundle, there always exist sequences of smooth metrics converging to it in $L_{1}^{2}$. If moreover the limit $L_{1}^{2}$ metric is continuous, a convolution argument shows the existence of a uniformly converging sequence of smooth metrics, converging in $L_{1}^{2}$ as well. In this last case, the $L^{2}$ metric is given by the same rule as in the smooth situation.
\end{remark}

\begin{notation}\label{notation:truncated_quillen}
With the hypothesis of the lemma, we write $\|\cdot\|_{Q}$, $\|\cdot\|_{L^{2}}$ and $\detp\Delta_{\cpd}$ for the respective limits. If needed, we indicate the metric in subindex position, to avoid confusions. Observe that $\detp\Delta_{\cpd}$ is a strictly positive real number.
\end{notation}

\subsubsection{} Let again $X$ be a cusp uniformization of a quotient $\Gamma\backslash\HH$, for some fuchsian group $\Gamma\subset\PSL_{2}(\RR)$. Denote its cusps and elliptic fixed points by $p_{1},\ldots,p_{n}$, with multiplicities $m_{1},\ldots,m_{n}\leq\infty$. For a small $\varepsilon>0$, we defined in Example \ref{example:truncated_hyperbolic} a truncated hyperbolic metric on $\omega_{X}$, $\|\cdot\|_{\hyp,\varepsilon}$. This is an example of $L_{1}^{2}$ and continuous hermitian metric, for which the conclusions of Lemma \ref{lemma:convergence_quillen} hold. Therefore, we have well-defined Quillen and $L^{2}$ metrics, as well as a determinant of laplacian. The notations will be $\|\cdot\|_{Q,\varepsilon}$, $\|\cdot\|_{L^{2},\varepsilon}$ and $\detp\Delta_{\cpd,\varepsilon}$. With these notations, we automatically have a Riemann--Roch type isometry
\begin{equation}\label{eq:16}
	(\lambda(\OO_{X}),\|\cdot\|_{Q,\varepsilon})^{\otimes 12}\overset{\sim}{\longrightarrow}\langle\ov{\omega}_{X,\hyp,\varepsilon},
	\ov{\omega}_{X,\hyp,\varepsilon}\rangle\otimes\OO(C(g)).
\end{equation}
We wish to let the parameter $\varepsilon$ tend to 0, and derive the behavior of $\detp\Delta_{\cpd,\varepsilon}$. For this we take into account propositions \ref{prop:appearance_wolpert_metric_1}--\ref{prop:appearance_wolpert_metric_2}. We find a canonical isomorphism
\begin{equation}\label{eq:15}
	\begin{split}
	\langle\omega_{X}^{m}(mD)_{\varepsilon},\omega_{X}^{m}(mD)_{\varepsilon}\rangle
	\overset{\sim}{\longrightarrow}&\langle\omega_{X,\varepsilon}^{m},\omega_{X,\varepsilon}^{m}\rangle\\
	&\otimes\bigotimes_{i}(\omega_{X,p_{i}},\|\cdot\|_{W,p_{i}})^{\otimes m^{2}( 1-m_{i}^{-2})}\\
	&\hspace{0.3cm}\otimes (\CC,m^{2}R(\varepsilon)|\cdot|),
	\end{split}
\end{equation}
where we used the notation $D=\sum_{i}(1-m_{i}^{-1})p_{i}$ and we put
\begin{displaymath}
	\begin{split}
	R(\varepsilon):=&2^{c-\sum_{m_{i}<\infty}(1-m_{i}^{-1})}\prod_{m_{i}<\infty}m_{i}^{2m^{2}(1-m_{i}^{-1})}\\
	&\varepsilon^{\sum_{i}(1-m_{i}^{-1})^{2}}(\log\varepsilon^{-1})^{2c}\prod_{m_{i}<\infty}(1-\varepsilon^{2/m_{i}})^{2(1-m_{i}^{-1})}.
	\end{split}
\end{displaymath}
Here $c$ denotes the number of cusps $c:=\sharp\lbrace m_{i}=\infty\rbrace$.
\begin{proposition}\label{prop:naive_laplacian}
Let the notations be as above, so that $\det\Delta^{\prime}_{\cpd,\hyp,\varepsilon}$ is the limit laplacian attached to the $L_{1}^{2}$ truncated hyperbolic metric.

i. The quantity
\begin{displaymath}
	\frac{\detp\Delta_{\cpd,\hyp,\varepsilon}}{\varepsilon^{\frac{1}{6}\sum_{i}(1-m_{i}^{-1})^{2}}(\log\varepsilon^{-1})^{\frac{c}{3}}}
\end{displaymath}
has a finite limit as $\varepsilon\to 0$. Let it be $\deta\Delta_{\hyp}$. It is a strictly positive real number.

ii. The Quillen type metric defined by
\begin{displaymath}
	\|\cdot\|_{Q,\hyp,\ast}=(\deta\Delta_{\hyp})^{-1/2}\|\cdot\|_{L^{2},\hyp}
\end{displaymath}
is such that Deligne's isomorphism induces a canonical isometry
\begin{displaymath}
	\begin{split}
	\lambda(\OO_{X})^{\otimes 12m^{2}}\otimes
	\bigotimes_{i}(\omega_{X,p_{i}},\|\cdot\|_{W,p_{i}})^{\otimes m^{2} ( 1-m_{i}^{-2})}
	\overset{\sim}{\longrightarrow}&\langle\omega_{X}^{m}(mD)_{\hyp},\omega_{X}^{m}(mD)_{\hyp}\rangle\\
	&\otimes\OO(m^{2}C^{\ast}(g)),
	\end{split}
\end{displaymath}
where $\OO(m^{2}C^{\ast}(g))=(\CC,m^{2}C^{\ast}(g)|\cdot|)$ and
\begin{equation}\label{eq:top-ctt}
	C^{\ast}(g)=2^{-c+\sum_{m_{i}<\infty}(1-m_{i}^{-1}) }
	\exp\left((2g-2)\left(\frac{\zeta^{\prime}(-1)}{\zeta(-1)}+\frac{1}{2}\right)\right)\prod_{m_{i}<\infty}m_{i}^{2(1-m_{i}^{-1})}.
\end{equation}
\end{proposition}
\begin{proof}
One easily sees that the $L^{2}$ metric with respect to the hyperbolic metric is well-defined, and the sequence $\|\cdot\|_{L^{2},\varepsilon}$ converges to $\|\cdot\|_{L^{2},\hyp}$. Then the proof is a consequence of the isometries \eqref{eq:16} and \eqref{eq:15}, and the convergence stated in Proposition \ref{prop:convergence_deligne_hyp} \emph{iii}.
\end{proof}
 We record the notation introduced in the statement of the theorem.
 \begin{notation}\label{notation:naive_det}
We define $\deta\Delta_{\hyp}$ as the finite limit
\begin{displaymath}
	\deta\Delta_{\hyp}:=\lim_{\varepsilon\to 0} \frac{\detp\Delta_{\cpd,\hyp,\varepsilon}}{\varepsilon^{\frac{1}{6}\sum_{i}(1-m_{i}^{-1})^{2}}(\log\varepsilon^{-1})^{\frac{c}{3}}}.
\end{displaymath}
We call it the \emph{naive determinant} of $\Delta_{\hyp}$. Observe this is a strictly positive real number.
\end{notation}
The next task will be to give a spectral interpretation of $\deta\Delta_{\hyp}$. For this, we will use the Mayer--Vietoris type formula for determinants of laplacians, and also the Selberg trace formula.

\section{Towards a spectral interpretation of the naive determinant}\label{section:spectral}
Let $X$ be a cusp compactification of a Riemann surface $\Gamma\backslash\HH$, for some fuchsian group $\Gamma\subset\PSL_{2}(\RR)$, and follow the conventions in \textsection\ref{subsubsec:discussion-fuchsian}. In the previous section we discussed the construction of a naive regularized determinant of the $\cpd$ hyperbolic laplacian, that we denoted by $\deta\Delta_{\cpd,\hyp}$. Defining a Quillen metric on $\lambda(\OO_{X})$ by means of $\deta\Delta_{\cpd,\hyp}$ and the $L^{2}$ metric (that is well-defined in this situation), one then has a Deligne type isometry (Proposition \ref{prop:naive_laplacian}). However a spectral interpretation of $\deta\Delta_{\cpd,\hyp}$ is missing. This is the question we proceed to address. The main tool will be the Mayer--Vietoris formula for determinants of laplacians (Theorem \ref{thm:MV-smooth}).

\subsection{Smoothening truncated metrics. Anomaly formulas.} The Mayer--Vietoris formula of Burghelea--Friedlander--Kappeler does not apply to truncated hyperbolic metrics, because these are only piecewise smooth. To circumvent this technical detail, we may smoothen the truncated metrics. 

\subsubsection{} Let $\varepsilon>0$ and consider \emph{rs} coordinates $\zeta_{i}$ at the points $p_{i}$ containing the closed disks $\ov{D}(0,\varepsilon)$ in their images. We suppose these disks define disjoint neighborhoods of the $p_{i}$ in $X$. We denote by $V_{\varepsilon}$ the union of the open disks $D(0,\varepsilon)$ in coordinates $\zeta_{i}$, and by $\ov{V}_{\varepsilon}$ its closure. With these notations, $V_{\varepsilon/2}$ is relatively compact in $V_{\varepsilon}$, namely $\ov{V}_{\varepsilon/2}\subset V_{\varepsilon}$. We consider the truncated hyperbolic metric $\|\cdot\|_{\hyp,\varepsilon}$ of Example \ref{example:truncated_hyperbolic}. 
\begin{lemma}\label{lemma:smoothening}
There exists a sequence $\|\cdot\|_{\varepsilon,k}$ of smooth hermitian metrics on $\omega_{X}$ with the following properties:

i. $\|\cdot\|_{\varepsilon,k}$ converges uniformly and in $L_{1}^{2}$ to $\|\cdot\|_{\hyp,\varepsilon}$;

ii. the metric $\|\cdot\|_{\varepsilon,k}$ coincides with $\|\cdot\|_{\hyp,\varepsilon}$ on an open neighborhood of $\ov{V}_{\varepsilon/2}$.
\end{lemma}
\begin{proof}
The proof relies on a standard convolution argument. The details are left as a routine exercise. 
\end{proof}
\begin{remark}
With respect to the simplified strategy presented in the introduction \textsection\ref{subsec:strategy}, the picture to keep in mind is the figure below. 
On the disks defined by $|z|\leq\varepsilon/2$ we put the frozen hyperbolic metric at height $\varepsilon$. The areas $\varepsilon/2<|z|<\varepsilon$ are ``transition" regions, that smoothly interpolate between the frozen metric and the hyperbolic metric. This smoothens the jumps occurring in the simplified strategy.
\begin{figure}[h]\label{figure2}
\begin{center}
\includegraphics[scale=0.4]{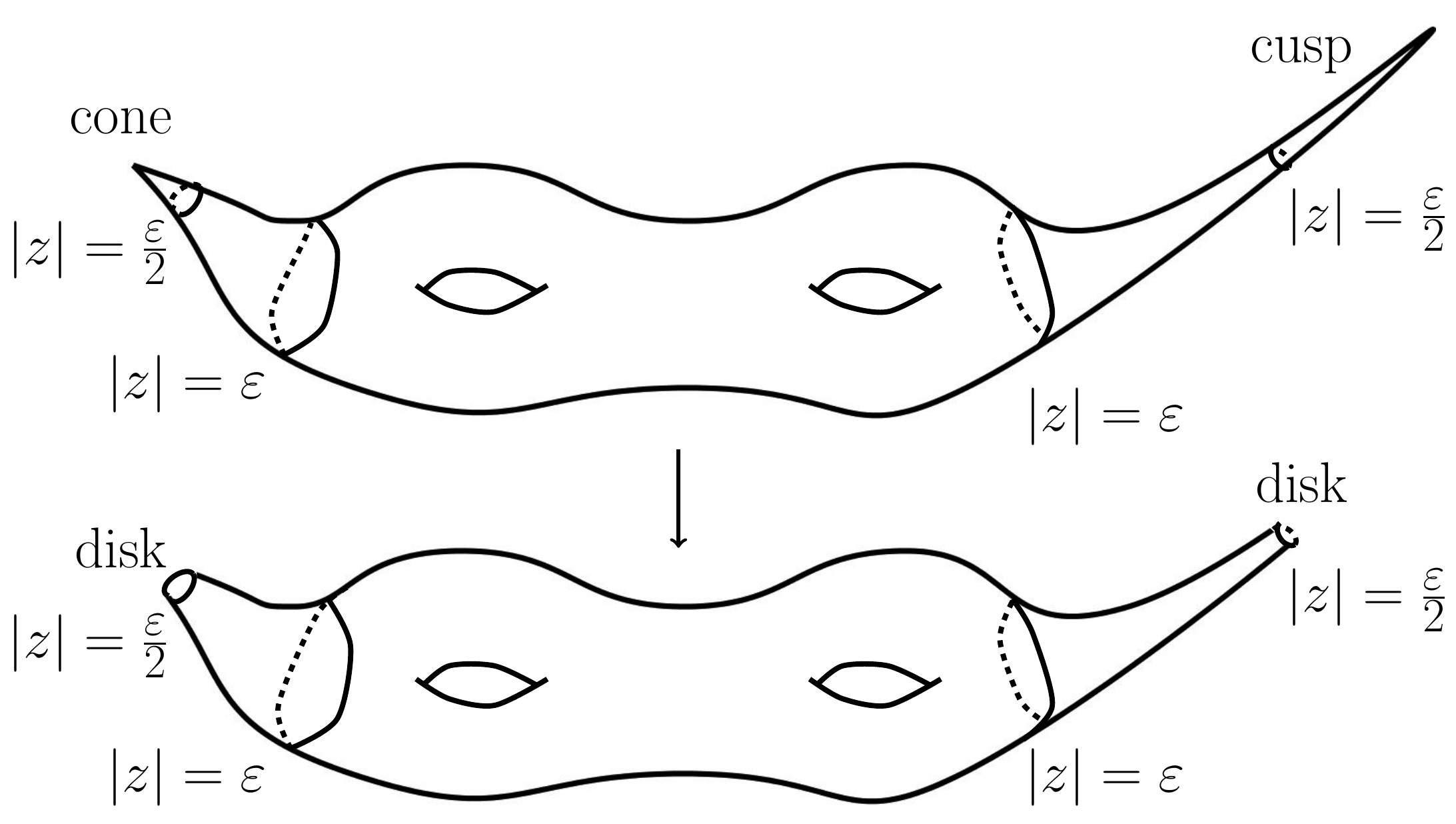}
\end{center}\caption{Regions in the smoothening of truncated metrics.}
\end{figure}
\end{remark}

\subsubsection{} Let us fix a sequence $\|\cdot\|_{\varepsilon,k}$ as in Lemma \ref{lemma:smoothening}. We define smooth functions $\varphi_{\varepsilon,k}$ on $X\setminus\lbrace p_{1},\ldots,p_{n}\rbrace$ by
\begin{displaymath}
	\|\cdot\|_{\varepsilon,k}=e^{-\varphi_{\varepsilon,k}}\|\cdot\|_{\hyp}.
\end{displaymath}
The convergence of $\|\cdot\|_{\varepsilon,k}$ to $\|\cdot\|_{\hyp,\varepsilon}$ then implies the local uniform and $L_{1\loc}^{2}$ convergence of the functions $\varphi_{\varepsilon,k}$ to the function $\varphi_{\varepsilon}$ defined by
\begin{displaymath}
	\|\cdot\|_{\hyp,\varepsilon}=e^{-\varphi_{\varepsilon}}\|\cdot\|_{\hyp}.
\end{displaymath}
In particular, this convergence is uniform and in $L_{1}^{2}$ on any relatively compact open subset of $X\setminus\lbrace p_{1},\ldots, p_{n}\rbrace$.

We put $M=X\setminus V_{\varepsilon/2}$, considered as a Riemann surface with boundary. We want to compare the determinants $\detp\Delta_{M,\hyp}$ and $\detp\Delta_{M,\varepsilon,k}$ (for $\|\cdot\|_{\varepsilon,k}$), with Dirichlet boundary conditions. This is achieved through the anomaly formula for determinants of scalar laplacians on riemannian surfaces with boundary. To be conformal with the literature we quote below, we introduce the following notations:
\begin{enumerate}
	\item[\textbullet] $\nabla$ is the riemannian gradient with respect to $ds_{\hyp}^{2}$ on $M$;
	\item[\textbullet] $d\vol$ is the riemannian volume with respect to $ds_{\hyp}^{2}$ on $M$;
	\item[\textbullet] $K\equiv -1$ is the gaussian curvature of $ds_{\hyp}^{2}$ on $M$;
	\item[\textbullet] $d\ell$ is the length element on $\pd M$ with respect to the metric induced by $ds_{\hyp}^{2}$;
	\item[\textbullet] $\pd_{n}$ is the derivative with respect to the exterior unit normal vector on $\pd M$, with respect to $ds_{\hyp}^{2}$;
	\item[\textbullet] $k_{g}$ is the geodesic curvature of $\pd M$.
\end{enumerate}
\begin{proposition}\label{prop:anomaly-boundary}
Let the notations be as above.

i. We have an equality of real numbers
\begin{equation}\label{eq:19}
	\begin{split}
	\log\detp\Delta_{M,\varepsilon,k}=&\log\detp\Delta_{M,\hyp}-\frac{1}{6\pi}\Big [\frac{1}{2}\int_{M}\|\nabla\varphi_{\varepsilon,k}\|^{2}d\vol
	+\int_{M}K\varphi_{\varepsilon,k}d\vol
	+\int_{\pd M}k_{g}\varphi_{\varepsilon,k}d\ell\Big]\\
	&
	-\frac{1}{4\pi}\int_{\pd M}\pd_{n}\varphi_{\varepsilon,k}d\ell.
	\end{split}
\end{equation}

ii. As $k\to\infty$, the sequence $\detp\Delta_{M,\varepsilon,k}$ converges to a positive real number $\detp\Delta_{M,\varepsilon}$ satisfying
\begin{displaymath}
	\begin{split}
		\log\detp\Delta_{M,\varepsilon}=\log\detp\Delta_{M,\hyp}+(\frac{1}{6}\log 2+\frac{1}{2})\sum_{i}(1-m_{i}^{-1}).
	\end{split}
\end{displaymath}
\end{proposition}
\begin{proof}
The first property is a restatement of \cite[Eq.~(1.17)]{POS}. For the convergence of $\detp\Delta_{M,\varepsilon,k}$ we first observe that
\begin{displaymath}
	\|\nabla\varphi_{\varepsilon,k}\|^{2}d\vol=2i\pd\varphi_{\varepsilon,k}\wedge\cpd\varphi_{\varepsilon,k}.
\end{displaymath}
and also the equality on $\pd M=\pd V_{\varepsilon/2}$
\begin{displaymath}
	\varphi_{\varepsilon,k}=\varphi_{\varepsilon},\quad \pd_{n}\varphi_{\varepsilon,k}=\pd_{n}\varphi_{\varepsilon}
\end{displaymath}
since $\|\cdot\|_{\varepsilon,k}$ coincides with $\|\cdot\|_{\varepsilon}$ on an open neighborhood of $\ov{V}_{\varepsilon/2}$. Then the uniform and $L_{1}^{2}$ convergence of $\varphi_{\varepsilon,k}$ to $\varphi_{\varepsilon}$ ensures the convergence of \eqref{eq:19} to the finite quantity
\begin{displaymath}
	\begin{split}
	\log\detp\Delta_{M,\varepsilon}:=&\log\detp\Delta_{M,\hyp}-\frac{1}{6\pi}\Big [i\int_{M}\pd\varphi_{\varepsilon}\wedge\cpd\varphi_{\varepsilon}
	+\int_{M}K\varphi_{\varepsilon}d\vol_{\hyp}
	+\int_{\pd M}k_{g}\varphi_{\varepsilon}d\ell\Big]\\
	&
	-\frac{1}{4\pi}\int_{\pd M}\pd_{n}\varphi_{\varepsilon}d\ell.
	\end{split}
\end{displaymath}
We thus have to explicitly compute the integrals in these expression. This is a routine computation, whose details we leave to the reader.
\end{proof}
\subsection{Mayer--Vietoris and truncated metrics}\label{subsec:MV-trunated}
The notations of the preceding sections are still in force. In particular we consider the sequence of smooth hermitian metrics $\|\cdot\|_{\varepsilon,k}$ on $\omega_{X}$, coinciding with $\|\cdot\|_{\hyp,\varepsilon}$ on a neighborhood of $\ov{V}_{\varepsilon/2}$. Let us denote by $\Delta_{\cpd, \varepsilon,k}$ the corresponding $\cpd$ laplacian. Because the sequence of metrics $\|\cdot\|_{\varepsilon,k}$ converges to $\|\cdot\|_{\hyp,\varepsilon}$ uniformly and in $L_{1}^{2}$, we can apply Lemma \ref{lemma:convergence_quillen} and derive that the limit
\begin{equation}\label{eq:22}
	\detp\Delta_{\cpd,\hyp,\varepsilon}:=\lim_{k\to\infty}\detp\Delta_{\cpd,\varepsilon,k}
\end{equation}
exists and is a real positive number. If $M=X\setminus V_{\varepsilon/2}$ and $\Delta_{M,\varepsilon,k}$ is the scalar laplacian on $M$ attached to $\|\cdot\|_{\varepsilon,k}$ (with Dirichlet boundary condition), by Proposition \ref{prop:anomaly-boundary} we find
\begin{equation}\label{eq:23}
	\lim_{k\to\infty}\detp\Delta_{M,\varepsilon,k}
	=\exp\left((\frac{1}{6}\log 2 +\frac{1}{2})\sum_{i}(1-m_{i}^{-1})\right)\detp\Delta_{M,\hyp}.
\end{equation}
Beware that $\Delta_{\cpd,\varepsilon,k}$ is a laplacian on the whole of $X$, while $\Delta_{M,\varepsilon,k}$ is a laplacian on $M$, so they are not the same. We follow a similar notation for $\Delta_{\ov{V}_{\varepsilon/2},\varepsilon,k}$. By construction $\|\cdot\|_{\varepsilon,k}$ coincides with $\|\cdot\|_{\hyp,\varepsilon}$ on a neighborhood of $\ov{V}_{\varepsilon/2}$, and the metric $\|\cdot\|_{\hyp,\varepsilon}$ is smooth on $V_{\varepsilon}$ (it is a flat metric). Thus
\begin{equation}\label{eq:20}
	\detp\Delta_{\ov{V}_{\varepsilon/2},\varepsilon,k}=\detp\Delta_{\ov{V}_{\varepsilon/2},\hyp,\varepsilon}.
\end{equation}
After these observations, we can formulate the next lemma.
\begin{lemma}\label{lemma:prelim_MV}
We have the equality
\begin{displaymath}
	\frac{\detp\Delta_{\cpd,\hyp,\varepsilon}}{\detp\Delta_{M,\hyp}\detp\Delta_{\ov{V}_{\varepsilon/2},\hyp,\varepsilon}}
	=C\frac{\vol(X,ds_{\hyp,\varepsilon}^{2})}{\ell(\pd V_{\varepsilon/2},ds_{\hyp}^{2})}\detp R_{\hyp},
\end{displaymath}
where $ds_{\hyp,\varepsilon}^{2}$ is the continuous riemannian metric associated to $\|\cdot\|_{\hyp,\varepsilon}$ and $C$ is the constant
\begin{equation}\label{eq:24}
	C=2^{(g+2)/3}\exp\left((\frac{1}{6}\log 2+\frac{1}{2})\sum_{i}(1-m_{i}^{-1})\right).
\end{equation}
\end{lemma}
\begin{proof}
We apply the Mayer--Vietoris formula (Theorem \ref{thm:MV-smooth}) to the smooth metrics $\|\cdot\|_{\varepsilon,k}$, after cutting the surface $X$ along $\pd V_{\varepsilon/2}$. The result reads
\begin{equation}\label{eq:21}
	\frac{\detp\Delta_{X,\varepsilon,k}}{\detp\Delta_{M,\varepsilon,k}\detp\Delta_{\ov{V}_{\varepsilon/2},\varepsilon,k}}
	=\frac{\vol(X,ds_{\varepsilon,k}^{2})}{\ell(\pd V_{\varepsilon/2},ds_{\varepsilon,k}^{2})}\detp R_{\varepsilon,k},
\end{equation}
with self-explanatory notations for $ds^{2}_{\varepsilon,k}$, and $R_{\varepsilon,k}$. For the left hand side of \eqref{eq:21}, we take into account the K\"ahler identity \eqref{eq:13} and the preceding observations \eqref{eq:22}--\eqref{eq:20}. For the right hand side, we observe the volume of $X$ with respect to $ds_{\hyp,\varepsilon}^{2}$ is well-defined and
\begin{displaymath}
	\lim_{k\to\infty}\vol(X,ds^{2}_{\hyp,\varepsilon,k})=\vol(X,ds^{2}_{\hyp,\varepsilon}).
\end{displaymath}
Besides, by Proposition \ref{prop:DN_invariant}
\begin{displaymath}
	\frac{\detp R_{\hyp}}{\ell(\pd V_{\varepsilon/2},ds_{\hyp}^{2})}
	=\frac{\detp R_{\varepsilon,k}}{\ell(\pd V_{\varepsilon/2},ds_{\varepsilon,k}^{2})}.
\end{displaymath}
All these ingredients establish the lemma.
\end{proof}
We now decompose $\ov{V}_{\varepsilon/2}$ into connected components $X_{i}$ corresponding to the points $p_{i}$, $i=1,\ldots,n$. We follow the notations of \textsection \ref{subsec:case_singular_hyp}.
\begin{theorem}\label{prop:Mayer--Vietoris}
The equality
\begin{displaymath}
	\detp(\Delta_{\cpd,\hyp,\varepsilon})=C\frac{\vol(X,ds_{\hyp,\varepsilon}^{2})}{\vol(X,ds_{\hyp}^{2})}\frac{\detp(\Delta_{\hyp},\Delta_{\hyp,\cusp})\detp\Delta_{\ov{V}_{\varepsilon/2},\hyp,\varepsilon}}{
	\prod_{m_{i}<\infty}\detp\Delta_{\hyp,i}}
\end{displaymath}
holds, where $C$ is the constant \eqref{eq:24}.
\end{theorem}
\begin{proof}
With the notations of \textsection \ref{subsec:case_singular_hyp} we have $X_{0}=M$ and $\Sigma=\pd V_{\varepsilon/2}$. The assertion is derived by a combination of Theorem \ref{thm:MV_hyp} and Lemma \ref{lemma:prelim_MV}.
\end{proof}

\section{Evaluation of the determinant of the pseudo-laplacian on a model cusp}
\label{section:evaluation-determinants}
We arrive at the technical core of the present work. We accomplish the task of asymptotically evaluating the determinant of the pseudo-laplacian on a model cusp (Theorem \ref{thm:det-pseudo-laplacian}). Recall that following Colin de 
Verdi\`ere, we introduced the pseudo-laplacian operator on the model cusp $\C_{a}$ (review \textsection\ref{subsec:model-cusp}). We used the notation $\Delta_{\ps}$ for this operator. If $\Delta_{D}$ is the hyperbolic laplacian on $\C_{a}$ with Dirichlet boundary condition, the very definitions of the pseudo-laplacian and the relative determinants show that
\begin{displaymath}
	\det(\Delta_{D},\Delta_{a})=\det\Delta_{\ps},
\end{displaymath}
of course if it makes sense. From the study of the eigenvalue problem for $\Delta_{\ps}$ discussed in \textsection\ref{subsubsec:eigen_problem_cusp}, we will be able to derive an integral representation for the associated spectral zeta function, establish the holomorphic continuation in a neighborhood of $s=0$, and evaluate $\det\Delta_{\ps}$. This will prove Theorem \ref{thm:det-pseudo-laplacian} (see \textsection\ref{subsubsec:pseudo-laplacian-conclusion} below). More precisely, we will show:
\begin{theorem}\label{theorem:zeta-pseudo-laplacian}
The spectral zeta function $\zeta_{\ps}(s)$ of the pseudo-laplacian $\Delta_{\ps}$ converges absolutely
and locally uniformly for $\Real(s)>1$, and admits a meromorphic continuation to the half-plane $\Real(s)>-\sigma_{0}$, for some $\sigma_{0}>0$. Furthermore, the derivative of $\zeta_{\ps}(s)$ at $s=0$ satisfies
the equality
\begin{displaymath}
	\zeta_{\ps}'(0)=4\pi\zeta(-1)a+\zeta(0)\log(a)+o(1),
\end{displaymath}
as $a\to +\infty$.
\end{theorem}

The proof of Theorem \ref{theorem:zeta-pseudo-laplacian} is long and technical, and we divide it in several steps \textsection\ref{subsubsec:pseudo-laplacian-1}--\textsection\ref{subsubsec:pseudo-laplacian-conclusion} to follow. Here is a guide to these paragraphs. In the first step \textsection\ref{subsubsec:pseudo-laplacian-1}, we discuss an integral representation for $\zeta_{\ps}(s)$ in terms of modified Bessel functions $K_{\nu}(y)$. 
This representation will be used to establish the meromorphic continuation and to evaluate the spectral zeta function at $s=0$. This occupies paragraphs \textsection\ref{subsubsec:pseudo-laplacian-2} to \textsection\ref{subsubsec:pseudo-laplacian-4}. We conclude with the proof of Theorem \ref{theorem:zeta-pseudo-laplacian} in \textsection\ref{subsubsec:pseudo-laplacian-conclusion}. The method of proof finds its inspiration in the physics literature (see for instance \cite{Barvinsky, Bordag, Flachi, Fucci}) which lacks of mathematical rigor.

\subsection{The spectral zeta function of $\Delta_{\ps}$: an integral representation}\label{subsubsec:pseudo-laplacian-1} In \textsection\ref{subsubsec:pseudo-laplacian}--\textsection\ref{subsubsec:eigen_problem_cusp}, we saw that $\Delta_{\ps}$ has discrete spectrum $\lbrace\lambda_{j}\rbrace_{j}\subset (1/4,+\infty)$, and the eigenvalues have multiplicity $2$. These eigenvalues are determined by the infinite set of equations
\begin{displaymath}
	K_{\nu-1/2}(2\pi |k| a)=0,\quad k\in\mathbb{Z}\setminus\lbrace 0\rbrace.
\end{displaymath}
In terms of $\nu$, the eigenvalues are given by $\lambda=\nu(1-\nu)$. We also saw that the zeros of $K_{\nu-1/2}(2\pi|k|a)$ in $\nu$ are simple. Therefore, it is enough to solve the equation for $k\geq 1$ and count the solutions $\lambda=\nu(1-\nu)$ twice. By \cite[Sec.~4]{CdVII}, the eigenvalue counting function $N$ for $\Delta_{\ps}$ satiefies the Weyl type bound
\begin{displaymath}
	N(\lambda)\leq C\lambda,
\end{displaymath}
for some real positive constant $C>0$. The proof of \emph{loc.~cit.}~actually covers the case of compact Riemann surfaces without boundary, but it can be adapted to the boundary case, like $\C_{a}$, after imposing Dirichlet boundary conditions. From the Weyl type bound, one sees that the spectral zeta function
of $\Delta_{\ps}$ defined by
\begin{displaymath}
	\zeta_{\ps}(s)=\sum_{j}\frac{1}{\lambda_{j}^{s}}
\end{displaymath}
is absolutely and locally uniformly convergent for $s\in\CC$ with $\Real(s)>1$. 
Hence, 
in this region, we can write
\begin{displaymath}
	\zeta_{\ps}(s)=2\sum_{k\geq 1}\sum_{j}\frac{1}{(1/4+r_{k,j}^{2})^{s}},
\end{displaymath}
where, for fixed $k$, $\nu_j=1/2+ir_{k,j}$ with $r_{k,j}>0$ are the zeros of $K_{\nu-1/2}(2\pi k a)$ as a function of $\nu$. Observe that the factor $2$ in this new expression for $\zeta_{\ps}(s)$ takes into account the multiplicity of the spectrum. We will establish the meromorphic continuation and holomorphicity of $\zeta_{\ps}(s)$ at $s=0$. By the inverse Mellin transform, this is equivalent to the usual necessary asymptotics for $\tr(e^{-t\Delta_{\ps}})$ as $t\to 0$.

By the residue theorem and since the zeros of $K_{s-1/2}(2\pi k a)$ are simple, 
for $s\in\CC$ with $\Real(s)>1$, we can write $\zeta_{\ps}(s)$ as an absolutely 
convergent sum of path integrals
\begin{displaymath}
	\zeta_{\ps}(s)
	=2\sum_{k\geq 1}\frac{1}{2\pi i}\int_{\gamma_{\theta}} (1/4+r^{2})^{-s}\frac{\partial}{\partial r}
	\log K_{ir}(2\pi k a) du.
\end{displaymath}
Some comments are in order. First of all, because we showed that the eigenvalues are $>1/4$, we can take the path
\begin{displaymath}
	\gamma_{\theta}=[+\infty e^{i\theta},0]\cup [0,+\infty e^{-i\theta}],
\end{displaymath}
for any angle $0<\theta<\pi/2$. Second, we need to fix the principal branch of the logarithm to define $(1/4+r^{2})^{-s}$. We can rotate the picture by making the change of variabels $t=ir$, so that
\begin{displaymath}
	\zeta_{\ps}(s)=2\sum_{k\geq 1}\frac{1}{2\pi i}\int_{i\gamma_{\theta}}(1/4-t^{2})^{-s}\frac{\partial}{\partial t}\log K_{t}(2\pi k a)dt.
\end{displaymath}
Here, we used the symmetry  $K_{-t}=K_{t}$ of the modified Bessel function. We would now like to move the path of integration $i\gamma_{\theta}$ to the real axis, by letting the angle $\theta\to\pi/2$. In this process, one encounters convergence problems at $t=1/2$. To fix this, we regularize by substracting a suitable quantity (actually zero!). Let us define the function
\begin{equation}\label{def:26a}
	f_{k}(t)=\frac{\partial}{\partial t}\log K_{t}(2\pi k a)-2t\frac{\partial}{\partial t}\mid_{t=1/2}\log K_{t}(2\pi k a).
\end{equation}
The function $f_{k}(t)$ is an odd function in $t$, since $K_{t}$ is even in $t$. Moreover, for $s\in\CC$ with $\Real(s)>1$, we have
\begin{displaymath}
	\int_{i\gamma_{\theta}}(1/4-t^{2})^{-s}tdt=0,
\end{displaymath}
as we see by computing a primitive. Therefore, we obtain
\begin{displaymath}
	\int_{i\gamma_{\theta}}(1/4-t^{2})^{-s}f_{k}(t)dt=\int_{i\gamma_{\theta}}(1/4-t^{2})^{-s}\frac{\partial}{\partial t}\log K_{t}(2\pi k a)dt.
\end{displaymath}
The advantage is that now we can deform the path to the real axis. Indeed, the function $f_{k}(t)$ is analytic in $t$ and $f_{k}(1/2)=f_{k}(-1/2)$. One then easily sees that, for $s\in\CC$ with $1<\Real(s)<2$, we have
\begin{displaymath}
	\begin{split}
	\int_{i\gamma_{\theta}}(1/4-t^{2})^{-s}f_{k}(t)dt&=\lim_{\theta\to\pi/2}\int_{i\gamma_{\theta}}(1/4-t^{2})^{-s}f_{k}(t)dt\\
	&=\int_{-\infty}^{+\infty}(1/4-t^{2})^{-s}f_{k}(t)dt.
	\end{split}
\end{displaymath}
In the last expression, one has to be careful with the branches of the logarithm, because $1/4-t^{2}$ vanishes at $t=\pm 1/2$. Using that $f_{k}(t)$ is odd, one finds that
\begin{displaymath}
	\int_{-\infty}^{+\infty}(1/4-t^{2})^{-s}f_{k}(t)dt=\int_{-\infty}^{-1/2}(1/4-t^{2})^{-s}f_{k}(t)dt +\int_{1/2}^{+\infty}(1/4-t^{2})^{-s}f_{k}(t)dt.
\end{displaymath}
In the first integral on the right hand side, one has
\begin{displaymath}
	(1/4-t^{2})^{-s}=e^{-si\pi}(t^{2}-1/4)^{-s},\quad t>1/2,
\end{displaymath}
and in the second integral
\begin{displaymath}
	(1/4-t^{2})^{-s}=e^{si\pi}(t^{2}-1/4)^{-s},\quad t>1/2.
\end{displaymath}
By the property of $f_{k}(t)$ being odd, we arrive at
\begin{displaymath}
	\int_{-\infty}^{+\infty}(1/4-t^{2})^{-s}f_{k}(t)dt=2i\sin(\pi s)\int_{1/2}^{+\infty}(t^{2}-1/4)^{-s}f_{k}(t)dt.
\end{displaymath}
Hence, for $s\in\CC$ with $1<\Real(s)<2$, we obtain
\begin{equation}\label{eq:26}
	\zeta_{\ps}(s)=\sum_{k\geq 1}2\frac{\sin(\pi s)}{\pi}\int_{1/2}^{+\infty}(t^{2}-1/4)^{-s}f_{k}(t)dt.
\end{equation}
We stress that the region $1<\Real(s)<2$ is contained in the region of absolute convergence of the spectral zeta function. To conclude with this representation, it will be important to have a more precise understanding of $f_{k}(t)$. For this we will employ the special value computation
\begin{displaymath}
	\frac{\partial}{\partial t}\mid_{t=1/2}\log K_{t}(z)=\mathbb{E}_{1}(2z)e^{2z},
\end{displaymath}
where $\mathbb{E}_{1}(z)$ is the exponential integral
\begin{displaymath}
	\mathbb{E}_{1}(z)=\int_{z}^{+\infty}\frac{e^{-t}}{t}dt.
\end{displaymath}
Thus, for $s\in\CC$ with $1<\Real(s)<2$, the zeta function $\zeta_{\ps}(s)$ is an absolutely convergent sum of integrals
\begin{equation}\label{eq:27}
	I_{k}(s):=2\frac{\sin(\pi s)}{\pi}\int_{1/2}^{+\infty}(t^{2}-1/4)^{-s}\left\lbrace \frac{\partial}{\partial t}\log K_{t}(2\pi k a)-2t\,\mathbb{E}_{1}(4\pi k a)e^{4\pi k a}\right\rbrace dt.
\end{equation}
In the sequel, we will study the integrals $I_{k}(s)$. It will be necessary to divide the interval $(1/2,+\infty)$. More precisely, introducing a small parameter $0<\delta< 1/8$, the division will be
\begin{displaymath}
	(1/2,+\infty)=(1/2,k^{\delta}]\cup [k^{\delta},+\infty).
\end{displaymath}
Accordingly, for $k\in\NN_{>0}$, we have $I_{k}(s)=L_{k}(s)+M_{k}(s)$, with
\begin{align}
	&L_{k}(s):=2\frac{\sin(\pi s)}{\pi}\int_{1/2}^{k^{\delta}}(t^{2}-1/4)^{-s}f_{k}(t)dt,\label{eq:28}\\
	&M_{k}(s):=2\frac{\sin(\pi s)}{\pi}\int_{k^{\delta}}^{+\infty}(t^{2}-1/4)^{-s} f_{k}(t)dt.\label{eq:29}
\end{align}
This division corresponds to different asymptotic behaviours of the modified Bessel functions. The dependence of the interval on $k$ is a source of difficulties.\footnote{The physics literature neglects this feature. However, this is one of the delicate and crucial points of the computation.}

\subsection{Study of the integrals $L_{k}(s)$}\label{subsubsec:pseudo-laplacian-2} Recall the definition \eqref{eq:28} of the integral $L_{k}(s)$, corresponding to the interval $(1/2,k^{\delta}]$, for some fixed parameter $0<\delta< 1/8$. We observe that the integrals $L_{k}(s)$ are valid in the region $\Real(s)<2$. The aim of this paragraph is the proof of the following theorem.

\begin{theorem}\label{thm:L-integral-cusp}
The infinite sum $\zeta_{L}(s)=\sum_{k}L_{k}(s)$ is absolutely and locally uniformly convergent for 
$s\in\CC$ with $-3<\Real(s)<2$. Furthermore, we have
\begin{displaymath}
	\zeta_{L}'(0)=O\left(\frac{1}{a^{2}}\right)\quad\text{and}\quad\lim_{a\to+\infty}\zeta_{L}'(0)=0.
\end{displaymath}
\end{theorem}
The proof invokes several properties of the modified Bessel function $K_{\nu}(z)$ and the exponential integral $\mathbb{E}_{1}(z)$, that we next summarize.
\begin{proposition}\label{prop:asymptotics-bessel-1}
	\begin{enumerate}
		\item For complex $\nu$ and real $z$, the Bessel function $K_{\nu}(z)$ satisfies:
		\begin{displaymath}
			K_{\nu}(z)=K_{1/2}(z)\left(\sum_{m=0}^{\ell-1}\frac{a_{m}(\nu)}{z^{m}}+\rho_{\ell}(\nu,z)\frac{a_{\ell}(\nu)}{z^{\ell}}\right),
		\end{displaymath}
		where $a_{0}\equiv 1$ and
		\begin{displaymath}
			a_{m}(\nu)=\frac{1}{m!8^{m}}\prod_{j=1}^{m}(4\nu^{2}-(2j-1)^{2}),
		\end{displaymath}
		for $m\geq 1$. Furthermore, we have
		\begin{displaymath}
			|\rho_{\ell}(\nu,z)|\leq 2\exp(|\nu^{2}-1/4|/|z|).
		\end{displaymath}
		\item For real $z$, the exponential integral $\mathbb{E}_{1}(z)$ admits the asymptotics
		\begin{equation}\label{eq:30}
			\mathbb{E}_{1}(z)e^{z}\sim\frac{1}{z}\sum_{m=0}^{\infty}(-1)^{m}\frac{m!}{z^{m}},
		\end{equation}
		as $z\to +\infty$.
	\end{enumerate}
\end{proposition}
\begin{proof}
These are classical facts \cite{Olver}.
\end{proof}
Now, we define the function
\begin{equation}\label{eq:31}
		F_{k}(t):=\log K_{t}(2\pi k a)-\log K_{1/2}(2\pi k a)
		-(t^{2}-1/4)\,\mathbb{E}_{1}(4\pi k a)e^{4\pi k a}.
\end{equation}
The previous proposition will be applied to obtain the following estimate.
\begin{corollary}\label{cor:nightmare-bessel}
On the interval $[1/2,k^{\delta}]$ the function $F_{k}(t)$, given by \eqref{eq:31}, can be written as
\begin{displaymath}
		F_{k}(t)=(t^{2}-1/4)^{2}R(t),
\end{displaymath}
where $R$ is analytic in $t$ and, on $[1/2,k^{\delta}]$, we have
\begin{displaymath}
	R(t)=O\left(\frac{1}{k^{2-4\delta}a^{2}}\right),
\end{displaymath}
with an implicit constant independent of $k$.
\end{corollary}
\begin{proof}
The function $F_{k}$ has been defined so that $F_{k}(\pm 1/2)=F_{k}^{\prime}(\pm 1/2)=0$. It is an even entire function of $t$, so of the form $F_{k}(t)=h_{k}(t^{2})$, where $h_k$ is entire and $h_k(1/4)=h'_k(1/4)=0$. We consider the Taylor expansion of $h_{k}(u)$ at $u=1/4$, and derive
\begin{equation}\label{eq:34}
	F_{k}(t)=h_{k}(t^{2})=\frac{1}{2}(t^{2}-1/4)^{2}h''(\xi^{2}),
\end{equation}
for some $\xi\in [1/2, k^{\delta}]$. The derivatives of $F_{k}$ and $h_{k}$ satisfy the relation
\begin{equation*}
	h''_{k}(t^{2})=\frac{F_{k}''(t)}{4t^{2}}-\frac{F_{k}'(t)}{4t^{3}}.
\end{equation*}
Taking into account the definition \eqref{eq:31} of $F_{k}(t)$, we simplify to
\begin{displaymath}
	h''_{k}(t^{2})=\frac{1}{4t^{2}}\frac{\partial^{2}}{\partial t^{2}}\log K_{t}(2\pi k a)-\frac{1}{4t^{3}}\frac{\partial}{\partial t}\log K_{t}(2\pi k a).
\end{displaymath}
Therefore, we write \eqref{eq:34} as
\begin{equation}\label{eq:37}
	F_{k}(t)=(t^{2}-1/4)^{2}\left(\frac{1}{8\xi^{2}}\frac{\partial^{2}}{\partial t^{2}}\mid_{t=\xi}\log K_{t}(2\pi k a)-\frac{1}{8\xi^{3}}\frac{\partial}{\partial t}\mid_{t=\xi}\log K_{t}(2\pi k a) \right).
\end{equation}
We now need to estimate the derivatives in this expression, and for this it is enough to estimate the expressions
\begin{displaymath}
	\frac{1}{K_{t}(2\pi k a)}\frac{\partial}{\partial t} K_{t}(2\pi k a),\quad \frac{1}{K_{t}(2\pi k a)}\frac{\partial^{2}}{\partial t^{2}} K_{t}(2\pi k a),
\end{displaymath}
on the interval $[1/2,k^{\delta}]$. We deal with the first one, and leave the second one to the reader. Let $\xi\in [1/2, k^{\delta}]$ and $D_{\xi}\subset\CC$ be the disk centered at $\xi$ and of radius $1/4$. Then, because $K_{\nu}(z)$ is entire in $\nu$ for $z>0$, we have
\begin{displaymath}
	\frac{\partial}{\partial t}\mid_{t=\xi}\log K_{t}(2\pi k a)=\frac{1}{2\pi i K_{\xi}(2\pi k a)}\int_{\partial D_{\xi}}\frac{K_{\nu}(2\pi k a)}{(\nu-\xi)^{2}}d\nu.
\end{displaymath}
We now use the asymptotics for the Bessel functions to the second order (Proposition \ref{prop:asymptotics-bessel-1}), to derive
\begin{equation}\label{eq:35}
	\begin{split}
	\frac{\partial}{\partial t}\mid_{t=\xi}\log K_{t}(2\pi k a)=&\frac{1}{2\pi i}\frac{K_{1/2}(2\pi k a)}{K_{\xi}(2\pi k a)}
	\\
	&
	\cdot\int_{\partial D_{\xi}}\frac{1}{(\nu-\xi)^{2}}\left(1+\frac{a_{1}(\nu)}{2\pi k a}+\rho_{2}(\nu,2\pi k a)\frac{a_{2}(\nu)}{(2\pi k a)^{2}}\right)d\nu\\
	=&
	\frac{K_{1/2}(2\pi k a)}{K_{\xi}(2\pi ka)}\left(\frac{\xi}{2\pi k a}+O\left(\frac{1}{k^{2-4\delta} a^{2}}\right)\right).
	\end{split}
\end{equation}
We used the expressions of $a_1$ and $a_2$ and the estimate for the remainder $\rho_{2}$ provided by Proposition \ref{prop:asymptotics-bessel-1}. In particular we see that the $O$ term has an implicit constant independent of $k$. The corresponding result for the second derivative is
\begin{equation}\label{eq:36}
	\frac{\partial^{2}}{\partial t^{2}}\mid_{t=\xi}\log K_{t}(2\pi k a)=\frac{K_{1/2}(2\pi k a)}{K_{\xi}(2\pi ka)}\left(\frac{1}{2\pi k a}+O\left(\frac{1}{k^{2-4\delta} a^{2}}\right)\right).
\end{equation}
Inserting \eqref{eq:35}--\eqref{eq:36} into \eqref{eq:37}, one finds
\begin{displaymath}
	F_{k}(t)=\frac{K_{1/2}(2\pi k a)}{K_{\xi}(2\pi ka)}\,O\left(\frac{1}{k^{2-4\delta} a^{2}}\right),
\end{displaymath}
with implicit constant independent of $k$. To conclude, the asymptotics of the Bessel functions show that the quotient $K_{1/2}(2\pi k a)/K_{\xi}(2\pi ka)$ is bounded on $[1/2,k^{\delta}]$, uniformly in $k$.
\end{proof}

\begin{proof}[Proof of Theorem \ref{thm:L-integral-cusp}]
The integral $L_{k}(s)$ ($k\in\NN_{>0}$) defines a holomorphic function for $s\in\CC$ with $\Real(s)<2$, and integration by parts gives the value
\begin{displaymath}
		L_{k}(s)=2\frac{\sin(\pi s)}{\pi}(k^{2\delta}-1/4)^{-s} F_{k}(k^{\delta})
		+2\frac{\sin(\pi s)}{\pi}\int_{1/2}^{k^{\delta}}2st(t^{2}-1/4)^{-s-1}F_{k}(t)dt,
\end{displaymath}
valid only in the region $\Real(s)<1$. The first term in this expression will be denoted by $A_{k}(s)$ and the second one by $B_{k}(s)$. Observe that $A_{k}(s)$ is holomorphic on the whole complex $s$-plane. Moreover, by Corollary \ref{cor:nightmare-bessel}, we see
\begin{equation}\label{eq:32}
	\begin{split}
		A_{k}(s)=2\frac{\sin(\pi s)}{\pi}\frac{1}{(k^{2\delta}-1/4)^{s-2}} O\left(\frac{1}{k^{2-4\delta}a^{2}}\right).
	\end{split}
\end{equation}
It follows that the sum $\sum_{k}A_{k}(s)$ is absolutely and locally uniformly convergent for $\Real(s)>4\delta-1/2\delta$. Since from the beginning we are supposing $\delta<1/8$, the sum is absolutely and locally uniformly convergent for $\Real(s)>-3$, where it defines a holomorphic function. Moreover, it is clear that
\begin{equation}\label{eq:33}
	\frac{\partial}{\partial s}\mid_{s=0}\sum_{k}A_{k}(s)=O\left(\frac{1}{a^{2}}\right).
\end{equation}
For the integral $B_{k}(s)$, again by Corollary \ref{cor:nightmare-bessel}, we have
\begin{displaymath}
	B_{k}(s)=2\frac{\sin(\pi s)}{\pi}\int_{1/2}^{k^{\delta}}2st(t^{2}-1/4)^{-s+1} R(t)dt.
\end{displaymath}
This defines a holomorphic function for $\Real(s)<2$. Moreover, by the estimate on $R(t)$ ensured by the corollary, we have the bound
\begin{displaymath}
	|B_{k}(s)|\ll\frac{|s\sin(\pi s)|}{(2-\Real(s))}\frac{1}{(k^{2\delta}-1/4)^{\Real(s)-1}}\frac{1}{k^{2-4\delta}a^{2}},
\end{displaymath}
with an implicit constant independent of $k$. This implies the absolute and locally uniform convergence of the sum $\sum_{k} B_{k}(s)$ for $-3<\Real(s)<2$.
It is also clear that
\begin{displaymath}
	\frac{\partial}{\partial s}\mid_{s=0}\sum_{k}B_{k}(s)=O(1/a^{2}).
\end{displaymath}
This completes the proof of the theorem.
\end{proof}
\begin{remark}
The delicate point in the proof of Theorem \ref{thm:L-integral-cusp} is Corollary \ref{cor:nightmare-bessel}. The method we followed circumvents the lack of precise asymptotics with remainder for the derivatives of the Bessel functions. 
\end{remark}

\subsection{Study of the integrals $M_{k}(s)$. First part}\label{subsubsec:M-int-1}
Recall the definition \eqref{eq:29} of the integral $M_{k}(s)$ 
corresponding to the interval $[k^{\delta},+\infty)$. For $s\in\CC$ with $1<\Real(s)<2$,
we perform the splitting
\begin{displaymath}
	M_{k}(s)=\widetilde{M}_{k}(s)+R_{k}(s)
\end{displaymath}
with
\begin{align}
	\widetilde{M}_{k}(s)&:=2\frac{\sin(\pi s)}{\pi}\int_{k^{\delta}}^{\infty}(t^{2}-1/4)^{-s}\frac{\partial}{\partial t}\log K_{t}(2\pi k a) dt,\label{eq:38}\\
	R_{k}(s)&:=-2\frac{\sin(\pi s)}{\pi}\int_{k^{\delta}}^{\infty}(t^{2}-1/4)^{-s}2t\,\mathbb{E}_{1}(4\pi k a)e^{4\pi k a}dt,
	\label{eq:39}
\end{align}
for some fixed parameter $0<\delta< 1/8$. 
\begin{proposition}
The infinite sum $\zeta_{R}(s)=\sum_{k}R_{k}(s)$ is absolutely and locally uniformly convergent for $\Real(s)>1$, and admits a meromorphic function to the half-plane $\Real(s)>-3$ with a unique pole at $s=1$. Moreover,
we have 
\begin{displaymath}
	\zeta_{R}'(0)=O\left(\frac{1}{a}\right)\quad\text{and}\quad\lim_{a\to +\infty}\zeta_{R}'(0)=0.
\end{displaymath}
\end{proposition}
\begin{proof}
The proof is an easy application of the asymptotics for the exponential integral $\mathbb{E}_{1}(z)$ (Proposition \ref{prop:asymptotics-bessel-1}). We leave the details to the reader.
\end{proof}
To deal with the infinite sum $\sum_{k}\widetilde{M}_{k}(s)$, we will appeal to another asymptotic for the Bessel function $K_{t}(2\pi k a)$, that will hold on the interval $[k^{\delta},+\infty]$. 
\begin{proposition}\label{prop:asymptotics-bessel-2}
For real $\nu, z>0$, we have
\begin{equation}\label{eq:40}
	K_{\nu}(\nu z)=\sqrt{\frac{\pi}{2 \nu}} \frac{e^{-\nu \eta}}{(1+z^2)^{1/4}}
	\left( \sum\limits_{n=0}^{\ell-1}(-1)^n \frac{u_n(\tau)}{\nu^n}+\frac{\rho_{\ell}(\nu, z)}{\nu^{\ell}}\right),
\end{equation}
where the functions $\tau$ and $\eta$ are given in terms of $z$ by
\begin{displaymath}
	\tau=\tau(z)=\frac{1}{\sqrt{1+z^2}},
\end{displaymath}
and
\begin{displaymath}
	\eta=\eta(z)=\log \left(\frac{z}{1+\sqrt{1+z^2}}\right) +\sqrt{1+z^2}.
\end{displaymath}
The $u_n(\tau)$ are polynomials in $\tau$ of degree $3 n$, given by $u_0(\tau)=1$ and 
\begin{align*}
	u_{n+1}(\tau)=\frac{1}{2} \tau^2 (1-\tau^2) u'_{n}(\tau)+\frac{1}{8}  
\int\limits_{0}^{\tau}(1-5 x^2) u_n(x) dx..
\end{align*}
In particular, we have
\begin{displaymath}
	u_1(\tau)=(3 \tau-5 \tau^3)/24, \qquad u_2(\tau)=(81 \tau^2-462 \tau^4 +385 \tau^6)/1152.
\end{displaymath}
Finally, the remainder term $\rho_{\ell}(\nu, z)$ is uniformly bounded, and for fixed $\nu$ it satisfies
\begin{displaymath}
	\rho_{\ell}(\nu, z)=O\left(\frac{1}{z^{\ell}}\right).
\end{displaymath}
\end{proposition}
\begin{proof}
The statements summarize classical bounds for modified Bessel functions of large order, see for instance \cite{Olver}.
\end{proof}
\begin{remark}
To apply Proposition \ref{prop:asymptotics-bessel-2}, we write
\begin{displaymath}
	K_{t}(2\pi k a)=K_{t}(t\frac{2\pi k a}{t}).
\end{displaymath}
Then we restrict the parameter $t$ to $[k^{\delta},+\infty)$, so that in this interval $t\geq k^{\delta}$ and $k^{\delta}\to +\infty$ when $k\to +\infty$. Therefore, for the final summation and analytic continuation, the asymptotics \eqref{eq:40} can be used. This is one of the interests of cutting the interval of integration at $k^{\delta}$.
\end{remark}
With the proposition at hand, and looking at the logarithm of the Bessel asymptotics, one can guess a suitable correction term that will be needed to turn the sum $\sum_{k}\widetilde{M}_{k}(s)$ into an analytic function in a neighborhood of $s=0$, while keeping the absolute convergence on the region $1<\Real(s)<2$. Let us elaborate on this important manipulation. We follow the notations of Proposition \ref{prop:asymptotics-bessel-2}, hence $\tau=\tau(2\pi k a/t)$ and $\nu=t$, and we put
\begin{displaymath}
	\mathcal{E}_{k}(t):=\log\left( \sqrt{\frac{\pi}{2 t}} \frac{e^{-t \eta}}{(1+(2\pi k a/t)^2)^{1/4}}\right)+\frac{U_{1}(\tau)}{t},
\end{displaymath}
where $U_{1}$ is the polynomial
\begin{displaymath}
	U_{1}(\tau)=-(3\tau-5\tau^{3})/24.
\end{displaymath}
We then define the ``error integral"
\begin{equation}\label{eq:41}
	E_{k}(s):=\int_{k^{\delta}}^{\infty}(t^{2}-1/4)^{-s}\frac{\partial}{\partial t}\mathcal{E}_{k}(t)dt.
\end{equation}
One easily checks the convergence on $1<\Real(s)<2$. We modify $\widetilde{M}_{k}(s)$ by substracting $E_{k}(s)$:
\begin{displaymath}
	\begin{split}
		M^{\ast}_{k}(s):=&\widetilde{M}_{k}(s)-E_{k}(s)\\
		=&2\frac{\sin(\pi s)}{\pi}\int_{k^{\delta}}^{\infty}(t^{2}-1/4)^{-s}\frac{\partial}{\partial t}\left(\log K_{t}(2\pi k a)-\mathcal{E}_{k}(t)\right)dt.
	\end{split}
\end{displaymath}
The integrals $M_{k}^{\ast}(s)$ won't contribute to the final computations, as we next explain.

\begin{theorem}\label{thm:Mast-integral-cusp}
The infinite sum $\zeta_{M^{\ast}}(s)=\sum_{k} M_{k}^{\ast}(s)$ is absolutely and locally uniformly convergent for
$s\in\CC$ with $1<\Real(s)<2$, and admits a meromorphic continuation to $\Real(s)>-\sigma_{0}$, for some $\sigma_{0}>0$. Furthermore, $\zeta_{M^{\ast}}(s)$ is holomorphic at $s=0$ and
\begin{displaymath}
	\lim_{a\to +\infty}\zeta_{M^{\ast}}'(0)=0.
\end{displaymath}
\end{theorem}
\begin{proof}
For a well-chosen integer $\ell\geq 3$, to be adjusted below, and as a consequence of Proposition \ref{prop:asymptotics-bessel-2}, we expand
\begin{equation}\label{eq:expansion-log-bessel}
	\log K_{t}(2\pi k a)-\mathcal{E}_{k}(t)=\sum_{n=2}^{\ell-1}\frac{U_{n}(\tau)}{t^{n}}+\frac{\widetilde{\rho}_{\ell}(t,2\pi ka /t)}{t^{\ell}}.
\end{equation}
The $U_{n}(\tau)$ are explicitly computable polynomials and the remainder term $\widetilde{\rho}_{\ell}$ is uniformly bounded. For fixed $t$, we have
\begin{equation}\label{eq:pain-3}
	\widetilde{\rho}_{\ell}(t,2\pi ka /t)=O(1/(2\pi k a/t)^{\ell})=O(1/a^{\ell}).
\end{equation}
We split the integral according to the expansion \eqref{eq:expansion-log-bessel}. We start by studying the contribution of the remainder term. After integration by parts, we are left with the expression
\begin{displaymath}
		N_{k}(s):=
		2\frac{\sin(\pi s)}{\pi}\frac{1}{(k^{2\delta}-1/4)^{s}}\frac{\widetilde{\rho}_{\ell}(k^{\delta},2\pi a k^{1-\delta})}{k^{\ell\delta}}
		+2\frac{\sin(\pi s)}{\pi}\int_{k^{\delta}}^{\infty}2st(t^{2}-1/4)^{-s-1}\frac{\widetilde{\rho}_{\ell}}{t^{\ell}}dt,
\end{displaymath}
which is valid as long as $\Real(s)>0$. We now choose the order $\ell$ of the expansion so that $\ell\delta>1$. Because $\widetilde{\rho}_{\ell}$ is uniformly bounded, we then deduce that the function
\begin{displaymath}
	f(a,s):=\sum_{k\geq 1}\frac{1}{(k^{2\delta}-1/4)^{s}}\frac{\widetilde{\rho}_{\ell}}{k^{\ell\delta}}
\end{displaymath}
is absolutely and locally uniformly convergent if $\Real(s)>-\sigma_{0}$, for some $\sigma_{0}>0$. Besides, for fixed $k$, \eqref{eq:pain-3} guarantees that the summands individually converge to 0 as $a\to +\infty$. By the dominate convergence theorem, we then deduce that
\begin{displaymath}
	\lim_{a\to +\infty}f(a,s)=0.
\end{displaymath}
This is valid locally uniformly for $\Real(s)>-\sigma_{0}$. Therefore, the same is true for the derivatives, being a family of holomorphic functions in $s$. This implies
\begin{equation}\label{eq:pain-1}
	\lim_{a\to +\infty} \frac{\partial}{\partial s}\mid_{s=0} 2\frac{\sin(\pi s)}{\pi}f(a,s)=0.
\end{equation}
For the integrals, we effect the change of variables $t=ku$. They become
\begin{displaymath}
	\int_{k^{\delta-1}}^{\infty}2s(ku)((ku)^{2}-1/4)^{-s-1}\frac{\widetilde{\rho}_{\ell}}{k^{\ell -1}u^{\ell}}du.
\end{displaymath}
On the interval of integration, we have uniform bounds
\begin{displaymath}
	\left |\frac{\widetilde{\rho}_{\ell}}{k^{\ell -1}u^{\ell}}\right |\ll_{a} \frac{1}{k^{\ell -1+\ell(\delta-1)}}.
\end{displaymath}
We need to adjust $\ell$ so that
\begin{displaymath}
	\ell -1 +\ell(\delta-1)>1.
\end{displaymath}
For this we need $\ell>2/\delta$. This covers the range $\ell>1/\delta$, hence guarantees \eqref{eq:pain-1}. Under this assumption, we see that
\begin{displaymath}
	g(a,s):=\sum_{k\geq 1}\int_{k^{\delta -1}}^{\infty}2s(ku)((ku)^{2}-1/4)^{-s-1}\frac{\widetilde{\rho}_{\ell}}{k^{\ell -1}u^{\ell}}du
\end{displaymath}
is absolutely and locally uniformly convergent for $\Real(s)>-\sigma_{0}$ and some $\sigma_{0}>0$. For fixed $k$ and $u$, the remainder term satisfies $\widetilde{\rho}_{\ell}=O(1/a^{\ell})$, and hence the integrands converge to 0 pointwise. By the dominate convergence theorem we see that, for fixed $s\in\CC$ with $\Real(s)>-\sigma_{0}$,
we have
\begin{displaymath}
	\lim_{a\to +\infty}g(a,s)=0.
\end{displaymath}
It is no longer clear that this assertion holds locally uniform in $s$. Nevertheless, by the uniform bound for $\widetilde{\rho}_{\ell}$, it is true that locally uniformly in $s$, the functions $g(a,s)$ are bounded independently of $a$. Since the $g(a,s)$ are holomorphic in $\Real(s)>-\sigma_{0}$, by the Cauchy integral representation and the dominate convergence theorem, all the derivatives of $g(a,s)$ converge to 0 as $a\to +\infty$, as well. We conclude that
\begin{equation}\label{eq:pain-2}
	\lim_{a\to +\infty}\frac{\partial}{\partial s}\mid_{s=0}2\frac{\sin(\pi s)}{\pi}g(a,s)=0.
\end{equation}
Taking into account \eqref{eq:pain-1}--\eqref{eq:pain-2}, we infer that the infinite sum 
\begin{displaymath}
\zeta_{N}(s):=\sum_{k\geq 1} N_{k}(s)
\end{displaymath}
is absolutely convergent for $1<\Real(s)<2$, admits a holomorphic continuation to $\Real(s)>-\sigma_{0}$, and 
satisfies
\begin{displaymath}
	\lim_{a\to +\infty}\zeta_{N}'(0)=0.
\end{displaymath}
For the polynomial terms, it is easily seen that $U_{n}(\tau)$ is divisible by $\tau^{n}$. It is then enough to deal with expressions of the form $\tau^{m}/t^{n}$, for $m\geq n\geq 2$. We distinguish two cases. The first and easiest is when $m\geq 3$. These are treated in an analogous manner as the remainder term. The argument is simpler since we have an explicit expression, contrary to the remainder $\widetilde{\rho}_{\ell}$. Second, the more cumbersome case $m=n=2$. The strategy in this case is similar to Proposition \ref{prop:E2-integral} below (that treats $U_{1}(\tau)$), and we leave the details to the reader.
\end{proof}

\subsection{Study of the integrals $M_{k}(s)$. Second part}\label{subsubsec:pseudo-laplacian-4}
The rest of the section focuses on the study of the infinite sum
\begin{displaymath}
	\zeta_{E}(s):=\sum_{k\geq 1} E_{k}(s)
\end{displaymath}
with $ E_{k}(s)$ given by \eqref{eq:41}.
This zeta function encapsulates the whole determinant of the pseudo-lapla\-cian.
\begin{theorem}\label{thm:E-integral-cusp}
The zeta function $\zeta_{E}(s)$ is absolutely and uniformly convergent for $1<\Real(s)<2$, and admits a meromorphic continuation to $\CC$. Furthermore, $\zeta_{E}(s)$ is holomorphic at $s=0$ and
\begin{displaymath}
	\zeta_{E}'(0)=4\pi\zeta(-1)a+\zeta(0)\log(a)+O\left(\frac{1}{a}\right),
\end{displaymath}
as $a\to+\infty$. Here, $\zeta(s)$ denotes the Riemann zeta function.
\end{theorem}
We structure the proof in three contributions, corresponding to a suitable decomposition of $E_{k}(s)$ as a sum of three integrals. First, in the definition \eqref{eq:41} of $E_{k}(s)$, we effect the change of variables $t=ku$. Then, after a straight-forward computation, one can write $E_{k}(s)=E_{0,k}(s)+E_{1,k}(s)+E_{2,k}(s)$, where the integrals $E_{\ell,k}(s)$ ($\ell=0,1,2$) are defined as
\begin{align}\label{def_Eellk}
	E_{\ell,k}(s)=2\frac{\sin(\pi s)}{\pi}\int_{k^{\delta-1}}^{\infty}((ku)^{2}-1/4)^{-s}\frac{\partial}{\partial u}\mathcal{E}_{\ell}(u,k)du,
\end{align}
and the integrands are explicitly given by the following formulas
\begin{align*}
	\frac{\partial}{\partial u}\mathcal{E}_0(u,k)&:= -\frac{1}{2}\frac{u}{u^2+ (2 \pi a)^2},\\
	\frac{\partial}{\partial u} \mathcal{E}_1(u,k)&:=k \log\left(\frac{u+\sqrt{u^2+(2 \pi a)^2} }{2 \pi a}\right),\\
	\frac{\partial}{\partial u} \mathcal{E}_2(u,k)&:=\frac{1}{k} 
\frac{u\left(u^2+ (2 \pi a)^2\right)^{-3/2}}{24}\left( 13
-15 u \left(u^2+ (2 \pi a)^2\right)^{-1}
\right).
\end{align*}
The three ``error integrals" given in \eqref{def_Eellk} are still absolutely and locally uniformly convergent 
for $s\in\CC$ with $1<\Real(s)<2$.  

\begin{proposition}\label{prop:E0-integral}
The infinite sum $\zeta_{E,0}(s)=\sum_{k} E_{0,k}(s)$ (with $E_{0,k}(s)$ given by \eqref{def_Eellk}) is absolutely and locally uniformly convergent 
for $s\in\CC$ with $1<\Real(s)<2$, and admits a meromorphic continuation to $\CC$. Furthermore, $\zeta_{E,0}(s)$ is holomorphic at $s=0$ and
\begin{displaymath}
	\zeta_{E,0}'(0)=\zeta(0)\log(a)+O\left(\frac{1}{a^{2}}\right),
\end{displaymath}
as $a\to +\infty$.
\end{proposition}
\begin{proof}
For $\Real(s)>1$, we have
\begin{displaymath}
	\zeta_{E,0}(s)=-\frac{\sin(\pi s)}{\pi}\sum_{k=1}^{\infty}\int\limits_{k^{\sigma-1}}^{\infty}
\left(k^2u^2-\frac{1}{4}\right)^{-s}
 \frac{u}{u^2+ (2\pi a)^2} du.
\end{displaymath}
For the range of validity of $u$, we can apply a binomial expansion.
We begin by writing
\begin{displaymath}
\left(k^2u^2-\frac{1}{4}\right)^{-s}=\sum\limits_{j=0}^{\infty}
\frac{(s)_j}{j! }\frac{1}{4^j}\frac{1}{k^{2s+2j}}\frac{1}{u^{2s+2j}},
\end{displaymath}
where $(s)_{j}=\Gamma(s+j)/\Gamma(s)=s(s+1)\ldots (s+j-1)$ is the standard notation for the Pochhammer symbol. This yields the expansion
\begin{displaymath}
	\zeta_{E,0}(s)=-\frac{\sin(\pi s)}{\pi}
\sum_{k=1}^{\infty}  \sum_{j=0}^{\infty}\frac{(s)_j}{j! }\frac{1}{4^j}
\frac{1}{k^{2s+2j}}
\int_{k^{\sigma-1}}^{\infty}
\frac{1}{u^{2s+2j-1}}
 \frac{1}{u^2+ (2\pi a)^2} du.
\end{displaymath}
The exchange of integral and sum is justified with ease. The integral in the above expression can be written in terms 
of hypergeometric functions:
\begin{align*}
\int\limits_{k^{\delta-1}}^{\infty}
&\frac{1}{u^{2s+2j-1}}\frac{1}{u^2+ (2\pi a)^2} du
=\\
&
\frac{1}{2(s+j-1)}  \frac{1}{(2\pi a)^{2}} \frac{1}{k^{2(\delta-1)(s+j-1)}}
 \cdot F(1,1-s-j;2-s-j;-(2\pi a)^{-2}k^{2\delta-2})\\
&-\frac{\pi}{2(-1)^{j}\sin(\pi s)}\frac{1}{(2\pi a)^{2(s+j)}}.
\end{align*}
We thus obtain 
\begin{equation}\label{eq:43}
	\begin{split}
	\zeta_{E,0}(s)
	=&-\frac{1}{(2\pi a)^{2}}\frac{\sin(\pi s)}{2\pi}\sum\limits_{k=1}^{\infty} \sum\limits_{j=0}^{\infty}\frac{(s)_j}{j! }\frac{1}{4^j}
	\frac{1}{(s+j-1)}\frac{1}{k^{2\delta(s+j)-2(\delta-1)}}\\
	&\hspace{3cm}\cdot F(1,1-s-j;2-s-j;-(2\pi a)^{-2}k^{2\delta-2})\\
	&-\sum\limits_{k=1}^{\infty} \sum\limits_{j=0}^{\infty}\frac{(-1)^{j}(s)_j}{j! }\frac{1}{2^{2j+1}}
	\frac{1}{k^{2(s+j)}}\frac{1}{(2\pi a)^{2(s+j)}}.
	\end{split}
\end{equation}
From this expansion, we see that $\zeta_{E,0}(s)$ is absolutely convergent for $1<\Real(s)<2$ and admits a meromorphic continuation to $\CC$. One can also deduce the following Laurent expansion at $s=0$:
\begin{displaymath}
	\begin{split}
		\zeta_{E,0}(s)
		&=-\frac{1}{2}\frac{\zeta(2s)}{(2\pi a)^{2s}}+f_{1}(a)+f_{2}(a)s+O(s^{2})\\
		&=(\zeta(0)^{2}+f_{1}(a))+(\zeta(0)\log(a)+f_{2}(a))s+O(s^{2}),
	\end{split}
\end{displaymath}
where $f_{1}(a)$ and $f_{2}(a)$ are $O(1/a^{2})$, as $a\to +\infty$. This concludes the proof.
\end{proof}

\begin{proposition}\label{prop:E1-integral}
The infinite sum $\zeta_{E,1}(s)=\sum_{k} E_{1,k}(s)$  (with $E_{1,k}(s)$ given by \eqref{def_Eellk})
is absolutely and locally uniformly convergent for $1<\Real(s)<2$, and admits a meromorphic continuation to $\CC$. Furthermore, $\zeta_{E,1}(s)$ is holomorphic at $s=0$ and
\begin{displaymath}
	\zeta_{E,1}'(0)=4\pi\zeta(-1) a+O\left(\frac{1}{a}\right),
\end{displaymath}
as $a\to +\infty$.
\end{proposition}

\begin{proof}
The main lines of the proof go as for Proposition \ref{prop:E0-integral}. We detail the important points. One writes
\begin{displaymath}
	\frac{\partial}{\partial u}\mathcal{E}_{1}(u,k)=k\arcsinh\left(\frac{u}{2\pi a}\right).
\end{displaymath}
Then, by a binomial expansion, one gets
\begin{displaymath}
	\zeta_{E,1}(s)=\frac{2\sin(\pi s)}{\pi}
	\sum_{k=1}^{\infty}  \sum\limits_{j=0}^{\infty}\frac{(s)_j}{j! }\frac{1}{4^j}
	\frac{1}{k^{2s+2j-1}}
	\int\limits_{k^{\delta-1}}^{\infty}
	\frac{1}{u^{2s+2j}}
	 \arcsinh(\frac{u}{2\pi a})du.
\end{displaymath}	
After a lengthy computation, one obtains a suitable expression for the $\arcsinh$ integrals:
\begin{align*}
\int\limits_{k^{\delta-1}}^{\infty}
&\frac{1}{u^{2s+2j}} \arcsinh(\frac{u}{2\pi a})du
=\\
&\frac{\pi}{(-1)^j\sin(\pi s)}\frac{\Gamma(s+j+\frac{1}{2})}{\sqrt{\pi}\,\Gamma(s+j)(1-2s-2j)^2 }
\frac{1}{(2\pi a)^{2s+2j-1}}\\
&-\frac{1}{2\pi a} 
\frac{1}{2(2s+2j-1)(s+j-1)}
\cdot\frac{F\left(\frac{1}{2},1-s-j;2-s-j;-\frac{1}{(2\pi a)^2 k^{2(1-\delta)} }\right)}{k^{(\delta-1)(2s+2j-2)}}\\
&-\frac{1}{2\pi a}
\frac{1}{2s+2j-1}\frac{1}{k^{(\delta-1)(2s+2j-1)}}
\cdot
\sum\limits_{r=0}^{\infty}\frac{(-1)^r (2r)!}{2^{2r}(2r!)^2 (2 r+1)} 
\frac{1}{(2\pi a)^{2r}} \frac{1}{k^{(1-\delta)(2r+1)}}.
\end{align*}
We insert the values of the integrals in the expression for the zeta function, after observing that we can exchange integration and sum. We find
\begin{displaymath}
\begin{split}
\zeta_{E,1}(s)
=&
 \sum\limits_{j=0}^{\infty}\frac{(-1)^j (s)_j}{j! }\frac{1}{2^{2j-1}}
\frac{\Gamma(s+j+\frac{1}{2})}{\sqrt{\pi}\,\Gamma(s+j)(1-2s-2j)^2 }
\frac{\zeta(2s+2j-1)}{(2\pi a)^{2s+2j-1}}\\
&-\frac{1}{2\pi a}\frac{2\sin(\pi s)}{\pi}
\sum_{k=1}^{\infty}  \sum\limits_{j=0}^{\infty}\frac{(s)_j}{j! }\frac{1}{4^j}
\frac{1}{2s+2j-1}\\
&\hspace{3.5cm}\cdot\sum\limits_{r=0}^{\infty}\frac{(-1)^r (2r)!}{2^{2r}(2r!)^2 (2 r+1)}
\cdot\frac{1}{(2\pi a)^{2r}}\frac{1}{k^{\delta(2s+2j-1)+(1-\delta)(2r+1)}} \\
&-\frac{1}{2\pi a}\frac{2\sin(\pi s)}{\pi}
\sum_{k=1}^{\infty}  \sum\limits_{j=0}^{\infty}\frac{(s)_j}{j! }\frac{1}{4^j}
\frac{F\left(\frac{1}{2},1-s-j;2-s-j;-\frac{1}{(2\pi a)^2 k^{2(1-\delta)} }\right)}{2(2s+2j-1)(s+j-1)}\\
&\hspace{3.5cm}\cdot\frac{1}{k^{\delta(2s+2j-2)+1}}.
\end{split}
\end{displaymath}
From this expansion, one can conclude the absolute convergence on $1<\Real(s)<2$, the meromorphic continuation to $\CC$, and deduce the following Laurent expansion at $s=0$:
\begin{align*}
\zeta_{E,1}(s)&=\frac{2\Gamma(s+\frac{1}{2})}{\sqrt{\pi}\Gamma(s)(1-2s)^2 }\frac{\zeta(2s-1)}{a^{2s-1}}+
f(a)\cdot s +O(s^2)\\
&=(4 \zeta(-1)\pi  a +f(a) )s +O(s^2),
\end{align*}
with $f(a)=O(1/a)$, as $a\to +\infty$, as was to be shown.
\end{proof}

\begin{proposition}\label{prop:E2-integral}
The infinite sum $\zeta_{E,2}(s)=\sum_{k}E_{2,k}(s)$ (with $E_{2,k}(s)$ given by \eqref{def_Eellk}) is absolutely and locally uniformly convergent for $1<\Real(s)<2$, and admits a meromorphic continuation to $\CC$. Furthermore, $\zeta_{E,2}(s)$ is holomorphic at $s=0$ and
\begin{displaymath}
	\zeta_{E,2}'(0)=O\left(\frac{1}{a}\right),
\end{displaymath}
as $a\to+\infty$.
\end{proposition}

\begin{proof}
The proof is similar to the proof propositions \ref{prop:E0-integral}--\ref{prop:E1-integral}. For the convenience of the reader, we give the relevant input in the computation, namely an explicit integral evaluation under a suitable form:
\begin{displaymath}
\int\limits_{0}^{u}\frac{t^{\mu-1}}{(t+1)^{\nu}}dt=\frac{u^{\mu-\nu}}{\mu-\nu} 
 F(\nu,\nu-\mu;-\mu+\nu+1;-u^{-1})
+
\frac{\Gamma(\mu)\Gamma(\nu-\mu)}{\Gamma(\nu)},
\end{displaymath}
where $u,\nu>0$, and $\mu$ is complex with $\Real(\mu)>0$.
\end{proof}
\begin{proof}[Proof of Theorem \ref{thm:E-integral-cusp}]
The theorem is the conjunction of propositions \ref{prop:E0-integral}--\ref{prop:E2-integral}.
\end{proof}

\subsection{Conclusion: proofs of theorems \ref{thm:det-pseudo-laplacian} and \ref{theorem:zeta-pseudo-laplacian}}\label{subsubsec:pseudo-laplacian-conclusion}
After the long discussion \textsection\ref{subsubsec:pseudo-laplacian-1}--\textsection\ref{subsubsec:pseudo-laplacian-4}, we conclude that the spectral zeta function $\zeta_{\ps}(s)$ of the pseudo-laplacian $\Delta_{\ps}$ on the model cusp $\C_{a}$, for $a>0$, is absolutely and locally uniformly convergent 
for $\Real(s)>1$, admits a meromorphic continuation to $\Real(s)>-\sigma_{0}$ for some $\sigma_{0}>0$, and is holomorphic at $s=0$. Moreover, its derivative at $s=0$ satisfies the equality
\begin{displaymath}
	\zeta_{\ps}'(0)=4\pi\zeta(-1)a+\zeta(0)\log(a)+o(1),
\end{displaymath}
as $a\to +\infty$.
Equivalently, we have
\begin{displaymath}
	\log\det(\Delta_{\ps})=-4\pi\zeta(-1)a-\zeta(0)\log(a)+o(1),
\end{displaymath}
as $a\to +\infty$.
This concludes the proof of Theorem \ref{theorem:zeta-pseudo-laplacian} and hence of Theorem \ref{thm:det-pseudo-laplacian}.

\section{Evaluation of the determinant of the Dirichlet laplacian on a model cone}\label{subsection:determinant-cone}
We proceed to study the spectral zeta function for the model cone \textsection\ref{subsec:model-cone}, abutting to the computation of the zeta regularized determinant (Theorem \ref{thm:det_cone}). The strategy and methods are analogous to the case of the model cusp treated in Section \ref{section:evaluation-determinants}. Therefore, we won't provide full details of the computations. Nevertheless, we summarize the steps where some care is needed.

\subsection{The spectral zeta function of $\Delta_{D}$}
We parametrize the model hyperbolic cone of angle $2\pi/\omega$ by $\E_{\eta}=(0,\eta]\times [0,2\pi]$, with coordinates $(\rho,\theta)$, so that $\rho\to 0$ corresponds to the appex of the cone. We discussed the eigenvalue problem for $\Delta_{D}$ with Dirichlet boundary condition, reducing to the implicit equation
\begin{equation}\label{eq:eigen-cone}
	P_{1/2+ir}^{-|k|\omega}(\cosh(\eta))=0.
\end{equation}
By Lemma \ref{lemma:min-eigen-cone}, the spectrum of $\Delta_{D}$ is contained in $(1/4;+\infty)$. By Proposition \ref{prop:eigen-cone}, it can be divided in two packages: i) simple eigenvalues $\lambda_{0,n}$, solutions to \eqref{eq:eigen-cone} with $k=0$; ii) eigenvalues $\lambda_{k,n}$, of multiplicity 2, solutions to \eqref{eq:eigen-cone} with $k>0$. 

Proceeding exactly as in \textsection\ref{subsubsec:pseudo-laplacian-1},
the spectral zeta function $\zeta_{\omega}(s)$ for $\Delta_{D}$
can be represented in the region $1<\Real(s)<2$ as an absolutely and locally uniformly convergent series
\begin{displaymath}
\zeta_{\omega}(s)=\frac{\sin(\pi s)}{\pi}\sum_{k\geq 0} d_{k}\int_{1/2}^{\infty}(t^{2}-1/4)^{-s} f_{k}(t) dt,
\end{displaymath}
where $d_{0}=1$, $d_{k}=2$ for $k\geq 1$, and the function $f_k$
is defined (analogously to \eqref{def:26a} in \textsection\ref{subsubsec:pseudo-laplacian-1}) by
\begin{equation}\label{eq:fk-cone}
	f_{k}(t)=\frac{\partial}{\partial t}\log P_{-1/2+t}^{-k\omega}(\cosh(\eta))-2t\frac{\partial}{\partial t}\mid_{t=1/2}\log P_{-1/2+t}^{-k\omega}(\cosh(\eta)).
\end{equation}
Again, it will be necessary to divide the interval $(1/2,+\infty)$. More precisely, 
introducing a small parameter $0<\delta\ll 1$, we
 use the splitting
\begin{align*}
	\int_{1/2}^{\infty}=\int_{1/2}^{1}+\int_{1}^{\infty},\text{ if } &k=0;\\
	\int_{1/2}^{\infty}=\int_{1/2}^{k^{\delta}}+\int_{k^{\delta}}^{\infty},\text{ if } &k\geq 1.
\end{align*}
Accordingly, we obtain a decomposition of $\zeta_{\omega}(s)=\zeta_{L}(s)+\zeta_{M}(s)$, valid on the region $1<\Real(s)<2$.

\subsection{Negligible contributions}
We begin by isolating the negligible contributions (as $\eta\to 0$) to the spectral zeta function.
\begin{proposition}\label{prop:cone-zeta-L}
The infinite sum $\zeta_{L}(s)$ is absolutely and locally uniformly convergent for $1<\Real(s)<2$, 
and admits a meromorphic continuation to $\Real(s)>-\sigma_{0}$, 
for some $\sigma_{0}>0$. Furthermore, $\zeta_{L}(s)$ is holomorphic at $s=0$ and 
\begin{displaymath}
	\lim_{\eta\to 0}\zeta_{L}'(0)=0.
\end{displaymath}
\end{proposition}
\begin{proof}
We invoke the following expression for the Legendre function in terms of the hypergeometric function:
\begin{displaymath}
	\begin{split}
	P_{-1/2+t}^{-k\omega}(\cosh(\eta))=\frac{1}{\Gamma(kw+1)}&\left(\frac{\cosh(\eta)-1}{\cosh(\eta)+1}\right)^{k\omega/2}
	 F\left(\frac{1}{2}+t,\frac{1}{2}-t;k\omega+1;\frac{1-\cosh(\eta)}{2}\right).
	\end{split}
\end{displaymath}
The quantity $\xi:=(1-\cosh(\eta))/2$ tends to 0 as $\eta$ converges to 0, and hence we can use the series representation for the hypergeometric function:
\begin{displaymath}
	F(1/2+t,1/2-t;k\omega+1;\xi)
	=\sum_{n\geq 0}\frac{(1/2+t)_{n} (1/2-t)_{n}}{(k\omega+1)_{n}}\frac{\xi^{n}}{n!}.
\end{displaymath}
From this one deduce the special value
\begin{displaymath}
	\frac{\partial}{\partial t}\mid_{t=1/2} P_{-1/2+t}^{-k\omega}(\cosh(\eta))=-\frac{\xi}{k\omega+1}F(1,1;k\omega+2;\xi).
\end{displaymath}
As in \textsection\ref{subsubsec:pseudo-laplacian-2} and the proof of Theorem \ref{thm:L-integral-cusp}, one proceeds by integration by parts. As a primitive for $f_{k}(t)$, one takes
\begin{displaymath}
	F_{k}(t)=\log F(1/2+t,1/2-t;k\omega+1;\xi)+
	(t^{2}-1/2)\frac{\xi}{k\omega+1}F(1,1;k\omega+2;\xi).
\end{displaymath}
Now, the analogue role to the asymptotics of Proposition \ref{prop:asymptotics-bessel-1} is provided by \cite{Wagner}. For instance, if $k\geq 1$, in the interval $[1/2,k^{\delta}]$ we have $t^{2}=o(k)$ uniformly in $k$, by the choice of small $\delta$. Then \emph{loc.~cit.}~applies and produces an approximation with remainder
\begin{displaymath}
	F(1/2+t,1/2-t;k\omega+1;\xi)=
	\sum_{n=0}^{\ell-1}\frac{(1/2+t)_{n} (1/2-t)_{n}}{(k\omega+1)_{n}}\frac{\xi^{n}}{n!}
	+\rho_{\ell}(t,k)\frac{(1/2+t)_{\ell} (1/2-t)_{\ell}}{(k\omega+1)^{\ell}}\xi^{\ell},
\end{displaymath}
where $\rho_{\ell}(t,k)$ is uniformly bounded on $[1/2,k^{\delta}]$ (in the numerator of the remainder, $\ell$ is indeed an exponent). Also, observe that in this interval
\begin{displaymath}
	\frac{(1/2+t)_{\ell}(1/2-t)_{\ell}}{(k\omega+1)^{\ell}}=O\left(\frac{1}{k^{\ell-2\ell\delta}}\right).
\end{displaymath}
A similar computation as in the proof of Theorem \ref{thm:L-integral-cusp} allows to conclude. The details are left to the reader.
\end{proof}
Next, we decompose the function $\zeta_{M}(s)$ into 
\begin{equation}\label{splitdef_widetildem}
\zeta_{M}(s)=\zeta_{\widetilde{M}}(s)+\zeta_{R}(s),
\end{equation}
corresponding to the two terms defining the function $f_{k}$ in \eqref{eq:fk-cone}. 
\begin{proposition}
The infinite sum $\zeta_{R}(s)$ is absolutely and locally uniformly convergent for $1<\Real(s)<2$,
and admits a meromorphic continuation to $\CC$.
Furthermore, $\zeta_{R}(s)$ is holomorphic at $s=0$ and 
\begin{displaymath}
	\lim_{\eta\to 0}\zeta_{R}^{\prime}(0)=0.
\end{displaymath}
\end{proposition}

\begin{proof}
For instance, for $k\geq 1$, we have to treat the integral
\begin{displaymath}
	\begin{split}
	\frac{\sin(\pi s)}{\pi}\frac{F(1,1;k\omega+2;\xi)}{k\omega+1}&\int_{k^{\delta}}^{\infty}(t^{2}-1/4)^{-s} 2t \frac{\xi}{k\omega+1} dt=\\
	&\frac{\sin(\pi s)\xi}{\pi(s-1)}\frac{F(1,1;k\omega+1;\xi)}{k\omega+1}\frac{1}{(k^{2\delta}-1/4)^{s-1}}.
	\end{split}
\end{displaymath}
We proceed similarly for the integral with $k=0$. One concludes by considering the series expansion of the hypergeometric function. The details are left to the reader.
\end{proof}
\subsection{The main contribution}
The main contribution in the computation of the determinant of the Dirchlet laplacian on the hyperbolic cone will all come from 
the function (defined in \eqref{splitdef_widetildem})
\begin{displaymath}
	\begin{split}
		\zeta_{\widetilde{M}}(s)=&\frac{\sin(\pi s)}{\pi}\int_{1}^{\infty}(t^{2}-1/4)^{-s}\frac{\partial}{\partial t}\log P_{-1/2+t}^{0}(\cosh(\eta)) dt\\
		&+2\frac{\sin(\pi s)}{\pi}\sum_{k\geq 1}\int_{k^{\delta}}^{\infty}(t^{2}-1/4)^{-s}\frac{\partial}{\partial t}\log P_{-1/2+t}^{-k\omega}(\cosh(\eta))dt.
	\end{split}
\end{displaymath}
It will be convenient to treat the two lines separately, hence we further split 
\begin{equation}\label{splitdef_widetilde_mzeroandone}
\zeta_{\widetilde{M}}(s)=\zeta_{\widetilde{M}_{0}}(s)+\zeta_{\widetilde{M}_{1}}(s),
\end{equation}
accordingly (corresponding to $k=0$ and $\sum_{k\geq 1}$). 

We begin by considering the function $\zeta_{\widetilde{M}_{1}}(s)$. The strategy is parallel to the cusp case. We first seek asymptotics that play the role of Proposition \ref{prop:asymptotics-bessel-2}. The result we need is an adaptation of \cite{Khus}.  Following the method of \emph{loc.~cit.}~one obtains the asymptotics, which are \emph{uniform} in $u>0$:
\begin{displaymath}
P^{- k\omega}_{-\frac{1}{2}+k\omega u}(\cosh(\eta)) \sim
\frac{1}{\Gamma(k\omega+1)}
\frac{( 1-u^2 v^{2} )^{1/4}}{ (1-u^2)^{1/4}} e^{k\omega\cdot  S_{-1}(v)}
\sum\limits_{n=0}^{\infty}\frac{\psi_n(v)}{(k\omega)^{n}}.
\end{displaymath}
Let us clarify the notations and make some comments afterwards. First, we have set
\begin{displaymath}
	v=\frac{\cosh(\eta)}{\sqrt{1+u^2 \sinh(\eta)^2}}.
\end{displaymath}
The function $S_{-1}(v)$ is defined as
\begin{displaymath}
	S_{-1}(v)=
	\frac{1}{2} \log\left(\frac{1-v}{(1+v)(u^2-1)}\right)-u \log(u+1)
	+ u\log\left(u\cosh(\eta)+  v^{-1}\cosh(\eta)  \right).
\end{displaymath}
The notation $S_{-1}$ has been chosen in order to ease the comparison with \cite{Khus}. The $\psi_{n}(v)$ are functions obeying a recurrence scheme. Here we content ourselves by giving the first ones and quoting relevant properties for our aim. The first is $\psi_{0}\equiv 1$, and the second one
\begin{displaymath}
	\psi_{1}(v)=\frac{1}{8(1-u^2)}\cdot \left(\frac{5 u^{2} v^{3} -3u^{2} v-2u^{2}}{3}+\left(1-v\right)\right).
\end{displaymath}
The functions $\psi_{n}$ are uniformly bounded. Furthermore, for fixed $u$, as $\eta\to 0$ (and hence $v\to 1$), they converge to 0 for any $n\geq 1$. Finally, despite the shape of the formulas we wrote, that contain denominators $1-u^{2}$, the very definition of $v$ makes them well-defined for $u>0$. Because of the uniformity of the expansion in $u$ and the properties of the functions $\psi_{n}$, we conclude that we have an expansion of the form
\begin{equation}\label{eq:asym-P}
	P^{- k\omega}_{-\frac{1}{2}+k\omega u}(\cosh(\eta)) =
	\frac{1}{\Gamma(k\omega+1)}
	\frac{( 1-u^2 v^{2} )^{1/4}}{ (1-u^2)^{1/4}} e^{k\omega\cdot  S_{-1}(v)}
	\cdot\left(\sum_{n=0}^{\ell-1}\frac{\psi_{n}(v)}{(k\omega)^{n}}+\frac{\rho_{\ell}(u,\eta)}{(k\omega)^{\ell}}\right),
\end{equation}
where $\rho_{\ell}$ is bounded and, for fixed $u$, $\rho_{2}(u,\eta)\to 0$ as $\eta\to 0$. The reader will realize that this expansion is very similar to the one stated and used for Bessel functions in \textsection\ref{subsubsec:M-int-1}.

To use the asymptotics \eqref{eq:asym-P}, we effect the change of variables $t=uk\omega$ in the integrals. Similar to the cusp case, we guess a splitting 
\begin{displaymath}
\zeta_{\widetilde{M}_{1}}(s)=\zeta_{M^{\ast}}(s)+\zeta_{E}(s). 
\end{displaymath}
The function $\zeta_{E}(s)$ is defined analogously to \textsection\ref{subsubsec:pseudo-laplacian-4}. Using similar notations, one can write $E_{k}(s)=E_{0,k}(s)+E_{1,k}(s)+E_{2,k}(s)$,
where the integrals $E_{\ell,k}(s)$ ($\ell=0,1,2$) now involve the 
following integrands (cf.~\eqref{def_Eellk})
\begin{align*}
	\frac{\partial}{\partial u} \mathcal{E}_0(u,k)&:= -\frac{1}{2}\frac{u}{u^2 +\sinh(\eta)^{-2}},\\
	\frac{\partial}{\partial u} \mathcal{E}_1(u,k)&:=k \omega \cdot\left( -\log(u+1)+ \log\left(u\cosh(\eta)   +\sqrt{1+u^2 \sinh(\eta)^2}  \right)\right),\\
	\frac{\partial}{\partial u}\mathcal{E}_2(u,k)&:=\frac{1}{k\omega}(A_{1}(u,k)+A_{2}(u,k)
	+A_{3}(u,k))
\end{align*}
with
\begin{align*}
A_1(u,k)&:=\frac{\bigl(  5-\cosh(\eta)^2\bigr)\cosh(\eta)\sinh(\eta)^2 }
{8}  \frac{ u}{(1+u^2\sinh(\eta)^2)^{5/2}},\\
A_2(u,k)&:=
-\frac{\cosh(\eta)\sinh(\eta)^4}{8}  \frac{ u^3}{(1+u^2\sinh(\eta)^2)^{5/2}},\\ 
A_3(u,k)&:=
\frac{u}{24(u^2-1)^2 }
\left(
2
+\frac{3\cosh(\eta)^5 }{(1+u^2\sinh(\eta)^2)^{5/2}}
-\frac{5\cosh(\eta)^{3}}{ (1+u^2\sinh(\eta)^2)^{3/2}}
\right).
\end{align*}
The integrals defining $\zeta_{M^{\ast}}(s)$ are the rest and, by the following proposition, they don't contribute to the value at $s=0$, as $\eta\to 0$.
\begin{proposition}
The infinite sum $\zeta_{M^{\ast}}(s)$ is absolutely and locally uniformly convergent for $1<\Real(s)<2$, 
and admits a meromorphic continuation to $\Real(s)>-\sigma_{0}$, 
for some $\sigma_{0}>0$. 
Furthermore, $\zeta_{M^{\ast}}(s)$ is holomorphic at $s=0$ and 
\begin{displaymath}
	\lim_{\eta\to 0}\zeta_{M^{\ast}}^{\prime}(0)=0.
\end{displaymath}
\end{proposition}

\begin{proof}
The proof is analogous to the proof of Theorem \ref{thm:Mast-integral-cusp}, using integration by parts and then taking into account the shape of the properties of the remainder \eqref{eq:asym-P}.
\end{proof} 

We can finally focus on $\zeta_{E}(s)$, that again we separate into $\zeta_{E_{0}}(s)$, $\zeta_{E_{1}}(s)$, and $\zeta_{E_{2}}(s)$, according to the $\E_{j}$.
\begin{proposition}
The infinite sum $\zeta_{E_{0}}(s)$ is absolutely and locally uniformly convergent for $1<\Real(s)<2$, and admits a meromorphic continuation to $\CC$. Furthermore, $\zeta_{E_{0}}(s)$
is holomorphic at $s=0$ and
\begin{displaymath}
	\zeta^{\prime}_{E_{0}}(0)=\zeta(0)\log(\omega)-\zeta^{\prime}(0)-\zeta(0)\log(\eta)+o(1),
\end{displaymath}
as $\eta\to 0$.
\end{proposition}

\begin{proof}
The prove goes along the lines of the proof of Proposition \ref{prop:E0-integral},
starting with the binomial expansion of the term $(k^{2}\omega^{2}u^{2}-1/4)^{-s}$.
\end{proof}

The treat the ``error integral" $\zeta_{E_{1}}(s)$ is much harder; we obtain the following:
\begin{proposition}
The infinite sum $\zeta_{E_{1}}(s)$ is absolutely and locally uniformly convergent for $1<\Real(s)<2$, and admits a meromorphic continuation to $\Real(s)>-1/2$. Furthermore, $\zeta_{E_{1}}(s)$ is holomorphic at $s=0$ and
\begin{displaymath}
\zeta_{E_{1}}^{\prime}(0)=2\omega\bigl(-\zeta(-1) -\zeta^{\prime}(-1)+\zeta(-1)\log(2\omega)\bigr)
+\frac{\omega}{6}\log(\eta)
+o(1),
\end{displaymath}
as $\eta\to 0$.
\end{proposition}

\begin{proof}
After a binomial expansion of the term $(k^{2}\omega^{2}u^{2}-1/4)^{-s}$, we split $\zeta_{E_{1}}(s)$ into six contributions, whose expressions are valid in $\Real(s)>1$, namely 
$\zeta_{E_{1}}(s)=\sum_{\ell=1}^{6} \Xi_{\ell}(s)$ with 
\begin{align*}
\Xi_{0}(s)
&:=\frac{2\,\sin(\pi s)}{\pi}
\sum_{k=1}^{\infty}  \sum\limits_{j=0}^{\infty}\frac{(s)_j}{j! }\frac{1}{4^j}
\frac{1}{k^{2s+2j-1}}
\frac{1}{\omega^{2s+2j-1}}
\int\limits_{k^{\delta-1}\omega^{-1}}^{1}
\frac{\log\left(\sqrt{1+u^2 \sinh(\eta)^2}\right)}{u^{2s+2j}}du,\\
\Xi_{1}(s)
&:=-\frac{2\,\sin(\pi s)}{\pi}
\sum_{k=1}^{\infty}  \sum\limits_{j=0}^{\infty}\frac{(s)_j}{j! }\frac{1}{4^j}
\frac{1}{k^{2s+2j-1}}
\frac{1}{\omega^{2s+2j-1}}
\int\limits_{k^{\delta-1}\omega^{-1}}^{1}
\frac{\log(u+1) }{u^{2s+2j}}du,\\
\Xi_{2}(s)
&:=\frac{2\,\sin(\pi s)}{\pi}
\sum_{k=1}^{\infty}  \sum\limits_{j=0}^{\infty}\frac{(s)_j}{j! }\frac{1}{4^j}
\frac{1}{k^{2s+2j-1}}
\frac{1}{\omega^{2s+2j-1}}
\int\limits_{k^{\delta-1}\omega^{-1}}^{1}
\frac{1}{u^{2s+2j}} \log\left(1 +\frac{u\cosh(\eta)}{\sqrt{1+u^2 \sinh(\eta)^2}} \right)du,
\end{align*}
and 
\begin{align*}
\Xi_{3}(s)
&:=\frac{2\,\sin(\pi s)}{\pi}
\sum_{k=1}^{\infty}  \sum\limits_{j=0}^{\infty}\frac{(s)_j}{j! }\frac{1}{4^j}
\frac{1}{k^{2s+2j-1}}
\frac{1}{\omega^{2s+2j-1}}
\int\limits_{1}^{\infty}
\frac{\log(u)-\log(u+1) }{u^{2s+2j}}du,\\
\Xi_{4}(s)
&:=\frac{2\,\sin(\pi s)}{\pi}
\sum_{k=1}^{\infty}  \sum\limits_{j=0}^{\infty}\frac{(s)_j}{j! }\frac{1}{4^j}
\frac{1}{k^{2s+2j-1}}
\frac{1}{\omega^{2s+2j-1}}
\int\limits_{1}^{\infty}
\frac{\log\left(\cosh(\eta)\right)}{u^{2s+2j}}du,\\
\Xi_{5}(s)
&:=\frac{2\,\sin(\pi s)}{\pi}
\sum_{k=1}^{\infty}  \sum\limits_{j=0}^{\infty}\frac{(s)_j}{j! }\frac{1}{4^j}
\frac{1}{k^{2s+2j-1}}
\frac{1}{\omega^{2s+2j-1}}
\int\limits_{1}^{\infty}
\frac{1}{u^{2s+2j}} \log\left(1 +\frac{\sqrt{1+u^2 \sinh(\eta)^2}}{u\cosh(\eta)} \right)du.
\end{align*}

First, one treats the integrals on bounded intervals. For this, we expand the logarithms in power series, which is justified in the intervals of integration and because $0<\sinh(\eta)\ll 1$ for $0<\eta\ll 1$. Similarly, we handle the square root in $\Xi_{2}$ (which is expanded only after the $\log$). Next, one computes
\begin{align*}
\int\limits_{k^{\delta-1}\omega^{-1}}^{1}
\frac{u^{r}}{u^{2s+2j}}du
=-\frac{1}{2s+2j-r-1}+\frac{k^{(1-\delta) (2s+ 2j-r-1)} \omega^{2s+ 2j-r-1}}{2s+2j-r-1},
\end{align*}
which is legitimate for $\Real(s)>1$. One plugs this expression (with appropriate powers of $r$) into the representation of $\Xi_{0}$, $\Xi_{1}$, and $\Xi_{2}$ in sums of integrals of power series. 
One obtains explicit expressions such as
\begin{align*}
\Xi_{0}(s)
=&-\frac{\sin(\pi s)}{\pi}
\sum_{k=1}^{\infty}  \sum\limits_{j=0}^{\infty}\frac{(s)_j}{j! }\frac{1}{4^j}
\frac{1}{k^{2s+2j-1}}
\frac{1}{\omega^{2s+2j-1}}
\sum\limits_{r=1}^{\infty} \frac{(-1)^{r+1}}{r}
\frac{\sinh(\epsilon)^{2r}}{2s+2j-2r-1}
\notag\\
&+
\frac{\sin(\pi s)}{\pi}
\sum_{k=1}^{\infty}  \sum\limits_{j=0}^{\infty}\frac{(s)_j}{j! }\frac{1}{4^j}
\sum\limits_{r=1}^{\infty} \frac{(-1)^{r+1}}{r}\frac{\sinh(\epsilon)^{2r}}{\omega^{2r}}
\frac{k^{-2r}k^{-\delta(2s+ 2j-2r-1)}}{2s+2j-2r-1}.
\end{align*}
In this whole procedure, one first needs to restrict to $1<\Real(s)<3/2$, then one observes that the final expressions are absolutely summable for $\Real(s)>1$ and have a meromorphic continuation to $\Real(s)>-1/2$. One can then easily derive the Laurent expansion of $\Xi_{0}(s)$ at $s=0$ alone. For the rest, the sum $\Xi_{1}(s)+\Xi_{2}(s)$ is best understood, and one can exhibit the Laurent expansion at $s=0$, too. The result is of the form
\begin{displaymath}
	\Xi_{0}(s)+\Xi_{1}(s)+\Xi_{2}(s)=f(\eta)+g(\eta)s +O(s^{2}),
\end{displaymath}
where $f$ and $g$ are explicit functions converging to 0 as $\eta\to 0$. We thus conclude that the integrals over the bounded intervals won't contribute.

Second, we elaborate on the integrals on the unbounded interval $[1,+\infty)$. For $\Xi_{3}$, after a change of variables $u=t^{-1}$, one has
\begin{displaymath}
	\int\limits_{1}^{\infty}
\frac{\log(u)-\log(u+1) }{u^{2s+2j}}du=- \frac{\log(2)}{2s+2j-1}+ \frac{\psi(s+j+1/2)-\psi(s+j)}{2(2s+2j-1)},
\end{displaymath}
where $\psi(s)=\Gamma'(s)/\Gamma(s)$ is the digamma function. Then we insert this relation into $\Xi_{3}(s)$. By recalling that the $\psi$ function has poles at $0,-1,-2,\ldots$, one can then justify that $\Xi_{3}(s)$ has a meromorphic continuation to $\Real(s)>-1/2$. For the Laurent expansion at $s=0$, we first obtain
\begin{align*}
\Xi_{3}(s)= 
&-\frac{2\sin(\pi s)}{\pi}
\frac{\zeta(2s-1)}{\omega^{2s-1}}
 \frac{\log(2)}{2s-1} 
 +
\frac{\sin(\pi s)}{\pi}
\frac{\zeta(2s-1)}{\omega^{2s-1}}
\frac{\psi(s+1/2)-\psi(s)}{2s-1}\\
&
 -\frac{s\sin(\pi s)}{2\pi}
\frac{\zeta(2s+1)}{\omega^{2s+1}}
 \frac{\log(2)}{2s+1}
 +
\frac{s\sin(\pi s)}{4\pi}
\frac{\zeta(2s+1)}{\omega^{2s+1}}
\frac{\psi(s+3/2)-\psi(s+1)}{2s+1}+O(s^2).
\end{align*}
From this expression it is a matter of further expanding at $s=0$, to finally find
\begin{displaymath}
\begin{split}
\Xi_{3}(s)=-\zeta(-1)\omega +
\Bigl(&2\omega\bigl(\zeta(-1)\log(4)-\zeta(-1) -\zeta^{\prime}(-1)\bigr)\\
&+\frac{1}{\omega}\bigl(\zeta(0)^2+\zeta(0)\log(2)\bigr)
+2\omega \log(\omega)\zeta(-1)
\Bigr) \cdot s
+O(s^2).
\end{split}
\end{displaymath}
The sum $\Xi_{4}(s)$ is easier to deal with, and we just reproduce the result: it has a meromorphic continuation to $\Real(s)>-1/2$ and a Laurent expansion at $s=0$ of the form
\begin{displaymath}
	\Xi_{4}(s)=f(\eta)s+O(s^{2}),
\end{displaymath}
with some explicit function $f(\eta)=o(1)$ as $\eta\to 0$. The last efforts are for $\Xi_{5}(s)$, that after some work (expansion in power series, computation of the integral in terms of an hypergeometric function, and use of functional identities for hypergeometric functions) is seen to have a meromorphic continuation to the region $\Real(s)>-1/2$.
Further, in a neighborhood of $s=0$, we have
\begin{displaymath}
\Xi_{5}(s)
=\left(\frac{\zeta(0)+\log(2)}{2\omega }+
\frac{\log(2)}{6 }\omega 
+\frac{\log(\eta)}{6}\omega
+f(\eta)\right) s
+
O(s^{2}),
\end{displaymath}
with some explicit function $f(\eta)=o(1)$ as $\eta\to 0$. Finally, adding up all the values, 
we can deduce the assertion.
\end{proof}

The last ``error integral" is $\zeta_{E_{2}}(s)$, for which we obtain the following:
\begin{proposition}
The  infinite sum $\zeta_{E_{2}}(s)$ is absolutely and locally uniformly convergent for $1<\Real(s)<2$,
and admits a meromorphic to $\CC$. Futhermore, $\zeta_{E_{2}}^{\prime}(0)$ is holomorphic at $s=0$ and
\begin{displaymath}
	\zeta_{E_{2}}^{\prime}(0)=-\frac{\zeta(-1)}{\omega}(5+2 \gamma-\log(4)-2 \log(\omega))+\frac{1}{6\omega} \log(\eta)+o(1),
\end{displaymath}
as $\eta\to 0$.
\end{proposition}
\begin{proof}
Again, after a binomial expansion, we can split $\zeta_{E_{2}}(s)$ into the sum of four contributions, namely $\zeta_{E_{2}}(s)=\sum_{\ell=1}^{4} \Omega_{\ell}(s)$ with 
\begin{align*}
&\Omega_{1}(s):=
\frac{2\,\sin(\pi s)}{\pi}
\sum_{k=1}^{\infty}  \sum\limits_{j=0}^{\infty}\frac{(s)_j}{j! }\frac{1}{4^j}
\frac{1}{k^{2s+2j+1}}\frac{1}{\omega^{2s+2j+1}}
\int\limits_{k^{\delta-1}\omega^{-1}}^{\infty}
\frac{1}{u^{2s+2j}}A_1(u,k)du,\\
&\Omega_{2}(s):=
\frac{2\,\sin(\pi s)}{\pi}
\sum_{k=1}^{\infty}  \sum\limits_{j=0}^{\infty}\frac{(s)_j}{j! }\frac{1}{4^j}
\frac{1}{k^{2s+2j+1}}\frac{1}{\omega^{2s+2j+1}}
\int\limits_{k^{\delta-1}\omega^{-1}}^{\infty}
\frac{1}{u^{2s+2j}}A_2(u,k)du,
\end{align*}
and 
\begin{align*}
&\Omega_{3}(s):=
\frac{2\,\sin(\pi s)}{\pi}
\sum_{k=1}^{\infty}  \sum\limits_{j=0}^{\infty}\frac{(s)_j}{j! }\frac{1}{4^j}
\frac{1}{k^{2s+2j+1}}\frac{1}{\omega^{2s+2j+1}}
\int\limits_{k^{\delta-1}\omega^{-1}}^{1}
\frac{1}{u^{2s+2j}}A_{3}(u,k)du,\\
&\Omega_{4}(s):=
\frac{2\,\sin(\pi s)}{\pi}
\sum_{k=1}^{\infty}  \sum\limits_{j=0}^{\infty}\frac{(s)_j}{j! }\frac{1}{4^j}
\frac{1}{k^{2s+2j+1}}\frac{1}{\omega^{2s+2j+1}}
\int\limits_{1}^{\infty}
\frac{1}{u^{2s+2j}}A_{3}(u,k)du.
\end{align*}

The first two functions, $\Omega_{1}(s)$ and $\Omega_{2}(s)$, can be treated similarly. One 
makes use of the formula
\begin{displaymath}
	\int\limits_{\alpha}^{\infty}
	\frac{x^{\mu-1}}{(1+\beta x)^{\nu}}dx=-\alpha^{\mu}\Gamma(\mu) \frac{ F(\nu,\mu;\mu+1;-\alpha\beta)}{\Gamma(\mu+1)}
+\beta^{-\mu} \frac{\Gamma(\nu-\mu)\Gamma(\mu)}{\Gamma(\nu)},
\end{displaymath}
for suitable exponents $\mu$ and $\nu$. It will apply to our situation. Plugging in our data, 
in particular letting $\alpha:=k^{2(\delta-1)}\omega^{-2}$ and $\beta:=\sinh(\eta)^{2}$, 
and taking infinite sums, we obtain the meromorphic continuation of $\Omega_{1}(s)$, 
and a Laurent expansion at $s=0$ of the form
\begin{align*}
	\Omega_{1}(s)
&=\frac{1}{6\omega}+f_1(\eta)+
\left(\frac{\gamma}{3\omega}+\frac{1}{3\omega}-\frac{\log 2}{3\omega}-\frac{\log(\omega)}{3\omega}
+\frac{\log(\eta)}{3\omega}
+g_1(\eta)\right)s+O(s^2),
\end{align*}
for suitable functions $f_1(\eta)$ and $g_1(\eta)$, which are $o(1)$ as $\eta\to 0$;
here, $\gamma$ denotes the Euler constant. 
Proceeding similarly for $\Omega_{2}(s)$, we obtain a Laurent expansion at $s=0$ of the form
\begin{align*}
\Omega_{2}(s)
&=-\frac{1}{12 \omega}   +f_2(\eta)+
\left(-\frac{\gamma}{6 \omega}+\frac{1}{12 \omega}+\frac{\log 2}{6 \omega}
+\frac{\log(\omega)}{6\omega}
-\frac{\log(\eta)}{6\omega}
+g_2(\eta)\right)s+O(s^2),
\end{align*}
for suitable functions $f_2(\eta)$ and $g_2(\eta)$, which are $o(1)$ as $\eta\to 0$.

For $\Omega_{3}(s)$, on its interval of integration, we can use the representation
\begin{align*}
A_{3}(u,k)=
\frac{u}{24  }
\sum\limits_{r=0}^{\infty}
\frac{(3(\frac{5}{2})_{r+2} -5(\frac{3}{2})_{r+2}) (-1)^{r}}{(r+2)!}\tanh(\eta)^{2r+4}(u^2-1)^{r}.
\end{align*}
One thus reduces to the evaluation of integrals of the form
\begin{displaymath}
	\int_{\alpha}^{1}\frac{(x-1)^{\mu}}{x^{\nu}}dx,
\end{displaymath}
for some $\alpha>0$. Again, this can be evaluated in terms of hypergeometric functions. After these operations, one obtains a lengthy expression for $\Omega_{3}(s)$ that we won't write here. As before, we can deduce from this expression the meromorphic continuation of $\Omega_{3}(s)$ and its Laurent expansion at $s=0$. Because of the powers in $\tanh(\eta)$ (see the expression for $A_{3}(u,k)$), at $s=0$, one obtains
\begin{displaymath}
	\Omega_{3}(s)=f_3(\eta)+g_3(\eta)s+O(s^{2}),
\end{displaymath}
for suitable functions $f_3(\eta)$ and $g_3(\eta)$, which are $o(1)$ as $\eta\to 0$. Finally, for $\Omega_{4}(s)$, after an expansion in power series, to get rid of the denominators of $A_{3}(u,k)$, we arrive at an expression
\begin{displaymath}
\Omega_{4}(s)=
\frac{2\,\sin(\pi s)}{\pi}
  \sum\limits_{j=0}^{\infty}\frac{(s)_j}{j! }\frac{1}{4^j}
\frac{\zeta(2s+2j+1)}{\omega^{2s+2j+1}}
\left(C_1(s)+C_2(s)+C_3(s)+C_4(s)\right),
\end{displaymath}
where
\begin{align*}
C_1(s)&:=\frac{1}{16\cosh(\eta)^4}
\sum\limits_{r=0}^{\infty}
\frac{(\frac{5}{2})_{r+2} }{r^2+3r+2}\cosh(\eta)^{-2r}\frac{\Gamma(s+j+\frac{7}{2})}{\Gamma(s+j+r+\frac{9}{2})},\\
C_2(s)&:=-\frac{5}{48\cosh(\eta)^4}
\sum\limits_{r=0}^{\infty}
\frac{(\frac{3}{2})_{r+2}}{r^2+3r+2}\cosh(\eta)^{-2r}
\frac{\Gamma(s+j+\frac{5}{2})}{\Gamma(s+j+r+\frac{7}{2})},\\
C_3(s)&:=\frac{1}{24}\left(\frac{1}{2}+\frac{3}{2s+2j+3}+(s+j)(\psi(s+j+1)-\psi(s+j+3/2))\right),\\
C_4(s)&:=-\frac{15}{48\cosh(\eta)^2}\frac{1}{2 s+2j+5}.
\end{align*}
The meromorphic continuation already follows. For the Laurent expansion, one collects the terms that eventually contribute, that is for $j=0$:
\begin{align*}
&D_{1}(s):=
\frac{2\sin(\pi s)}{\pi}\frac{\zeta(2s+1)}{\omega^{2s+1}} 
\frac{\Gamma(s+\frac{7}{2})}{16\cosh(\eta)^4}
\sum\limits_{r=0}^{\infty}
\frac{(\frac{5}{2})_{r+2} }{r^2+3r+2}\cosh(\eta)^{-2r}\frac{1}{\Gamma(s+r+\frac{9}{2})},
\\
&D_{2}(s):=-
\frac{2\sin(\pi s)}{\pi}\frac{\zeta(2s+1)}{\omega^{2s+1}} 
\frac{5\Gamma(s+\frac{5}{2}) }{48\cosh(\eta)^4}
\sum\limits_{r=0}^{\infty}
\frac{(\frac{3}{2})_{r+2}}{r^2+3r+2}\cosh(\eta)^{-2r}
\frac{1}{\Gamma(s+r+\frac{7}{2})},
\\
&D_{3}(s):=
\frac{2\sin(\pi s)}{\pi}\frac{\zeta(2s+1)}{\omega^{2s+1}} 
\frac{1}{24}\left(\frac{1}{2}+\frac{3}{2s+3}+(s)(\psi(s+1)-\psi(s+3/2))\right),
\\
&D_{4}(s):=
-\frac{2\sin(\pi s)}{\pi}\frac{\zeta(2s+1)}{\omega^{2s+1}} 
\frac{15}{48\cosh(\eta)^2}\frac{1}{2 s+5}.
\end{align*}
Just to give an idea of the terms obtained, let us describe the Laurent expansion of $D_{1}(s)$ at $s=0$:
\begin{align*}
D_{1}(s)&=\frac{5}{32 \omega}+f(\eta)+
\left(\frac{75 \gamma-46+30 \log(2)-75 \log(\omega)}{240 \omega}+g(\eta)\right) s+O(s^2),
\end{align*}
for suitable functions $f(\eta)$ and $g(\eta)$, which are $o(1)$ as $\eta\to 0$.
The miracle is that, after adding all the contributions, the coefficients that remain are $o(1)$, as $\eta\to 0$, \emph{i.e.}, at $s=0$, we 
finally obtain
\begin{displaymath}
	\Omega_{4}(s)=f_4(\eta)+g_4(\eta)s+O(s^{2}),
\end{displaymath}
for suitable functions $f_4(\eta)$ and $g_4(\eta)$ converging to 0 as $\eta\to 0$. The conclusion is now a matter of adding up the contributions of all the $\Omega_{\ell}(s)$.
\end{proof}

At this point, we still have to consider the function $\zeta_{\widetilde{M}_{0}}(s)$. This was defined in 
\eqref{splitdef_widetilde_mzeroandone} as the contribution of the integrals for $k=0$, on the interval $[1,+\infty)$.
\begin{proposition}\label{prop:zeta-cone-M0}
The infinite sume $\zeta_{\widetilde{M}_{0}}(s)$ is absolutely and locally uniformly convergent for $1<\Real(s)<2$, and admits a meromorphic continuation to $\CC$. Furthermore, $\zeta_{\widetilde{M}_{0}}(s)$ is holomorphic at $s=0$ and
\begin{displaymath}
	\zeta_{\widetilde{M}_{0}}^{\prime}(0)=\zeta(0)\log(\eta)+\zeta^{\prime}(0)-  \zeta(0)^2+ o(1),
\end{displaymath}
as $\eta\to 0$.
\end{proposition}
\begin{proof}
We express $P_{-1/2+t}^{0}(\cosh(\eta))$ in terms of the hypergeometric function
\begin{displaymath}
	P_{-1/2+t}^{0}(\cosh(\eta))=F\left(\frac{1}{2}+t,\frac{1}{2}-t;1;\frac{1-\cosh(\eta)}{2}\right).
\end{displaymath}
Then we appeal to the asymptotics in \cite[eq.~(17), p.~77]{Erdelyi}, giving for $t\to +\infty$
\begin{align*}
F\left(\frac{1}{2}+t,\frac{1}{2}- t;1;\frac{1-\cosh(\eta)}{2}\right)&=
\frac{1}{\sqrt{\pi}}
(1-e^{-2\eta})^{-1/2}t^{-1/2}\\
&\times\bigl(e^{(t-1/2)\eta} +e^{ i \pi(1/2)} e^{-(t+1/2)\eta} \bigr)
\bigl[1+O(\vert t^{-1}\vert)].
\end{align*}
From the resulting asymptotics for $\log P_{-1/2+t}^{0}(\cosh(\eta))$ as $t \to +\infty$, we infer that 
we need to substract the regularizing term
\begin{displaymath}
	E(t,\eta):=-\frac{1}{2}\log(\pi)-\frac{1}{2}\log(1-e^{-2\eta})-\frac{1}{2}\log(t)
+\left(t-\frac{1}{2}\right)\eta.
\end{displaymath}
By the asymptotic expansion with remainder $O(t^{-1})$ and because $P_{-1/2+t}^{0}(\cosh(\eta))$ is entire in $t$, the expression
\begin{align*}
	\zeta_{\widetilde{M}_{0},1}(s):=
	\frac{\sin(\pi s)}{\pi}\int_{1}^{\infty}(t^{2}-1/4)^{-s}\frac{\partial}{\partial t}
	\left(\log P_{-1/2+t}^{0}(\cosh(\eta)) -E(t,\eta)\right)dt
\end{align*}
defines a holomorphic function for $\Real(s)>-1/2$. Besides, at $s=0$, the integral can be evaluated just by substituting $s=0$ and using the primitive. Notice the latter vanishes at infinity by construction. We thus obtain a Laurent expansion at $s=0$ of the form
\begin{align*} 
\zeta_{\widetilde{M}_{0},1}(s)
&=\left(-\log\bigl(P_{\frac{1}{2}}^{0}(\cosh(\eta))\bigr)+E(1,\eta)\right)  s+ O(s^2)\\
&=\left(-\frac{1}{2}\log(\eta) -\frac{1}{2}\log(2\pi)+   f(\eta)\right) s+ O(s^2),
\end{align*}
for some function $f(\eta)$, which is $o(1)$ as $\eta\to 0$. 
The integral
\begin{displaymath}
	\zeta_{\widetilde{M}_{0},2}(s):=\frac{\sin(\pi s)}{\pi}\int_{1}^{\infty}(t^{2}-1/4)^{-s}\frac{\partial}{\partial t}E(t,\eta)dt
\end{displaymath}
it easily evaluated in terms of hypergeometric series, directly exhibiting the meromorphic continuation and, at $s=0$, the expansion
\begin{displaymath}
	\zeta_{\widetilde{M}_{0},2}(s)=-(\zeta(0)^{2}+\eta)s+O(s^{2}).
\end{displaymath}
Finally adding the contributions of $\zeta_{\widetilde{M}_{0},1}(s)$ and $\zeta_{\widetilde{M}_{0},2}(s)$, we complete the proof.
\end{proof}

\subsection{Conclusion and proof of Theorem \ref{thm:det_cone}}
The claim of Theorem \ref{thm:det_cone} is a consequence of the next statement.
\begin{theorem}
The spectral zeta function $\zeta_{\omega}(s)$ of the model hyperbolic cone converges absolutely and locally uniformly for $\Real(s)>1$, and admits a meromorphic continuation to the half-plane $\Real(s)>-\sigma_{0}$, for some $\sigma_{0}>0$.  Furthermore, the derivative of $\zeta_{\omega}(s)$ at $s=0$ 
satisfies
the equality
\begin{displaymath} 
\begin{split} 
		\zeta_{\omega}^{\prime}(0)=&\left(\frac{\omega}{6}+\frac{1}{6\omega}\right)\log(\eta)
	    +\omega\left(-2\zeta^{\prime}(-1)-\frac{1}{6}\log 2+\frac{1}{6}\right)
		+\frac{1}{\omega}\left(\frac{\gamma}{6}-\frac{1}{6}\log 2+\frac{5}{12}\right)\\
		&-\frac{1}{2}\log\omega-\frac{1}{6}\omega\log\omega-\frac{1}{6}\frac{\log\omega}{\omega}-\frac{1}{4}+o(1),
\end{split}		
\end{displaymath}
as $\eta\to 0$.
\end{theorem}

\begin{proof}
The result is the conjunction of propositions \ref{prop:cone-zeta-L}--\ref{prop:zeta-cone-M0}.
\end{proof}
%

\section{Regularized determinants and the Selberg trace formula}\label{section:selberg}
Let $X$ be a compact Riemann surface, arising as a compactification of a quotient $\Gamma\backslash\HH$, for a fuchsian group $\Gamma$. Let be $p_{1},\ldots,p_{n}$ be the set of cusps and elliptic fixed points, with multiplicities $m_{i}\leq\infty$. By $c$ we denote the number of cusps of $\Gamma\backslash\HH$. We endow $X$ with the singular Poincar\'e metric $ds_{\hyp}^{2}$ descended from $\HH$. Let $\Delta_{\hyp}$ be the Friedrichs extension of the hyperbolic laplace operator, given in coordinates on $\HH$ by
\begin{displaymath}
	\Delta_{\hyp}=-y^{2}\left(\frac{\partial^{2}}{\partial x^{2}}+\frac{\partial^{2}}{\partial y^{2}}\right).
\end{displaymath}
Fix a positive parameter $a>0$. We similarly introduce the one dimensional laplacian $-y^{2}d^{2}/dy^{2}$ on $L^{2}([a,+\infty), dy^{2}/y^{2})$, with Dirichlet boundary condition at $y=a$. If $a>0$ is big enough, the cusp $\C_{a}=\mathbb{S}^{1}\times [a,+\infty)$ isometrically embeds into a neighborhood of each cusp $p_{i}$ of $X$. Then $-y^{2}d^{2}/dy^{2}$ can be trivially extended to these neighborhoods, and further extended by $0$ to the complementary. The resulting operator was denoted $\Delta_{a}$ in \textsection\ref{subsubsec:pseudo-laplacian}. The relative determinant
\begin{displaymath}
	\det(\Delta_{\hyp},\Delta_{a})
\end{displaymath}
is then defined, as will result from the discussion below. Furthermore, this determinant can be evaluated by means of the Selberg trace formula. The method goes back to d'Hoker--Phong \cite{dHP}, Sarnak \cite{Sarnak}, and Efrat \cite{Efrat}. Although the authors don't address the possible presence of elliptic fixed points, their reasoning easily carries over. We proceed to outline the main steps, with special stress on the particularities of our example. We make no claim of originality.

\subsection{Step 1: the difference heat trace} We consider the difference of heat operators $e^{-t\Delta_{\hyp}}-e^{-t\Delta_{a}}$. One needs to justify that it is of trace class. The heat operator $e^{-t\Delta_{\hyp}}$ has an explicit kernel. If $\lbrace\lambda_{j}\rbrace$ is the pure point spectrum, and $E_{p_{j}}(z,s)$ is the analytic continuation (in $s$) of the Eisenstein series associated to a cusp $p_{j}$, then, the spectral expansion of the heat kernel is
\begin{displaymath}
	K(z,w,t)=
	\sum_{j}\varphi_{j}(z)\varphi_{j}(w) e^{-\lambda_{j}t}
	+\sum_{p_{j}\text{ cusp}}\frac{1}{4\pi}\int_{-\infty}^{+\infty}E_{p_{j}}(z,\frac{1}{2}+ir)\overline{E_{p_{j}}(z,\frac{1}{2}+ir)}e^{-(1/4+r^{2})t}dr.
\end{displaymath}
The sum converges absolutely and uniformly on compact sets disjoint from the cusps. The Maass--Selberg relations \cite[Chapter 6, (6.29)]{Iwaniec} show that $K(z,z,t)$ is not integrable with respect to the hyperbolic measure, thus confirming that $e^{-t\Delta_{\hyp}}$ alone is not of trace class. One can also read off how to ``regularize" $K(z,w,t)$, in order to obtain a trace class operator: substract the kernel of $e^{-t\Delta_{a}}$. Precisely, on a neighborhood of a cusp, we write $z=(x,y)$ and $w=(x',y')$, for the standard parametrization of $\C_{a}=\mathbb{S}^{1}\times [a,+\infty)$. On the region $\C_{a}\times\C_{a}$, the kernel of $e^{-t\Delta_{a}}$ is explicitly given by
\begin{displaymath}
	K_{a}(z,w,t)=\frac{e^{-t/4}}{\sqrt{4\pi t}}(yy')^{1/2}
	\left\lbrace e^{-\log(y/y')^{2}/4t}-e^{-(\log(yy')-\log(a^{2}))^{2}/4t}\right\rbrace.
\end{displaymath}
It is zero elsewhere on $X\times X$. Note that $K_{a}$ satisfies the Dirichlet boundary condition. From the 
Maass--Selberg relations, one concludes that $e^{-t\Delta_{\hyp}}-e^{-t\Delta_{a}}$ is of the trace class and
\begin{displaymath}
	\begin{split}
		\theta(t):=\tr(e^{-t\Delta_{\hyp}}-e^{-t\Delta_{a}})=&\int_{X}(K(z,z,t)-K_{a}(z,z,t))\frac{dx\wedge dy}{y^{2}}\\
		=&\sum_{j}e^{-t\lambda_{j}}
		-\frac{1}{4\pi}\int_{-\infty}^{+\infty}e^{-(1/4+r^{2})t}\frac{\phi'}{\phi}(\frac{1}{2}+ir)dr\\
		&+\frac{1}{4}e^{-t/4}\tr(\Phi(\frac{1}{2}))
		+c\frac{e^{-t/4}}{\sqrt{4\pi t}}\log a
		+c\frac{e^{-t/4}}{4}.
	\end{split}
\end{displaymath}
The notation $\Phi$ stands for the scattering matrix (giving the functional equation of the vector of Eisenstein series) and $\phi=\det\Phi$ is its determinant.

\subsection{Step 2: the spectral zeta function} In preparation for the definition and analytic continuation of the spectral zeta function, one needs suitable asymptotics of the difference heat trace as $t\to 0$ and $t\to +\infty$. This is a a typical application of the Selberg trace formula \cite[Thm.~10.2]{Iwaniec}. As $t\to 0$, one has
\begin{displaymath}
	\theta(t)=a_{-1}t^{-1}+bt^{-1/2}(\log t) + a_{-1/2} t^{-1/2} + a_{0} +O(t^{1/2}\log t).
\end{displaymath}
See for instance \cite[Thm.~11.1]{Iwaniec}. With respect to the usual asymptotic expansion on compact riemannian surfaces, we note the ``unusual" term $t^{-1/2}\log t$ and the rest of the form $O(t^{1/2}\log t)$. As $t\to +\infty$
\begin{displaymath}
	\theta(t)=1+O(e^{-\delta t}),
\end{displaymath}
for some real $\delta>0$. The asymptotics guarantee that the following Mellin transform is absolutely convergent for $\Real(s)>1$ and extends to a meromorphic function on $\Real(s)>-1/2$, which is holomorphic at $s=0$:
\begin{displaymath}
	\zeta_{a}(s)=\frac{1}{\Gamma(s)}\int_{0}^{+\infty} t^{s-1}\left(\theta(t)-(\dim\ker\Delta_{\hyp}-\dim\ker\Delta_{a})\right)dt.
\end{displaymath}
Observe that
\begin{displaymath}
	\dim\ker\Delta_{\hyp}-\dim\ker\Delta_{a}=1-0=1.
\end{displaymath}
For a complex parameter $s$ with $\Real(s)>1$, an explicit expression for the spectral zeta function $\zeta_{a}(s)$ is 
\begin{displaymath}
	\begin{split}
		\zeta_{a}(s)=&\sum_{\lambda_{j}\neq 0}\frac{1}{\lambda_{j}^{s}}-\frac{1}{4}\int_{-\infty}^{+\infty}(\frac{1}{4}+r^{2})^{-s}\frac{\phi'}{\phi}(\frac{1}{2}+ir)dr\\
		&+4^{s-1}\tr(\Phi(\frac{1}{2}))
		+c\log a \frac{4^{s-1/2}}{\sqrt{4\pi}}\frac{\Gamma(s+1/2)}{(s-1/2)\Gamma(s)} +c\, 4^{s-1}.
	\end{split}
\end{displaymath}
In terms of this zeta functions, the relative determinant $(\Delta_{\hyp},\Delta_{a})$ is given by
\begin{displaymath}
	\det(\Delta_{\hyp},\Delta_{a})=\exp(-\zeta^{\prime}_{a}(0)).
\end{displaymath}
\subsection{Step 3: application of the trace formula} Applying the trace formula \cite[Thm.~10.2]{Iwaniec} in the lines of \cite{Sarnak} and \cite{Efrat}, one evaluates $\det(\Delta_{\hyp},\Delta_{a})$ in terms of $Z^{\prime}(1,\Gamma)$. The result is summarized here. 
\begin{proposition}\label{prop:det-Selberg-zeta}
The relative determinant $\det(\Delta_{\hyp},\Delta_{a})$ is given by
\begin{displaymath}
	\begin{split}
		\log\det(\Delta_{\hyp},\Delta_{a})=&\log Z'(1,\Gamma)\\
		&+(2g-2+\sum_{i}(1-m_{i}^{-1}))(2\zeta'(-1)-\frac{1}{4}+\frac{1}{2}\log(2\pi))\\
		&+\sum_{m_{i}<\infty}\sum_{k=0}^{m_{i}-2}\frac{2k+1-m_{i}}{m_{i}^{2}}\log\Gamma(\frac{k+1}{m_{i}})\\
		&+\frac{1}{6}\sum_{m_{i}<\infty}\left(1-\frac{1}{m_{i}^{2}}\right)\log m_{i}\\
		&-\frac{c}{2}(\log(4/a)-1).
	\end{split}
\end{displaymath}
\end{proposition}

\section{Proof of Theorem \ref{theorem:main}}\label{section:proof-main}
We now go on to proof the main result of this article Theorem \ref{theorem:main}.

\subsection{} Recall we already proved a ``naive" version of the isometry we seek, that is Proposition \ref{prop:naive_laplacian}. Up to the topological constant \eqref{eq:top-ctt}, the proposition establishes an isometry with a Quillen type metric involving the naive determinant $\det\Delta_{\hyp}^{\ast}$ (Notation \ref{notation:naive_det}):
\begin{equation}\label{eq:def-naive-det-bis}
	\deta\Delta_{\hyp}:=\lim_{\varepsilon\to 0} \frac{\detp\Delta_{\cpd,\hyp,\varepsilon}}{\varepsilon^{\frac{1}{6}\sum_{i}(1-m_{i}^{-1})^{2}}(\log\varepsilon^{-1})^{\frac{c}{3}}}.
\end{equation}
This reduces our task to the evaluation of the naive determinant. The evaluation appeals to the Mayer--Vietoris formulae of Section \ref{section:Mayer--Vietoris}, through the application provided by Theorem \ref{prop:Mayer--Vietoris}:
\begin{equation}\label{eq:naive-det-bis}
	\detp(\Delta_{\cpd,\hyp,\varepsilon})=C\frac{\vol(X,ds_{\hyp,\varepsilon}^{2})}{\vol(X,ds_{\hyp}^{2})}\frac{\detp(\Delta_{\hyp},\Delta_{\hyp,\cusp})\det\Delta_{\ov{V}_{\varepsilon/2},\hyp,\varepsilon}}{
	\prod_{m_{i}<\infty}\det\Delta_{\hyp,i}},
\end{equation}
where $C$ is the constant \eqref{eq:24}. The idea is to compute the limit defining the naive determinant, by plugging into \eqref{eq:naive-det-bis} our computations for the determinant of the model hyperbolic cusp and cone and then performing the limit \eqref{eq:def-naive-det-bis}. However, we still need to evaluate $\det\Delta_{\ov{V}_{\varepsilon/2},\hyp,\varepsilon}$. This is the content of the following lemma.
\begin{lemma}\label{lemma:spreafico}
With notations as in \textsection\ref{subsec:MV-trunated}, the value of $\det\Delta_{\ov{V}_{\varepsilon/2},\hyp,\varepsilon}$ satisfies
\begin{displaymath}
	\begin{split}
	\log\det\Delta_{\ov{V}_{\varepsilon/2},\hyp,\varepsilon}=&\frac{c}{3}\log(\log(1/\varepsilon))+\frac{c}{3}\log 2
	-\sum_{m_{i}<\infty}\frac{1}{3m_{i}}\log(\varepsilon)+\frac{1}{3}\sum_{m_{i}<\infty}\log(m_{i})\\
	&-n\left(2\zeta^{\prime}(-1)+\frac{1}{2}\log(2\pi)-\frac{1}{3}\log 2+\frac{5}{12}\right)+o(1),
	\end{split}
\end{displaymath}
as $\varepsilon\to 0$.
\end{lemma}
\begin{proof}
The claim follows from the formula for the determinant of the scalar laplacian on an euclidean disk of radius $r$, that we quote from \cite{Spreafico}. If $\zeta_{r}(s)$ is the spectral zeta function of the disk of radius $r$ with the euclidean metric, then we have
\begin{displaymath}
	\zeta^{\prime}_{r}(0)=\frac{1}{3}\log r-\frac{1}{3}\log 2+\frac{1}{2}\log(2\pi)+\frac{5}{12}+2\zeta^{\prime}(-1).
\end{displaymath}
We conclude by a suitable choice of $r$, obtained by computing the \emph{riemannian} volumes of the disks with the several truncated hyperbolic metrics. For instance, for an elliptic fixed point of order $m_{i}$, one equates
\begin{displaymath}
	\pi r^{2}=\int_{D(0,\varepsilon/2)}\frac{4dx\wedge dy}{m_{i}\varepsilon^{2-2/m_{i}}(1-\varepsilon^{2/m_{i}})^{2}},
\end{displaymath}
hence
\begin{displaymath}
	r=\frac{\varepsilon^{1/m_{i}}}{m_{i}(1-\varepsilon^{2/m_{i}})},\quad\log r=\frac{1}{m_{i}}\log\varepsilon-\log m_{i}+o(1).
\end{displaymath}
\end{proof}
\subsection{} Before applying \eqref{eq:naive-det-bis}, we need to be careful when choosing the parameter $a$ in Theorem \ref{thm:det-pseudo-laplacian} and $\eta$ in Theorem \ref{thm:det_cone}. For this we first stress that the disks $V_{\varepsilon/2}$ have radius $\varepsilon/2$ in the \emph{rs} coordinates. Then, for the parameter $a$ in the model cusps, the right quantity to use is
\begin{displaymath}
	a=\frac{1}{2\pi}\log(2/\varepsilon).
\end{displaymath}
For the hyperbolic cone of angle $2\pi/\omega$, the right parameter according to \eqref{eq:relation-eps-R} is 
\begin{displaymath}
	\eta=2(\varepsilon/2)^{1/\omega}.
\end{displaymath} 
This is legitimate, since we work up to $o(1)$ terms. It is now a matter of carefully plugging all the determinants into \eqref{eq:naive-det-bis}: Theorem \ref{thm:det-pseudo-laplacian}, Theorem \ref{thm:det_cone}, Proposition \ref{prop:det-Selberg-zeta} and Lemma \ref{lemma:spreafico}. Then compute the limit. Miraculously, the divergent terms as $\varepsilon\to 0$ cancel out. It remains a finite quantity. Combining this with Proposition \ref{prop:naive_laplacian}, and rearranging terms (so that the constant $C^{\ast}(g)$ therein is moved to the left hand side of the isometry), we obtain the desired isometry with Quillen metric of the form
\begin{displaymath}
	\|\cdot\|_{Q}=\|\cdot\|_{L^{2}}(C(\Gamma)Z^{\prime}(1,\Gamma))^{-1/2},
\end{displaymath}
and with the explicit constant $C(\Gamma)$ given by \eqref{eq:ctt-Gamma}. This completes the proof of the theorem.

\section{Arithmetic Riemann--Roch and the special value $Z^{\prime}(1,\PSL_{2}(\Int))$}\label{section:arithmetic}
The isometry theorem \ref{theorem:main} and the compatibility of Deligne's isomorphism with base change combine to settle a Riemann--Roch type formula for arithmetic surfaces. The result we aim to prove extends the arithmetic Riemann--Roch theorem of Gillet--Soul\'e \cite{GS} (see also \cite{Soule} for a survey), for the structure sheaf of an arithmetic surface with a \emph{smooth} K\"ahler metric at the archimedean places. The improvement concerns the smoothness hypothesis on the metric, by allowing a singular Poincar\'e type metric given by an arbitrary fuchsian uniformization of the archimedean fiber deprived of a number of \emph{cusps}. We stress that the fuchsian groups in question may have torsion, and this is the novelty with respect to previous work of the first author \cite{GFM}, whose techniques were limited to the torsion free case. The theorem is new even in the absence of cusps. Allowing cusps and elliptic fixed points at the same time provides a wealth of examples of arithmetic interest. For instance, integral models of modular curves with $\Gamma_{0}(N)$-level structure and no restriction on $N$. This is the case of the modular curve $X(1)=\PP^{1}_{\Int}$, whose complex points we interpret as $\PSL_{2}(\Int)\backslash\HH\cup\lbrace\infty\rbrace$. We will treat this example in detail, thus exhibiting the arithmetic content of the theorem. We derive from it a closed expression for the special value $Z^{\prime}(1,\PSL_{2}(\Int))$.\footnote{To our knowledge, and according to specialists, this answers a longstanding question.}

\subsection{An arithmetic Riemann--Roch theorem}
Let $K$ be a number field and $\pi\colon\XX\to\BS$ and arithmetic surface. In this article, by arithmetic surface we mean a regular and integral scheme $\XX$ of Krull dimension 2, projective and flat over $\BS=\Spec\OO_{K}$, with geometrically connected fibers. The regularity assumption can actually be weakened to $\pi$ being generically smooth, in the lines of \cite{GS}. We proceed to describe the several pieces that constitute the arithmetic Riemann--Roch formula in view.

\subsubsection{} We suppose given sections $\sigma_{1},\ldots,\sigma_{n}\colon\BS\to\XX$, and multiplicities $1\leq m_{1},\ldots,m_{n}\leq\infty$. To fix ideas, say $m_{1},\ldots, m_{c}=\infty$ and $m_{c+1},\ldots, m_{n}<\infty$. We require:
\begin{enumerate}
	\item the sections are generically disjoint; 
	\item for every complex embedding $\tau\colon K\hookrightarrow\CC$, the non-compact Riemann surface $\XX_{\tau}(\CC)\setminus\lbrace\sigma_{1}(\tau),\ldots,\sigma_{c}(\tau)\rbrace$ can be uniformized as $\Gamma_{\tau}\backslash\HH$, for some fuchsian group $\Gamma_{\tau}\subset\PSL_{2}(\RR)$. Besides, $\sigma_{1}(\tau),\ldots,\sigma_{c}(\tau)$ correspond to cusps and $\sigma_{c+1}(\tau),\ldots,\sigma_{n}(\tau)$ correspond to elliptic fixed points of respective orders $m_{c+1},\ldots,m_{n}$.
\end{enumerate}
We introduce the $\QQ$-line bundle
\begin{displaymath}
	\omega_{\XX/\BS}(D),\quad\text{with}\quad D=\sum_{i}\left(1-\frac{1}{m_{i}}\right)\sigma_{i}.
\end{displaymath}
After base change to $\CC$ through the complex embeddings $\tau\colon K\hookrightarrow\CC$, we endow $\omega_{\XX/\BS}(D)$ with the (dual) Poincar\'e metrics, coming from the fuchsian uniformizations $\Gamma_{\tau}\backslash\HH$. These data define a hermitian $\QQ$-line bundle that we denote $\omega_{\XX/\BS}(D)_{\hyp}$. The extended arithmetic intersection theories of Bost \cite{Bost} and K\"uhn \cite{Kuhn} apply and provide real numbers
\begin{equation}\label{eq:arith-int} 
	(\omega_{\XX/\BS}(D)_{\hyp},\omega_{\XX/\BS}(D)_{\hyp})\in\RR.
\end{equation}
The relation to the Deligne pairing is as follows. Let $m=\prod_{i>c}m_{i}$. The Deligne pairing $$\langle\omega_{\XX/\BS}^{m}(mD),\omega_{\XX/\BS}^{m}(mD)\rangle$$ is defined, and commutes with base change. Therefore, after base change through each $\tau\colon K\hookrightarrow\CC$, we can endow it with the Deligne--Bost norm of Definition \ref{definition:Deligne-Bost}, for the $m$-th power of the Poincar\'e metrics on $\omega_{\XX/\BS}^{m}(mD)_{\tau}$. We thus obtain a hermitian line bundle over $\Spec\OO_{K}$, hence an element in the arithmetic Picard group $\APic(\BS)$ of Arakelov geometry. Because Deligne's pairing is bilinear with respect to tensor product of line bundles, we can avoid taking the $m$-th powers and talk instead of a $\QQ$-hermitian line bundle over $\Spec\OO_{K}$, that we denote
\begin{displaymath}
	\langle\omega_{\XX/\BS}(D)_{\hyp},\omega_{\XX/\BS}(D)_{\hyp}\rangle\in\APic(\BS)\otimes_{\Int}\QQ.
\end{displaymath}
The arithmetic intersection pairing \eqref{eq:arith-int} is then expressed in terms of the Deligne pairing:
\begin{displaymath}
	(\omega_{\XX/\BS}(D)_{\hyp},\omega_{\XX/\BS}(D)_{\hyp})=\adeg\langle\omega_{\XX/\BS}(D)_{\hyp},\omega_{\XX/\BS}(D)_{\hyp}\rangle,
\end{displaymath}
where $\adeg$ is the arithmetic degree map defined on $\APic(\BS)\otimes_{\Int}\QQ$ and with real values.

\subsubsection{} Another hermitian $\QQ$-line bundle over $\Spec\OO_{K}$ we need is the \emph{psi} bundle. For every section $\sigma_{i}$, the line bundle $\psi_{i}=\sigma_{i}^{\ast}(\omega_{\XX/\BS})$ can be endowed, after every base change to $\CC$, with the Wolpert metric at $\sigma_{i}$ (Definition \ref{definition:Wolpert}). We write $\psi_{i,W}$ for the resulting hermitian line bundle, and abusively identify it with its class in $\APic(\BS)$. We define a hermitian $\QQ$-line bundle by
\begin{displaymath}
	\psi_{W}=\sum_{i}\left(1-\frac{1}{m_{i}^{2}}\right)\psi_{i,W}\in\APic(\BS)\otimes_{\Int}\QQ.
\end{displaymath}
The arithmetic degree $\adeg\psi_{W}$ is an interesting invariant that measures how far the \emph{rs} coordinates are from being formal algebraic.  

\subsubsection{} The determinant of the cohomology of the structure sheaf $\OO_{\XX}$ can be defined by the general theory of Knudsen--Mumford \cite{KM}. Precisely, $H^{0}(\XX,\OO_{\XX})$ and $H^{1}(\XX,\OO_{\XX})$ are finite type modules over the Dedekind domain $\OO_{K}$, and after taking finite projective resolutions (actually of length 2) one can define their determinants by imposing the alternate multiplicativity of the determinant functor on exact sequences. We write
\begin{displaymath}
	\det H^{\bullet}(\XX,\OO_{\XX})=\det H^{0}(\XX,\OO_{\XX})\otimes\det H^{1}(\XX,\OO_{\XX})^{-1}.
\end{displaymath}
The Knudsen--Mumford construction commutes with base change. Therefore, after base change to $\CC$, the determinant of cohomology can be endowed with the Quillen metric for the fuchsian uniformizations $\Gamma_{\tau}\backslash\HH$ (Theorem \ref{theorem:main}). We denote $\det H^{\bullet}(\XX,\OO_{\XX})_{Q}$ for the determinant of cohomology with the Quillen metrics at the archimedean places. To lighten notations, we put
\begin{displaymath}
	\adeg H^{\bullet}(\XX,\OO_{\XX})=\adeg\det H^{\bullet}(\XX,\OO_{\XX}).
\end{displaymath}

\subsubsection{} We are almost in position to state the arithmetic Riemann--Roch theorem. We still need an observation. By assumption, the sections $\sigma_{i}$ are generically disjoint. If we denote $\OO(\sigma_{i})$ the line bundles they define, the Deligne pairings
\begin{displaymath}
	\langle\OO(\sigma_{i}),\OO(\sigma_{j})\rangle\simeq\sigma_{i}^{\ast}\OO(\sigma_{j})
\end{displaymath}
for $i\neq j$ are canonically trivial after extending scalars to $\CC$. Thus it is legitimate to equip them with the trivial hermitian metric at the archimedean places. Their arithmetic degrees encode the geometric intersection numbers of $\sigma_{i}$ and $\sigma_{j}$ at finite places. Explicitly
\begin{displaymath}
	\adeg\langle\OO(\sigma_{i}),\OO(\sigma_{j})\rangle=\sum_{\mathfrak{p}}\mu_{\mathfrak{p}}\log\sharp(\OO_{K}/\mathfrak{p}),
\end{displaymath}
where the sum runs over maximal ideals of $\OO_{K}$, and if $p$ is a (closed) intersection point of $\sigma_{i}$ and $\sigma_{j}$, lying above $\mathfrak{p}$, then
\begin{equation}\label{eq:fin-int}
	\mu_{\mathfrak{p}}=\length\frac{\OO_{\XX,p}}{\OO(-\sigma_{i})_{p}+\OO(-\sigma_{j})_{p}}.
\end{equation}
Following customary notations, we denote these ``finite" arithmetic intersection numbers $(\sigma_{i},\sigma_{j})_{\fin}$.
\begin{theorem}\label{theorem:ARR}

There is an equality of real numbers
\begin{displaymath}
	\begin{split}
		12\adeg H^{\bullet}(\XX,\OO_{\XX})_{Q}-\delta+\adeg\psi_{W}=&
		(\omega_{\XX/\BS}(D)_{\hyp},\omega_{\XX/\BS}(D)_{\hyp})\\
		&-\sum_{i\neq j}\left(1-\frac{1}{m_{i}}\right)\left(1-\frac{1}{m_{j}}\right)(\sigma_{i},\sigma_{j})_{\fin},
	\end{split}
\end{displaymath}
where $\delta=\sum_{\mathfrak{p}}\delta_{\mathfrak{p}}\log\sharp(\OO_{K}/\mathfrak{p})$ and $\delta_{\mathfrak{p}}$ is the Artin conductor of $\XX$ at $\mathfrak{p}$ \cite{Saito}.
\end{theorem}
\begin{proof}
The claim is a straightforward consequence of Theorem \ref{theorem:main} and its counterpart for $\pi\colon\XX\to\BS$, $\sigma_{1},\ldots,\sigma_{n}$, together with the compatibility of Deligne's isomorphism with base change.
\end{proof}
\subsection{The special value $Z^{\prime}(1,\PSL_{2}(\Int))$}
The beauty of Theorem \ref{theorem:ARR} is that it applies to simple geometric situations where most of the terms can be explicitly evaluated. Yet, it implies highly non-trivial facts. Here we describe the case of the modular curve $X(1)$, and obtain the special value $Z^{\prime}(1,\PSL_{2}(\Int))$ in terms of $L$ functions. Recall, from the introduction, that we denote by $\chi_{i}=\left(\frac{-4}{\cdot}\right)$ the quadratic  character of $\QQ(i)$, and $\chi_{\rho}=\left(\frac{-3}{\cdot}\right)$ the quadratic character of $\QQ(\rho)$, $\rho=e^{2\pi i/3}$.

\begin{theorem}\label{thm:selberg-zeta}
The special value $Z^{\prime}(1,\PSL_{2}(\Int))$ is given by
\begin{displaymath}
	\begin{split}
	\log Z^{\prime}(1,\PSL_{2}(\Int))=&\frac{1}{4}\frac{L^{\prime}(0,\chi_{i})}{L(0,\chi_{i})}+\frac{13}{27}\frac{L^{\prime}(0,\chi_{\rho})}{L(0,\chi_{\rho})}
	+\frac{73}{72}\frac{\zeta^{\prime}(0)}{\zeta(0)}
	-\frac{37}{36}\frac{\zeta^{\prime}(-1)}{\zeta(-1)}-\frac{5}{36}\gamma\\
	&+\frac{5}{12}\log 3-\frac{167}{216}\log 2-\frac{5}{6}.
	\end{split}
\end{displaymath}
\end{theorem}
\subsubsection{} The proof is the concatenation of a series of lemmas and an application of Theorem \ref{theorem:ARR}. We introduce some notations. Recall $\PP^{1}_{\Int}\to\Spec\Int$ is the coarse moduli scheme of the Deligne--Mumford stack $\mathcal{M}_{1}\to\Spec\Int$ of generalized elliptic curves \cite{DR}. The natural map $j\colon\mathcal{M}_{1}\to\PP^{1}_{\Int}$ is an extension over $\Int$ of the $j$-invairant. We interpret $\PP^{1}_{\Int}(\CC)$ as the Riemann surface $\PSL_{2}(\Int)\backslash\HH\cup\lbrace\infty\rbrace$, while $\mathcal{M}_{1}(\CC)$ is the analytic stack $[\SL_{2}(\Int)\backslash\HH]\cup\lbrace{\infty}\rbrace$. The cusp at infinity and the elliptic fixed points $i$ (of order 2 and $j$-invariant $1728$) and $\rho=e^{2i\pi/3}$ (of order 3 and $j$-invariant 0) define integral sections $\Spec\Int\to\PP^{1}_{\Int}$. Let them be denoted by $\sigma_{\infty}$, $\sigma_{i}$, and $\sigma_{\rho}$. The cusp section $\sigma_{\infty}$ lifts to a section of $\mathcal{M}_{1}$. By contrast, the sections $\sigma_{i}$ and $\sigma_{\rho}$ don't lift to $\mathcal{M}_{1}$. The $j$-invariant is ramified along $\sigma_{i}$ and $\sigma_{\rho}$ and \'etale elsewhere \cite[Sec.~VI, Lemme 1.5]{DR}.

\begin{lemma}\label{lemma:selberg-1}
We have the following finite intersection numbers
\begin{align*}
	&(\sigma_{\infty},\sigma_{i})_{\fin}=0,\\
	&(\sigma_{\infty},\sigma_{\rho})_{\fin}=0,\\
	&(\sigma_{i},\sigma_{\rho})_{\fin}=\log(1728)=6\log 2+3\log 3.
\end{align*}
\end{lemma}
\begin{proof}
The divisors attached to the sections can be written in homogeneous coordinates $X_{0}$ and $X_{1}$:
\begin{align*}
	\infty&\colon X_{1}=0, \\
	i&\colon X_{0}-1728X_{1}=0,\\
	\rho&\colon X_{0}=0.
\end{align*}
The computation of the intersection multiplicities is readily read from the equations. Notice the first two just reflect the fact that the CM elliptic curves have potentially good reduction.
\end{proof}
\subsubsection{} Let $E_{i}$ and $E_{\rho}$ be the elliptic curves, defined over $\QQ$, having complex multiplication by $\QQ(i)$ and $\QQ(\rho)$, respectively. They have potentially good reduction. We denote by $h_{F}(E_{i})$ and $h_{F}(E_{\rho})$ their \emph{stable} Faltings height. This height is computed after an extension of the base field where $E_{i}$ and $E_{\rho}$ acquire good reduction. It is normalized so that it does not depend on the base extension. For instance, if $E_{i}$ has good reduction over $K$ and $\mathcal{E}_{i}$ is the N\'eron model over $\Spec\OO_{K}$, then
\begin{displaymath}
	h_{F}(E_{i})=\frac{1}{[K\colon\QQ]}\adeg H^{0}(\mathcal{E}_{i},\Omega^{1}_{\mathcal{E}_{i}/\OO_{K}})_{L^{2}},
\end{displaymath} 
where the $L^{2}$ metric on differential $1$-forms is given by the rule explained in \textsection\ref{subsec:prelim-L2}:
\begin{displaymath}
	\langle\alpha,\beta\rangle=\frac{i}{2\pi}\int_{E_{i}(\CC)}\alpha\wedge\ov{\beta}.
\end{displaymath}
The Faltings height of a CM elliptic curve can be explicitly evaluated by the Chowla--Selberg formula \cite{Colmez, Deligne:bourbaki}.
\begin{lemma}\label{lemma:selberg-2}
The arithmetic degree of $\psi_{W}$ is
\begin{displaymath}
	\adeg\psi_{W}=3h_{F}(E_{i})+\frac{16}{3}h_{F}(E_{\rho})-\frac{43}{18}(\sigma_{i},\sigma_{\rho})_{\fin}+\frac{25}{6}\log(4\pi).
\end{displaymath}
Consequently, by the Chowla--Selberg formula and Lemma \ref{lemma:selberg-1}, it can be written
\begin{displaymath}
	\adeg\psi_{W}=-\frac{3}{2}\frac{L^{\prime}(0,\chi_{i})}{L(0,\chi_{i})}-\frac{8}{3}\frac{L^{\prime}(0,\chi_{\rho})}{L(0,\chi_{\rho})}+\frac{25}{6}\frac{\zeta^{\prime}(0)}{\zeta(0)}
 -\frac{17}{2}\log 3	-\frac{15}{2}\log 2.
\end{displaymath}
\end{lemma}
\begin{proof}
For the cusp at infinity, the algebraic theory of the Tate curve and the relation with Fourier expansions \cite{DR}, provides an isometry
\begin{displaymath}
	\sigma_{\infty}^{\ast}(\omega_{\PP^{1}_{\Int}/\Int})_{W}\overset{\sim}{\longrightarrow}(\OO_{\Int},|\cdot|),
\end{displaymath}
where $\OO_{\Int}$ is the trivial line bundle on $\Int$ and $|\cdot|$ is the absolute value on $\CC$. Indeed, the formal completion of $\PP^{1}_{\Int}$ along $\sigma_{\infty}$ is canonically isomorphic to $\Spf\Int[[q]]$, where the formal parameter $q$ is the algebraic counterpart of $e^{2\pi i z}$, $z\in\HH$. Then
\begin{displaymath}
	\sigma_{\infty}^{\ast}(\omega_{\PP^{1}_{\Int}/\Int})\overset{\sim}{\longrightarrow}\Int\,dq.
\end{displaymath}
By construction, the Wolpert metric at the cusp assigns to $dq$ norm 1, thus proving the claim. As a result,
\begin{equation}\label{eq:psi-infty}
	\adeg\psi_{\infty,W}=0.
\end{equation}
For the elliptic fixed points, we treat in detail the \emph{psi} bundle at $i$, and leave $\rho$ as an exercise. Let $K$ be a finite extension of $\QQ$ where both $E_{i}$ and $E_{\rho}$ acquire good reduction, with N\'eron models $\mathcal{E}_{i}$ and $\mathcal{E}_{\rho}$. Their classifying maps $\Spec\OO_{K}\to\mathcal{M}_{1}$ are denoted $\bar{\sigma}_{i}$ and $\bar{\sigma}_{j}$. Also, the sections $\sigma_{i}$ and $\sigma_{\rho}$ can be trivially extended to $\Spec\OO_{K}\to\PP^{1}_{\Int}$, and we indicate these extensions by $\widetilde{\sigma}_{i}$ and $\widetilde{\sigma}_{\rho}$. Notice the relations
\begin{displaymath}
	\widetilde{\sigma}_{i}=j\circ\bar{\sigma}_{i},\quad \widetilde{\sigma}_{\rho}=j\circ\bar{\sigma}_{\rho}.
\end{displaymath}
As $j$ is ramified along $\sigma_{i}$ and $\sigma_{\rho}$ of orders $2$ and $3$, respectively, and \'etale elsewhere, we find 
\begin{displaymath}
	\widetilde{\sigma}^{\ast}_{i}(\omega_{\PP^{1}_{\Int}/\Int})=\bar{\sigma}_{i}^{\ast}\,j^{\ast}(\omega_{\PP^{1}_{\Int}/\Int})
	=\bar{\sigma}_{i}^{\ast}(\omega_{\mathcal{M}_{1;\OO_{K}}/\OO_{K}}(-\bar{\sigma}_{i}-2\bar{\sigma}_{\rho})).
\end{displaymath}
The subscript $\OO_{K}$ in $\omega_{\mathcal{M}_{1;\OO_{K}}/\OO_{K}}$ indicates the base change to $\OO_{K}$. By means of the residue isomorphism, we infer a canonical identification
\begin{displaymath}
	\widetilde{\sigma}^{\ast}_{i}(\omega_{\PP^{1}_{\Int}/\Int})\overset{\sim}{\longrightarrow}\bar{\sigma}_{i}^{\ast}(\omega_{\mathcal{M}_{1}/\Int})^{\otimes 2}\otimes\bar{\sigma}_{i}^{\ast}\OO(\bar{\sigma}_{\rho})^{\otimes(-2)}.
\end{displaymath}
Because $\bar{\sigma}_{i}$ is disjoint with the cusp $\infty$, this isomorphism is the same as
\begin{equation}\label{eq:iso-wol-1}
	\widetilde{\sigma}^{\ast}_{i}(\omega_{\PP^{1}_{\Int}/\Int})\overset{\sim}{\longrightarrow}\bar{\sigma}_{i}^{\ast}(\omega_{\mathcal{M}_{1}/\Int}(\infty))^{\otimes 2}\otimes\bar{\sigma}_{i}^{\ast}\OO(\bar{\sigma}_{\rho})^{\otimes(-2)}.
\end{equation}
Let us now introduce $\underline{\omega}$ the sheaf of modular forms of weight 1 on $\mathcal{M}_{1}$. Recall that if $\pi\colon\mathcal{E}\to\mathcal{M}_{1}$ is the universal generalized elliptic curve, with unit section $e$, then
\begin{displaymath}
	\underline{\omega}=e^{\ast}(\Omega^{1}_{\mathcal{E}/\mathcal{M}_{1}}).
\end{displaymath}
Deformation theory and the theory of the Tate curve provide the canonical Kodaira--Spencer isomorphism \cite[Sec.~10.13]{Ka-Ma}:
\begin{equation}\label{eq:iso-wol-2}
	\underline{\omega}^{\otimes 2}\overset{\sim}{\longrightarrow}\omega_{\mathcal{M}_{1}/\Int}(\infty).
\end{equation}
From \eqref{eq:iso-wol-1}--\eqref{eq:iso-wol-2} we derive a canonical isomorphism
\begin{equation}\label{eq:iso-wol-3}
	\widetilde{\sigma}^{\ast}_{i}(\omega_{\PP^{1}_{\Int}/\Int})\overset{\sim}{\longrightarrow}\bar{\sigma}_{i}^{\ast}(\underline{\omega})^{\otimes 4}\otimes\bar{\sigma}_{i}^{\ast}\OO(\bar{\sigma}_{\rho})^{\otimes(-2)}.
\end{equation}
Now we consider the compatibility of \eqref{eq:iso-wol-3} with hermitian metrics. The line bundle $\underline{\omega}$ carries a family $L^{2}$ metric. It is defined on the complement of the cusp through the canonical isomorphism
\begin{displaymath}
	\underline{\omega}\mid_{\mathcal{M}_{1}\setminus\lbrace\infty\rbrace}\overset{\sim}{\longrightarrow}\pi_{\ast}(\Omega^{1}_{\mathcal{E}/\mathcal{M}_{1}})\mid_{\mathcal{M}_{1}\setminus\lbrace\infty\rbrace}.
\end{displaymath}
If we transport the family $L^{2}$ metric to $\omega_{\mathcal{M}_{1}/\Int}(\infty)$ via the Kodaira--Spencer isomorphism, its expression in the coordinate $\tau\in\HH$ is
\begin{displaymath}
	\|d\tau\|=4\pi\Imag \tau.
\end{displaymath}
Together with the definition of the Wolpert metric, one then easily checks that \eqref{eq:iso-wol-3} is a quasi-isometry of norm $1/(4\pi)$:
\begin{displaymath}
	\widetilde{\sigma}^{\ast}_{i}(\omega_{\PP^{1}_{\Int}/\Int})_{W}\overset{\sim}{\longrightarrow}\bar{\sigma}_{i}^{\ast}(\underline{\omega}_{L^{2}})^{\otimes 4}\otimes\bar{\sigma}_{i}^{\ast}\OO(\bar{\sigma}_{\rho})^{\otimes (-2)}\otimes(\OO_{K},\frac{1}{(4\pi)^{2}}\lbrace|\cdot|_{\tau}\rbrace_{\tau\colon K\hookrightarrow\CC}).
\end{displaymath}
Of course, $\bar{\sigma}_{i}^{\ast}\OO(\bar{\sigma}_{\rho})^{\otimes(-2)}$ carries the trivial metric at the archimedean places. By taking arithmetic degrees, normalizing by $[K\colon\QQ]$ and by the functoriality properties of arithmetic intersection numbers, we arrive at
\begin{equation}\label{eq:psi-i}
	\adeg\psi_{i,W}=4h_{F}(E_{i})-2(\sigma_{i},\sigma_{\rho})_{\fin}+2\log(4\pi).
\end{equation}
Similarly,
\begin{equation}\label{eq:psi-rho}
	\adeg\psi_{\rho,W}=6h_{F}(E_{\rho})-(\sigma_{i},\sigma_{\rho})_{\fin}+3\log(4\pi).
\end{equation}
Substituting \eqref{eq:psi-infty}, \eqref{eq:psi-i} and \eqref{eq:psi-rho} in the formula
\begin{displaymath}
	\adeg\psi_{W}=\adeg\psi_{\infty,W}+\frac{3}{4}\adeg\psi_{i,W}+\frac{8}{9}\adeg\psi_{\rho,W},
\end{displaymath}
and adding up, we obtain the first statement. The second assertion follows by a straight-forward computation employing the formulas 
(see, e.g., \cite{GFM16} for the general formula for the Faltings height)
\begin{displaymath}
h_{F}(E_{i})=-\frac{1}{2} \frac{L^{\prime}(\chi_i,0)}{L(\chi_i,0)},\qquad
h_{F}(E_{\rho})=-\frac{1}{2} \frac{L^{\prime}(\chi_{\rho},0)}{L(\chi_{\rho},0)}-\frac{1}{4}\log 3+\frac{1}{2}\log 2,
\end{displaymath}
and Lemma \ref{lemma:selberg-1}.
\end{proof}
\begin{lemma}\label{lemma:selberg-3}
The arithmetic self-intersection number of $\omega_{\PP^{1}_{\Int}/\Int}(D)_{\hyp}$, for $D=\sigma_{\infty}+(1/2)\sigma_{i}+(2/3)\sigma_{\rho}$, is given by
\begin{displaymath}
	(\omega_{\PP^{1}_{\Int}/\Int}(D)_{\hyp},\omega_{\PP^{1}_{\Int}/\Int}(D)_{\hyp})=-\frac{1}{3}\left(\frac{\zeta'(-1)}{\zeta(-1)}+\frac{1}{2}\right)+\frac{1}{3}\log(2\pi)+\frac{1}{6}\log 2.
\end{displaymath}
\end{lemma}
\begin{proof}
The result is a reformulation of the computation of Bost and K\"uhn \cite{Kuhn}. With the notations therein, $(\omega_{\PP^{1}_{\Int}/\Int}(D))^{\otimes 6}=\mathcal{M}_{12}(\Gamma(1))=\OO(1)$. Up to a constant, the hyperbolic metric and the Petersson metric on $\mathcal{M}_{12}(\Gamma(1))$ used in \emph{loc.~cit.}~agree:
\begin{displaymath}
	\|\cdot\|_{\hyp}=\frac{8}{(4\pi)^{6}}\|\cdot\|_{Pet}.
\end{displaymath}
With this comparison at hand and by the definition of arithmetic intersection numbers, this yields
\begin{displaymath}
	(\omega_{\PP^{1}_{\Int}/\Int}(D)_{\hyp},\omega_{\PP^{1}_{\Int}/\Int}(D)_{\hyp})=
	\frac{1}{36}\left((\mathcal{M}_{12}(\Gamma(1))_{Pet},\mathcal{M}_{12}(\Gamma(1))_{Pet})+12\log(4\pi)-6\log 2\right).
\end{displaymath}
We conclude by inserting the value given by \cite[Thm.~6.1]{Kuhn}:
\begin{displaymath}
	(\mathcal{M}_{12}(\Gamma(1))_{Pet},
	\mathcal{M}_{12}(\Gamma(1))_{Pet})=-12\left(\frac{\zeta'(-1)}{\zeta(-1)}+\frac{1}{2}\right).
\end{displaymath}
\end{proof}
\begin{proof}[Proof of Theorem \ref{thm:selberg-zeta}]
To prove the statement, we apply the arithmetic Riemann--Roch theorem \ref{theorem:ARR} to $\PP^{1}_{\Int}\to\Spec\Int$, with sections $\sigma_{\infty}$ (multiplicity $\infty$), $\sigma_{i}$ (multiplicity 2), $\sigma_{\rho}$ (multiplicity 3), and the uniformization $\PP^{1}_{\Int}(\CC)\setminus\lbrace\infty\rbrace=\PSL_{2}(\Int)\backslash\HH$. Notice that
\begin{displaymath}
	H^{0}(\PP^{1}_{\Int},\OO_{\PP^{1}_{\Int}})=\Int,\quad H^{1}(\PP^{1}_{\Int},\OO_{\PP^{1}_{\Int}})=0.
\end{displaymath}
Therefore,
\begin{displaymath}
	12\adeg H^{\bullet}(\PP^{1}_{\Int},\OO_{\PP^{1}_{\Int}})_{Q}=-12\log\|1\|_{L^{2}}+6\log\left(C(\PSL_{2}(\Int))\cdot Z^{\prime}(\PSL_{2}(\Int),1)\right).
\end{displaymath}
The square-norm of $1$ for the $L^{2}$ metric is given by the volume:
\begin{displaymath}
	\|1\|_{L^{2}}^{2}=\int_{\PSL_{2}(\Int)\backslash\HH}\omega,
\end{displaymath}
where $\omega$ is the normalized singular K\"ahler form, given in coordinate $\tau=x+iy$ of $\HH$ by
\begin{displaymath}
	\omega=\frac{i}{2\pi}\frac{d\tau\wedge d\bar{\tau}}{\|d\tau\|_{\hyp}^{2}}=\frac{1}{2\pi}\frac{dx\wedge dy}{y^{2}}.
\end{displaymath}
Hence, we obtain
\begin{displaymath}
	\|1\|^{2}_{L^{2}}=\frac{1}{6}
\end{displaymath}
and
\begin{equation}\label{eq:comp-sel-1}
	-12\log\|1\|_{L^{2}}=6\log 2 + 6\log 3.
\end{equation}
The rest of the numerical invariants in the arithmetic Riemann--Roch formula were obtained in lemmas \ref{lemma:selberg-1}--\ref{lemma:selberg-3}. The value of $\log C(\PSL_{2}(\Int))$ is computed by substituting in \eqref{eq:ctt-Gamma} the orders $2$ and $3$ of the elliptic fixed points, $g=0$, $n=3$, and $c=1$. Here,
we invoke the Chowla--Selberg formula again, in order to write:
\begin{displaymath}
	\log\Gamma(1/3)=-\frac{1}{3}h_{F}(E_{\rho})+\frac{1}{2}\frac{\zeta'(0)}{\zeta(0)}-\frac{1}{6}\log 3+\frac{1}{6}\log 2.
\end{displaymath}
After collecting all the constants and adding up, one arrives at the claimed expression.
\end{proof}


\bibliography{biblio}{}
\bibliographystyle{amsplain}

\end{document}